\documentclass[12pt,a4paper]{amsart}
\usepackage{a4wide}
\usepackage{amsmath, amssymb, amsfonts,enumerate}
\usepackage[all]{xy}
\usepackage{amscd}
\usepackage{amsthm} 
\usepackage{epic,latexsym,amsbsy,amsmath,amscd,amssymb, mathtools, enumerate} 
\usepackage{comment}
\usepackage{mathtools}
\usepackage{url}
\input xy
\xyoption{all}
\usepackage{hyperref}
\numberwithin{equation}{section}

 \newtheorem{definition}{Definition} [section]       
 \newtheorem{remark}[definition]{Remark}       
 \newtheorem{example}[definition]{Example}

 \newtheorem{proposition}[definition]{Proposition}       
 \newtheorem{theorem}[definition]{Theorem}       
 \newtheorem{corollary}[definition]{Corollary}       
  \newtheorem{lemma}[definition]{Lemma}   
        
\newtheorem{openproblem}[definition]{Open Problem}

\renewcommand{\Im}{\mathrm{Im}} 
\newcommand{\Aut}{\mathrm{Aut}}
\newcommand{\Ker}{\mathrm{Ker}}
\newcommand{\Sym}{\mathrm{Sym}}

\renewcommand{\dim}{\mathrm{dim}}

\newcommand{\Ind}{\mathrm{Ind}}
\newcommand{\Id}{\mathrm{Id}}

\newcommand{\Res}{\mathrm{Res}}
\newcommand{\Stab}{\mathrm{Stab}}
\newcommand{\Hom}{\mathrm{Hom}}
\newcommand{\End}{\mathrm{End}}
\newcommand{\Tr}{\mathrm{Tr}}
\newcommand{\rk}{\mathrm{rk}}
\DeclareMathOperator{\GL}{GL}

\begin{document} 

\title{Homogeneous spaces of semidirect products and finite Gelfand pairs}
\author{Tullio Ceccherini-Silberstein}
\address{Dipartimento di Ingegneria, Universit\`a del Sannio, 82100 Benevento, Italy}
\address{Istituto Nazionale di Alta Matematica ``Francesco Severi'', 00185 Roma, Italy}
\email{tulliocs@sbai.uniroma1.it}

\author{Fabio Scarabotti}
\address{Dipartimento SBAI, Sapienza Universit\`a di Roma, 00161 Roma, Italy}
\address{Istituto Nazionale di Alta Matematica ``Francesco Severi'', 00185 Roma, Italy}
\email{fabio.scarabotti@sbai.uniroma1.it}

\author{Filippo Tolli}
\address{Dipartimento DIIEM, Universit\`a Roma Tre, 00146 Roma, Italy}
\address{Istituto Nazionale di Alta Matematica ``Francesco Severi'', 00185 Roma, Italy}
\email{filippo.tolli@uniroma3.it}


\begin{abstract} 
Let $K\leq H$ be two finite groups and let $C\leq A$ be two finite abelian groups, with
$H$ acting on $A$ as a group of isomorphisms admitting $C$ as a $K$-invariant subgroup. 
We study the homogeneous space $X\coloneqq\left(H\ltimes A\right)/\left(K\ltimes C\right)$ 
and determine the decomposition of the permutation representation of $H\ltimes A$ acting on $X$. 
We then characterize when this is multiplicity-free, that is, when $\left(H\ltimes A,K\ltimes C\right)$ is a Gelfand pair.
If this is the case, we explicitly calculate the corresponding spherical functions.
From our general construction and related analysis, we recover Dunkl's results on the 
$q$-analog of the nonbinary Johnson scheme.  
\end{abstract}
\maketitle 
\tableofcontents

\section{Introduction}

Finite Gelfand pairs and their spherical functions constitute a useful tool in many areas of pure and applied mathematics, such as Algebraic Combinatorics \cite{bannai1, bannai2}, the theory of Discrete Orthogonal Polynomials \cite{Du4, Stanton}, Probability and Statistics 
\cite{CST1,book, Dbook}, Number Theory \cite{Ter}, Finite Geometries and Design Theory \cite{Saxl},
and Symmetric Functions \cite{Macdonald}. Furthermore, Okounkov and Vershik \cite{OV} (see also \cite{book2}) 
used methods from the theory of finite Gelfand pairs in order to give a new approach to the representation 
theory of the symmetric groups.
Several fundamental examples of finite Gelfand pairs were constructed by means of suitable actions of linear groups over finite fields in a series of papers of Delsarte, Dunkl, and Stanton. A survey on these examples with the description of the associated spherical functions is \cite{Stanton}, while the point of view of association schemes is well exposed in the celebrated monograph 
\cite{bannai1}, now available also in the new expanded edition \cite{bannai2}.
An introduction to this subject is \cite[Chapter 8]{book}, while a further fundamental example, the finite upper half-plane, is 
exposed in Terras' monograph \cite[Chapter 19]{Ter}.
\par
Recall that if $G$ is a finite group and $K \leq G$ is a subgroup, then $(G,K)$ is a finite Gelfand pair when the permutation representation of $G$ acting on the homogeneous space $X\coloneqq G/K$ is multiplicity-free, that is, it decomposes into pairwise inequivalent irreducible sub-representations. Actually, in many cases of ``low rank'' Gelfand pairs (cf.\ \cite{bannai3, bannai4} for the definition of ``rank'' in the finite case) the following phenomenon occurs: the permutation representation of $K$ on each of its orbits in $X$ is multiplicity-free, and this was a key fact in Dunkl's derivation of addition formulas for orthogonal polynomials (see \cite{Du4}). In particular, this happens in the so-called $q$-Johnson scheme: here, we have $G=\GL(\mathbb{F}_q^n)$, the general linear group of $n \times n$ matrices with coefficients in a finite field with $q$ elements, 
$X$ is the Grassmannian variety of all $m$-dimensional subspaces of $\mathbb{F}_q^n$, $0\leq m\leq n/2$, and $K=\bigl(\GL(\mathbb{F}_q^m)\times\GL(\mathbb{F}_q^{n-m})\bigr)\ltimes\Hom(\mathbb{F}_q^{n-m},\mathbb{F}_q^m)$ is the stabilizer of an $m$-dimensional subspace. In \cite[Section 4,5]{DUNKL} Dunkl showed that the action of $K$ on each of its orbits gives rise to a multiplicity-free permutation representation, and computed the corresponding spherical functions (actually, he computed a more general set of invariant functions). 
He got this results by means of a formidable set of detailed computations involving actions on subspaces and on linear operators. 
During his visit in Rome in February 2023, Eiichi Bannai suggested us to look at these calculations: we then decided to develop a general theory in order to get a better understanding of the framework. The present paper is devoted to a presentation of the outcomes.  
\par
In Section \ref{Secprel}, we establish the basic notation. In Section \ref{Secgenconstr}, we describe a general construction that includes Dunkl's example. Here, we have two finite groups $K\leq H$, two finite abelian groups $C\leq A$ (written additively), with
$H$ acting on $A$ as a group of isomorphisms admitting $C$ as a $K$-invariant subgroup. We then study the homogeneous space 
$X\coloneqq\left(H\ltimes A\right)/\left(K\ltimes C\right)$. In Section \ref{Sectiondecperm} we obtain the decomposition of the permutation representation of $H\ltimes A$ acting on $X$ and we characterize, in a representation theoretical way, when it is multiplicity-free, that is, when $\left(H\ltimes A,K\ltimes C\right)$ is a Gelfand pair. The key tool is the Mackey-Wigner little group method, to which we have devoted an entire book \cite{book6}. Then, in Section \ref{SecIntBasis}, we introduce a suitable vector basis 
(called an intermediate basis) of the space of all $(K\ltimes C)$-invariant functions on $X$, made up of sums of characters of abelian groups. These correspond to the $q$-Krawtchouk polynomials in Dunkl's calculations. In Section \ref{Secsphfunc}, we compute the spherical functions under the assumption that 
$\left(H\ltimes A, K\ltimes C\right)$ is a Gelfand pair; one tricky point is to specify their values on the $(K\ltimes C$)-orbits. 
Section \ref{qJohstructure} is devoted to Dunkl's example and it is mainly an exposition, in the framework of the general construction presented in Section \ref{Secgenconstr}, of the homogeneous space studied by Dunkl in \cite[Section 4]{DUNKL}. 
We also adopt an approach that is both coordinate-free and characteristic-free: we use linear operators and direct sum decompositions in place of matrices and vector space bases (as in Dunkl's paper) and almost all results are formulated for vector spaces over an arbitrary field. In Section \ref{SectiondecpermqqJoh}, we apply the results in Sections \ref{Sectiondecperm}, \ref{SecIntBasis}, and 
\ref{Secsphfunc} to the homogeneous space described in Section \ref{qJohstructure} when the field is finite, and we compute the corresponding spherical representations and spherical functions, recovering Dunkl's results in a very different way. 
In particular, we show how the spherical representations can be described by means of induced representations and the Mackey-Wigner method.
In the final section, we collect some open problems related to possible further generalizations of our main construction. 
\par
Dunkl's example, that we analyze in Sections \ref{qJohstructure} and \ref{SectiondecpermqqJoh}, may be considered a $q$-analog of the nonbinary Johnson scheme. Indeed, the Johnson scheme coincides with the homogeneous space $S_n/(S_k\times S_{n-k})$, that is, 
the set of all $k$-subsets of $\{1,2,\dotsc,m\}$, while the nonbinary Johnson scheme (\cite{TAG}, \cite[Section 8.8]{book}) 
is the $S_m\wr S_n$-homogeneous space made up of all pairs $(A,f)$ where $A$ is a $k$-subset of $\{1,2,\dotsc, n\}$ and 
$f \colon A \to \{1,2,\dotsc, m\}$ is a map. Similarly, in Section \ref{qJohstructure}, we fix a direct sum decomposition 
$V\equiv \mathbb{F}_q^n=V_0\oplus V_1$, where $V_0$ is the $m$-dimensional subspace stabilized by $K$, 
and, for $0 \leq r \leq m$ and $0 \leq s \leq n-m$, the $K$-homogeneous space $X$ is made up of all triples 
$(U_0,U_1,\Psi)$, where $U_0$ is an $r$-dimensional subspace of $V_0$, 
$U_1$ is an $s$-dimensional subspace of $V_1$, and $\Psi\in\Hom(U_1,V_0/U_0)$.
These triples are used to parameterize the subspaces $U$ of $\mathbb{F}_q^n$ such that $\dim U=r+s$ and $\dim(U\cap V_0)=r$.
Moreover, for $r$ and $s$ fixed, the set of all such subspaces constitutes an orbit of $K$.
\par
These homogeneous spaces were also studied, from the point of view of the theory of association schemes, in several
papers: Wang, Guo, and Li in \cite{WGL} proved that these are symmetric association schemes
(this is essentially our Corollary \ref{lastCor}) and in \cite{WJL} they determined the parameters of this association scheme in the case $r=0$; Kurihara
in \cite{Kurihara}, again in the case $r=0$, following a method developed by Tarnanen, Aaltonen and Goethals in \cite{TAG} (which in turn depends on Dunkl's calculations for the nonbinary Johnson scheme in \cite{Du1}), computes the character table of the association scheme.  
\par
In our previous paper \cite[Section 6]{CST2}, we developed a far-reaching generalization of the nonbinary Johnson scheme, that includes, as particular cases, several previously studied examples of finite Gelfand pairs. Similarly, the results in Sections \ref{Secgenconstr}, 
\ref{Sectiondecperm}, \ref{SecIntBasis}, and \ref{Secsphfunc} of the present paper may be considered a further development of those in 
\cite[Section 5]{CST2}. The particular case $K=H$ and $C=\{0_A\}$ is also studied in the recent paper \cite[Section 4]{Ito}.
\par
In the present paper, the notation is often quite cumbersome but this is necessary because we have to take care, at almost every step of the proofs, of several actions or representations of groups or subgroups and of various operations on representations.

\section{Preliminaries and notation}\label{Secprel}
In this paper, all groups are \underline{finite} and all of their representations are unitary and finite dimensional.
\par
Let $G$ be a group. 
\par
We denote by $\iota_G$ the trivial representation of $G$. 
If $(\sigma,U)$ and $(\theta,W)$ are equivalent $G$-representations, we then write $\sigma\sim\theta$ and/or $U \simeq W$ 
($\sim$ for the representations, $\simeq$ for their spaces).
\par
For a finite set $X$, we denote by $L(X)$ the vector space of all complex valued functions on $X$.
Also, in order to avoid unnecessary heavy notation, given an action of $G$ on a set $X$ (e.g., a group representation), 
when the action is clear from the context, we shall often simply write $gx$ to denote the image of the point $x \in X$ 
under the group element $g \in G$.
If $G$ acts on $X$, the associated \emph{permutation representation} is the representation of $G$ on $L(X)$ defined by
setting $[gf](x)\coloneqq f(g^{-1}x)$, for all $x\in X$, $g\in G$, and $f\in L(X)$. When $X = G$, the above is the
\emph{left regular representation} of $G$.
Alternatively, fixing a point $x_0\in X$ and denoting by $K \coloneqq \{g \in G: gx_0 =x_0\}$ the $G$-stabilizer of $x_0$,
we may identify $X$ with the space $G/K$, of all left cosets of $K$ in $G$, and $L(X)$ with the vector space $L^K(G) \coloneqq
\{f \in L(G): f(gk) = f(g)$ for all $k \in K$ and $g \in G\}$ of all right $K$-invariant complex functions on $G$. 
Note that $L^K(G)$ is invariant under the left regular representation of $G$.
Explicitly, a natural equivalence $L(X)\simeq L^K(G)$ (thus, between the permutation representation and the restriction of the left regular representation to the vector subspace of all $K$-invariant functions) is given by the map
\begin{equation}\label{defftilde} 
L(X)\ni f\mapsto \overset{\triangle}{f}\in L(G/K),\qquad\text{ where }\qquad\overset{\triangle}{f}(g)\coloneqq f(gx_0), \ g \in G.
\end{equation}
\par
If $K\leq G$ is a subgroup and $(\sigma,V)$ is a (unitary) $K$-representation, then the associated
\emph{induced representation} is the $G$-representation $(\Ind_K^G\sigma, \Ind_K^G V)$ where
$\Ind_K^GV$ is the vector space made up of all functions $f\colon G\rightarrow V$ such that 
\begin{equation}\label{timessquarep2}
f(gk)=\sigma(k^{-1})f(g),
\end{equation}
for all $g\in G$ and $k\in K$, and $\Ind_K^G\sigma$ is given by
\[
\bigl[gf\bigr](h)\coloneqq f(g^{-1}h)
\] 
for all $g,h\in G$.

\begin{remark}\label{remconjindrep}{\rm
In the notation above, if we fix $g\in G$, then setting $\left[\,^g\sigma\right](s)\coloneqq \sigma(g^{-1}sg)$,
$s\in gKg^{-1}$, defines a representation $^g\sigma$ of the conjugate subgroup $gKg^{-1}$ on the same space $V$. 
Then $\Ind_{gKg^{-1}}^G(\,^g\sigma)$  is equivalent to $\Ind_K^G\sigma$. 
In order to get a concrete equivalence, define $\rho(g)\colon\Ind_K^G V\rightarrow\Ind_{gKg^{-1}}^G V$ 
by setting $\left[\rho(g)f\right](h)\coloneqq f(hg)$, for all $g,h\in G$ and $f\in\Ind_K^G V$. 
Indeed, for $h\in G$ and $s=gkg^{-1}\in gKg^{-1}$, we have
\[
[\rho(g)f](hs)=
f(g_0sg)=f(hgk)=\sigma(k^{-1})f(hg)=\left[\,^g\sigma\right](s^{-1})\left[\rho(g)f\right](h),
\]
so that $\rho(g)f$ really belongs to $\Ind_{gKg^{-1}}^G V$. It is obvious that setting
$\lambda\coloneqq\Ind_K^G\sigma$ and $\lambda'\coloneqq\Ind_{gKg^{-1}}^G\;^g\sigma$ we have $\lambda'(h)\rho(g)=\rho(g)\lambda(h)$, 
for all $h \in G$, so that $\rho(g)$ yields the desired equivalence.
}
\end{remark}

If $A$ is a (finite) abelian group then $\widehat{A}$ denotes its dual group, which is made up of all characters of $A$ 
(these are the group homomorphisms $\chi \colon A \to \mathbb{T}$ of $A$ into the circle group $\mathbb{T} = \{z \in \mathbb{C}: |z|=1\} \equiv \mathbb{R}/\mathbb{Z}$). If $\chi\in\widehat{A}$, then $\overline{\chi}$ denotes the conjugate of $\chi$. 
\par
If $X$ is a finite set then $\Sym(X)$ denotes the symmetric group on $X$, that is, the group of all permutations of $X$. 
\par
The symbols $\sqcup$ and $\bigsqcup$ will denote disjoint unions. 
\par
We will use the symbol $\cong$ to denote group isomorphisms and $\simeq$ to denote isomorphisms of vector spaces. 
\par
The symbol $\coloneqq$ will be systematically used to denote definitions, and $\equiv$ for obvious identities.
\par
If $V$ and $W$ are finite dimensional vector spaces, we denote by $\Hom(V,W)$ the vector space of all linear maps from $V$ to $W$. 
When $W = V$ we set $\End(V) \coloneqq \Hom(V,V)$. For $S\in\Hom(V,W)$ we denote by 
$\rk(S)\coloneqq\dim\Im(S)$ its rank and, for $0 \leq r \leq \dim(W)$  denote by 
$\Hom_r(V,W)\coloneqq\{S\in\Hom(V,W)\colon \rk(S)=r\}$ the subset of all linear maps of rank $r$.

\section{The homogeneous space}\label{Secgenconstr}

Let $H$ be a finite group and let $A$ be a finite abelian group (written additively). 
Suppose we are given an action $H$ on $A$ by automorphisms.
Thus we have $h(a+a')=ha+ha'$, $h(h'a)=(hh')a$, and $1_Ha=a$, for all $h,h'\in H$ and $a,a'\in A$. 
We can form the semidirect product $H\ltimes A$ by defining the product law
by setting (cf.\ \cite[Definition 1.3.12]{book3}, \cite[Definition 8.14.2]{book4}, and \cite[Definition 1.6]{book6}):
\begin{equation}\label{tensp2}
(h,a)(h',a')\coloneqq (hh',a+ha')
\end{equation}
for all $h,h'\in H$ and $a,a'\in A$. 
The inverse of an element $(h,a)$ is then given by
\begin{equation}\label{2tensp2}
(h,a)^{-1}=(h^{-1},-h^{-1}a). 
\end{equation}

Note also that the action of $H$ on $A$ becomes conjugation in $H\ltimes A$:
\begin{equation}\label{Hconjugation}
(h,0_A)(1_H,a)(h,0_A)^{-1}=(1_H,ha)
\end{equation}
for all $h \in H$ and $a \in A$.

Suppose now that $H$ acts transitively on a set $Y$. We shall then refer to $Y$ as to a $H$-homogeneous space. 
Let also $(B_y)_{y\in Y}$ be a family of pairwise isomorphic finite abelian groups (written additively) indexed by the finite set $Y$. 
For $h\in H$ and $y\in Y$, we fix a group isomorphism
\begin{equation}\label{ByBhy}
h\colon B_y \to B_{hy}.
\end{equation}
Thus, \eqref{ByBhy} is a bijection satisfying $h(b+b')=hb+hb'$ for all $b,b'\in B_y$. 
We also assume that $(hh')b=h(h'b)$ and $1_Hb=b$ for all $h,h'\in H$ and $b \in B_y$.
The final assumption is the following: for every $y\in Y$ we are given a surjective group homomorphism 
$\pi_y\colon A\rightarrow B_y$ such that
\begin{equation}\label{basequat}
h\pi_y=\pi_{hy}h \qquad \text{ for all } y\in Y \mbox{ and } h\in H.
\end{equation}
Note that the $h$ on the left (resp.\ right) denotes the map \ref{ByBhy} (resp.\ the action of $H$ on $A$).
In particular, if $y_0\in Y$ and $K \leq H$ denotes its stabilizer, then
\begin{equation}\label{basequatK}
k\pi_{y_0}=\pi_{y_0}k \qquad \text{ for all } k\in K,
\end{equation}
and $kb\in B_{y_0}$ for all $k\in K$ and $b\in B_{y_0}$.
\begin{proposition}\label{PropX}
Let $X\coloneqq \{(y,b)\colon y\in Y, b\in B_{y}\}$ and set
\begin{equation}
\label{squarp3}
(h,a)(y,b)\coloneqq (hy,\pi_{hy}(a)+hb)
\end{equation} 
for all $(h,a)\in H\ltimes A$ and $(y,b)\in X$. 
Then \eqref{squarp3} defines an action of $H\ltimes A$ on $X$. 
Moreover, if $y_0$ and $K$ are as in \eqref{basequatK} then $C\coloneqq\Ker(\pi_{y_0})$ is a $K$-invariant subgroup of $A$ and the stabilizer of $(y_0,0_{B_{y_0}})$ is $K\ltimes C$. Finally,
\begin{equation}\label{starrrp3}
\Ker(\pi_{hy_0})=hC \qquad\text{ for all } h\in H.
\end{equation}
\end{proposition}
\begin{proof}
Let $(h,a), (h',a')\in H\ltimes A$ and $(y,b) \in X$. We have
\begin{align*}
(h,a)\left[(h',a')(y,b)\right]& = (h,a)(h'y, \pi_{h'y}(a') + h'b)&(\text{by }\eqref{squarp3})\\
& = (\left(h(h'y), \pi_{h(h'y)}(a) + h(\pi_{h'y}(a') + h'b)\right)&(\text{by }\eqref{squarp3})\\
& = \left(hh'y),\pi_{hh'y}(a)+h\pi_{h'y}(a')+hh'b\right)\\
& = \left((hh')y,\pi_{(hh')y}(a+ha')+(hh')b\right)&(\text{by }\eqref{basequat})\\
& = (hh', a+ha')(y,b)&(\text{by }\eqref{squarp3})\\
& = \left[(h,a)(h',a')\right](y,b)&(\text{by }\eqref{tensp2}).\\
\end{align*}
Moreover, $(1_H,0_A)(y,b)=(y,b)$. The $K$-invariance of $C$ follows from \eqref{basequatK}. 
Also, $(h,a)(y_0,0_{B_{y_0}})\equiv(hy_0,\pi_{hy_0}(a))$ is equal to $(y_0,0_{B_{y_0}})$ if and only if $h\in K$ and $a\in C$. 
Finally, from \eqref{basequat} it follows that $\pi_{hy_0}=h\pi_{y_0}h^{-1}$, and this yields \eqref{starrrp3},
\end{proof}

In other words, $X$ corresponds to the homogeneous space $\left(H\ltimes A\right)/\left(K\ltimes C\right)$.\\

\begin{remark}\label{Reminvconst}
{\rm
The above construction can be reversed as follows.
Suppose that $H,A,Y,y_0,K$ are as above and that $C\leq A$ is a $K$-invariant subgroup of $A$.
For $y \in Y$, let $h \in H$ such that $hy_0=y$ and set $B_y\coloneqq A/(hC)$ (cf.\ \eqref{starrrp3}). 
Then define $h'\colon B_y\rightarrow B_{h'y}$ by setting $h'(a+hC)\coloneqq h'a+h'hC$, for all $h'\in H$ and $a \in A$, 
and $\pi_y\colon A\rightarrow B_y$ by setting $\pi_y(a)\coloneqq a+hC$, for all $a\in A$. 
It is then easy to check that we get a system with all the above properties.}
\end{remark}

Let $Z$ be a complete set of representatives for the orbits of $K$ on $Y$ (so that $y_0\in Z$). Let $z \in Z$. Then the set 
$\sqcup_{y \in Kz}B_y$ is $K$-invariant (with respect to \eqref{ByBhy}). Moreover, if $K_z$ denotes the stabilizer of $z$ in $K$, 
then $B_z$ is $K_z$-invariant: we denote by $\Psi_z$ a complete set of representatives for the orbits of $K_z$ on $B_z$.
It the follows that $\{(z,d)\colon z\in Z,d\in \Psi_z\}$ constitutes a set of representatives for the orbits of $K$ on $X$. 
Indeed, for all $y\in Kz$ and $b\in B_y$ there exists $k\in K$ such that $kz=y$ and therefore \eqref{squarp3} yields 
$(k^{-1},0_A)(y,b)=(z,k^{-1}b)$, so that $(y,b)$ is in the $K$-orbit of $(z,k^{-1}b)$ with $k^{-1}b\in B_z$. 
If $d \in \Psi_z$ is the representative of the $K_z$-orbit of the element $k^{-1}b\in B_z$, we can find $k_z \in K_z$
such that $k_zd = k^{-1}b$. Setting $k' \coloneqq kk_z \in K$, we then have $(k',0_A)(z,d) = (y,b)$.
\par
By virtue of \eqref{basequat}, $B_z$ is also $K_z\ltimes C$-invariant, and we may thus select a complete set 
$\Psi_z^0 \subset \Psi_z$ of representatives for the $K_z\ltimes C$-orbits on $B_z$. 
This way, we have that  
\begin{equation}\label{defGamma0}
\Gamma_0\coloneqq\left\{(z,d)\colon z\in Z, d\in \Psi^0_z\right\}
\end{equation}
is a complete set of representatives for the orbits of $K\ltimes C$ on $X$. Clearly, the action of $C$ on $B=B_{y_0}$ is trivial and therefore $\Psi_{y_0}\equiv\Psi_{y_0}^0$. Moreover, $(y_0,0_{B_{y_0}})\in \Gamma_0$. Finally, by \cite[Proposition 10.4.12]{book4}, 
\[
X\times X=\bigsqcup_{(z,d)\in\Gamma_0}\bigl\{\bigl(\left(hz,\pi_{hz}(a)+hd\right),\left(hy_0,\pi_{hy_0}(a)\right)\bigr)\colon h\in H,a\in A\bigr\}
\]
is the decomposition of $X\times X$ into the $H\ltimes A$-diagonal orbits. 

From \eqref{basequat} it follows that $\pi_z(C)$ is a $K_z$-invariant subgroup of $B_z$ and therefore we may consider the natural action of $K_z$ on $B_z/\pi_z(C)$ defined by setting $k(b+\pi_z(C))\coloneqq kb+\pi_z(C)$, for all $k\in K_z$ and $b\in B_z$. We denote by $K_{z,b}$ the stabilizer of $b+\pi_z(C)$ in $K_z$, that is, $K_{z,b}\coloneqq\{k\in K_z\colon kb\in b+\pi_z(C)\}$. 
\begin{lemma}
The cardinality of the $(K\ltimes C)$-orbit containing $(z,d)\in \Gamma_0$ is given by
\begin{equation}\label{defgammazd}
\frac{\lvert K\rvert\cdot\lvert C\rvert}{\left\lvert K_{z,d}\right\rvert\cdot \lvert hC\cap C\rvert }
\end{equation}
where $h \in H$ is such that $hy_0=z$.
\end{lemma}
\begin{proof}
From \eqref{squarp3} we deduce that, for $(z,d)\in \Gamma_0$ and $(k,c)\in K\ltimes C$, we have
\[
\begin{split}
(k,c)(z,d)=(z,d)\quad&\Longleftrightarrow\quad kz=z\; \text{ and }\;\pi_z(c)+kd=d,\\
&\Longleftrightarrow\quad k\in K_{z,d}\; \text{ and }\;\pi_z(c)=d-kd,\\
\end{split}
\]
so that the stabilizer of $(z,d)$ in $K\ltimes C$ is equal to 
\[
(K\ltimes C)_{(z,d)}\coloneqq \{(k,c)\in K_{z,d}\ltimes C\colon \pi_z(c)=d-kd\}.
\] 
Note also that if $(k,c')\in(K\ltimes C)_{(z,d)}$ and $c\in C$ then $(k,c)$ is in $(K\ltimes C)_{(z,d)}$ if and only if $c-c'\in\Ker\pi_z\cap C=hC\cap C$, with $h\in H$ such that $hy_0=z$ (cf.\ \eqref{starrrp3}). In other words, for each $k\in K_{z,d}$ there exist $\lvert C\cap hC\rvert$ elements $c\in C$ such that $(k,c)\in (K\ltimes C)_{(z,d)}$ and therefore the cardinality $\frac{\lvert K\ltimes C\rvert}{\left\lvert (K\ltimes C)_{(z,d)}\right\rvert }$ of the $(K\ltimes C)$-orbit containing $(z,d)\in \Gamma_0$ is equal to
\eqref{defgammazd}.
\end{proof}

\begin{example}
{\rm
We now examine some particular cases of our construction.
\begin{enumerate}
\item
If $K=\{1_H\}$ (so that $Y = H$), then every $C\leq A$ is $K$-invariant but, in general, it is not normal in $H\ltimes A$ (cf.\ \eqref{Hconjugation}). Moreover, $B_h=A/hC$, for all $h\in H$, and $X=\{(h,b)\colon h\in H, b\in B_h\}$.
\item
If $K=H$ (so that $Y=\{y_0\}$) and $C \leq A$ is $H$-invariant, then by \eqref{Hconjugation} it is actually normal in $H\ltimes A$. Moreover, $X\equiv B=A/C$ with the action $(h,a)b=hb+\pi(a)$, where $\pi\colon A\rightarrow B$ is the quotient homomorphism.
\item
If $C=A$ and $K$ is arbitrary, then $X\equiv Y$ as an $H\ltimes A$-homogeneous space, with $A$ acting trivially.
\item
If $C=\{0_A\}$ and $K$ is arbitrary, then $X= Y\times A$ with the action of $H \ltimes A$ given by 
$(h,a)(y,a')=(hy,a+ha')$, for $a,a'\in A$, $y\in Y$, and $h\in H$.
\end{enumerate}
}
\end{example}

\begin{remark}
{\rm
Our target is to study the permutation representation of $H\ltimes A$ on $L(X)$. 
This is quite a hard problem which includes the following important cases. 
(i) For $K=\{1_H\}$ and $C=\{0_A\}$ (so that $X = H \ltimes A$) it corresponds with the whole representation theory of $H\ltimes A$, which is given by the little group method (we shall recall it in Section \ref{Sectiondecperm}). 
(ii) For $A=C=\{0_A\}$ it corresponds to the study of the permutation representation of $H$ on $L(H/K)$. 
In the general case, we need to develop a sort of interpolation between cases (i) and (ii).}
\end{remark}

For future reference, we explicitly express the representation of $H\ltimes A$ on $L(X)$: from \eqref{2tensp2}, \eqref{squarp3}, and \eqref{basequat} it follws that, for $f\in L(X)$, $(y,b)\in X$, and $(h,a)\in H\ltimes A$, then
\begin{equation}\label{genpermrep}
[(h,a)f](y,b)= f\left(h^{-1}y,-h^{-1}\pi_y(a)+h^{-1}b\right).
\end{equation}

Recall that $\left(\left(H\ltimes A\right),\left(K\ltimes C\right)\right)$ is a symmetric Gelfand pair (cf.\ \cite[Section 4.3]{book}) 
if for all choices of $h\in H$ and $a\in A$ there exist $k_1,k_2\in K$ and $c_1,c_2\in C$ such that
\[
(h^{-1},-h^{-1}a)\equiv (h,a)^{-1}=(k_1,c_1)(h,a)(k_2,c_2)\equiv(k_1hk_2,c_1+k_1a+k_1hc_2)
\]
that is
\begin{equation}\label{symmGel}
\begin{cases}
h^{-1}=k_1hk_2\\
-h^{-1}a=k_1hc_2+k_1a+c_1.
\end{cases}
\end{equation}
Note that the first condition is satisfied if and only of $(H,K)$ is a symmetric Gelfand pair. 
\par
Alternatively (cf.\ \cite[Lemma 4.3.4]{book}), 
$(H \ltimes A,K \ltimes C)$ is a symmetric Gelfand pair if and
only if for all $(y,b), (y',b')\in Y$  on has that $((y,b), (y',b'))$ and $((y',b'), (y,b))$ belong to the same orbit under the diagonal action of $H\ltimes A$ on $X\times X$, that is, if and only if there exist $h\in H$ and $a\in A$ 
such that $(h,a)(y,b)=(y',b')$ and $(h,a)(y',b')=(y,b)$, i.e.,
\begin{equation}\label{symmGel2}
\begin{cases}
hy=y' \mbox{ and } \pi_{hy}(a)+hb=b'\\
hy'=y \mbox{ and } \pi_{hy'}(a')+hb'=b.\\
\end{cases}
\end{equation}

We now present two examples, that will be further discussed in Section \ref{Sectiondecperm}, in order to clarify some tricky points connected with the action of $H$ on the dual of $A$.

\begin{example}\label{Esempiowreath}{\rm
Let $H$ be a group acting transitively on a finite set $Z$ and that $D$ be a finite abelian group. 
Then the corresponding \emph{wreath product} is the group $D\wr H$ made up of all pairs $(h,f)$, where $h\in H$ and $f \in D^Z$, with the product law given by $(h,f)(h',f')\coloneqq (hh',f+hf')$, where $[hf'](z)\coloneqq f'(h^{-1}z)$, for all $h,h' \in H$, 
$f,f' \in D^Z$, and $z \in Z$. It is the semidirect product of $H$ with $A\coloneqq D^Z$ (cf.\ \cite[Section 2.1]{book3}). 
We then suppose that $K$ is a subgroup of $H$ whose action on $Z$ is not transitive. We then denote by 
$Z=\bigsqcup_{j=1}^\ell Z_j$ the decomposition of $Z$ into $K$-orbits (with $\ell\geq 2$).

\begin{enumerate}
\item\label{Esempiowreath0} We first take $C$ as the subgroup of $D^Z$ consisting of all $f \in D^Z$ such that 
$\sum_{z\in Z_j} f(z)=0_D,$ for all $j=1,2,\dotsc,\ell$, that is, such that the sum of the values of $f$ over each $K$-orbit is zero. 
Then $C$ is clearly $K$-invariant.
\par
In general, $(H\ltimes A,K\ltimes C)$ is not a symmetric Gelfand pair: one can easily generalize \cite[Example 4.8.3]{book}. 
For instance, take $Z_1 \coloneqq\{1\}$, $Z_2 \coloneqq\{2,3\}$, $Z \coloneqq Z_1\sqcup Z_2$, $H=\Sym(Z)$, $K=\Sym(Z_1)\times\Sym(Z_2)$, 
and $D=\mathbb{Z}_n$, the cyclic group of order $n\geq 3$. If $h \coloneqq (1 \, 2)$ (transposition) then in \eqref{symmGel} 
we must take $k_1=k_2=1_H$ and if $(d_1,d_2,d_3)\in A = D^Z$, $c_1 =(0,d',-d'), c_2=(0,d'',-d'')\in C$, 
then the second equation in \eqref{symmGel} becomes
\[
(-d_2,-d_1,-d_3)=(d''+d_1,d_2+d',-d''+d_3-d') \iff \left\{\begin{array}{l}
d''=-d_1-d_2\\
d'=-d_1-d_2\\
d'+d''=2d_3\\
\end{array}
\right.
\]
which has no solutions if $2(d_1+d_2+d_3)\neq 0$. 
\par
On the other hand, in Example \ref{Esempiowreath2} we will prove that this specific case is, anyway, a Gelfand pair.

\item\label{Esempiowreath01}
We now take $C$ as the subgroup of $D^Z$ consisting of all $f \in D^Z$ which are constant on the orbits of $K$. 
As in the previous case, $C$ is clearly $K$-invariant.
\end{enumerate}}
\end{example}

\section{A decomposition of the permutation representation}\label{Sectiondecperm}

We begin by recalling the Mackey-Wigner little group method for finite semidirect products with an abelian normal subgroup, that is, for groups of the form $H\ltimes A$ as in Section \ref{Secgenconstr}. We refer to our monographs \cite[Definition 1.3.11]{book3}, where this
powerful tool is deduced, even in a more general form, from Clifford theory; \cite[Theorem 11 6.2]{book4}, where it is deduced from Mackey theory (see also \cite[Chapter 5]{Simon} for a similar approach); and to the whole of \cite{book6}, which contains several generalizations due to Schur and Mackey among others. 

Let $H$ be a group acting by automorphisms of an abelian group $A$. The action induces an action of $H$ on the dual group $\widehat{A}$: 
if $\psi$ is a character of $A$ and $h\in H$, then the $h$-\emph{conjugate} (or translate) of $\psi$ is the character $^h\psi$
given by setting (cf.\ \eqref{Hconjugation}):
\begin{equation}
\label{defhtransl}
\!^h\psi(a)\coloneqq\psi(h^{-1}a), \qquad\text{for all }a\in A.
\end{equation}

Let $\Pi$ be a set of representatives for the orbits of $H$ on $\widehat{A}$ and, for each $\psi\in\Pi$, denote by
$H_\psi\coloneqq\{h\in H\colon \,^h\psi=\psi\}$ the $H$-stabilizer of $\psi$. Then the group $H_\psi\ltimes A$ is called the \emph{inertia group} of $\psi$ and this character has an extension $\widetilde{\psi}$ to $H_\psi\ltimes A$ given by setting 
$\widetilde{\psi}(h,a)\coloneqq\psi(a)$, for all $(h,a)\in H_\psi\ltimes A$. If $\sigma$ is an irreducible representation of $H_\psi$, we denote by $\overset{\triangle}{\sigma}$ its \emph{inflation} on $H_\psi\ltimes A$, that is, the representation given by setting
$\overset{\triangle}{\sigma}(h,a)\coloneqq \sigma(h)$, for all $(h,a)\in H_\psi\ltimes A$. 
We then have:
\begin{equation}\label{litgroupmeth}
\widehat{H\ltimes A}=\left\{\Ind_{H_\psi\ltimes A}^{H\ltimes A}\left(\widetilde{\psi}\otimes\overset{\triangle}{\sigma}\right)\colon\psi\in\Pi, \sigma\in\widehat{H_\psi}\right\}.
\end{equation}
That is, \eqref{litgroupmeth} gives a list of all irreducible representations of $H\ltimes A$ and, for distinct choices of $\psi$ and $\sigma$, we get inequivalent representations. Note that, in the present case, the tensor product $\widetilde{\psi}\otimes\overset{\triangle}{\sigma}$ is just a product of an operator-valued representation by a scalar-valued one, that is, $\widetilde{\psi}\otimes\overset{\triangle}{\sigma}\equiv\widetilde{\psi}\cdot\overset{\triangle}{\sigma}$; however, we prefer to maintain the tensor product notation in order to distinguish this significant case from other occurrences of products by scalar-valued representations.

In the sequel, we shall need to deal with the irreducible representations of groups of the form $H\ltimes A$, so that we now describe them explicitly.

\begin{proposition}\label{Propindrep}
With the above notation, let $U$ be the representation space of $\sigma$. 
For every $f\in \Ind_{H_\psi}^HU$ define $F_f\colon H\ltimes A\rightarrow U$ by setting
\begin{equation}\label{indrephatris}
F_f(h,a)\coloneqq\!^h\overline{\psi}(a)f(h),\qquad \text{for all }a\in A \mbox{ and } h\in H.
\end{equation}
Then the space $W$ of $\Ind_{H_\psi\ltimes A}^{H\ltimes A}\left(\widetilde{\psi}\otimes\overset{\triangle}{\sigma}\right)$ is made up of all such functions $F_f$. Moreover, defining $\psi^\flat(a) \in L(H)$ by setting 
\begin{equation}
\label{defpsiflat}
\left[\psi^\flat(a)\right](h)\coloneqq \psi(h^{-1}a)\equiv \!^h\psi(a),
\end{equation}
for all $h\in H$ and $a\in A$, then the pointwise product $\psi^\flat(a)f$ belongs to $\Ind_{H_\psi}^HU$, for all $f\in\Ind_{H_\psi}^HU$, 
and the action of $H\ltimes A$ is given by
\begin{equation}\label{indrephaaction}
(h,a)F_f=F_{\psi^\flat(a)\left[hf\right]}
\end{equation}
for all $h \in H$, $a \in A$, and $f \in \Ind_{H_\psi}^HU$.
\par
Finally, if $f_1,f_2\in\Ind_{H_\psi}^HU$, then
\begin{equation}\label{scalprodf1f2}
\left\langle F_{f_1},F_{f_2}\right\rangle_W=\langle f_1,f_2\rangle_{\Ind_{H_\psi}^HU},
\end{equation}
that is, the map $f\mapsto F_f$ is an isometry.
\end{proposition}
\begin{proof}
The representation $\widetilde{\psi}\otimes\overset{\triangle}{\sigma}$ may be realized on the space $U$ by setting
\begin{equation}\label{chiotsig}
\left[\left(\widetilde{\psi}\otimes\overset{\triangle}{\sigma}\right)(h',a')\right]w=\psi(a')\sigma(h')w,
\end{equation}
for all $w\in U$, $h'\in H_\psi$, and $a'\in A$. Recall that $W$ is made up of all $F\colon H\ltimes A\rightarrow U$ such that
\[
\begin{split}
F((h,a)(h',a'))& = \left[\left(\widetilde{\psi}\otimes\overset{\triangle}{\sigma}\right)\left((h',a')^{-1}\right)\right]F(h,a)\\
(\text{by }\eqref{2tensp2}\text{ and }\eqref{timessquarep2})\quad& = \left[\left(\widetilde{\psi}\otimes\overset{\triangle}{\sigma}\right)\left(h'^{-1},-h'^{-1}a'\right)\right]F(h,a),
\end{split}
\] 
equivalently (cf.\  \eqref{tensp2}, \eqref{chiotsig}, and recalling that $^{h'}\!\psi=\psi$)
\begin{equation}\label{indrepha}
F(hh',a+ha')=\overline{\psi}(a')\sigma(h')^*F(h,a),
\end{equation}
for all $h\in H, h'\in H_\psi$ and $a,a'\in A$.

Note that, by setting $h'=1_H$, $a'=h^{-1}a_0$, and $a=0_A$ the above yields
\begin{equation}\label{indrephabis}
F(h,a_0)=\overline{\psi}(h^{-1}a_0)F(h,0_A)
\end{equation}
while, for $a=a'=0_A$, one obtains
\begin{equation}\label{indrephabis2}
F(hh',0_A)=\sigma(h')^*F(h,0_A).
\end{equation}

Conversely, it is easy to check that conditions \eqref{indrephabis} and \eqref{indrephabis2} imply \eqref{indrepha}, and that, moreover,
given $F \in W$ and defining $f \colon H \to U$ by setting $f(h) \coloneqq F(h, 0_A)$, one has that $f \in \Ind_{H_\psi}^HU$
and $F = F_f$. This shows that $W\equiv\left\{F_f\colon f\in\Ind_{H_\psi}^HU\right\}$.
We now note that 
\[
\left[\psi^\flat(a)\right](hh')=\psi(h'^{-1}h^{-1}a)=\psi(h^{-1}a)=\left[\psi^\flat(a)\right](h),
\]
for all $h\in H$, $h'\in H_\psi$, and $a\in A$, and this implies that $\psi^\flat(a)f\in\Ind_{H_\psi}^HU$ for all $f\in\Ind_{H_\psi}^HU$.
The action of $(h,a)\in H\ltimes A$ on $W$ is given by
\[
\begin{split}
[(h,a)F_f](h',a') & = F_f\left((h,a)^{-1}(h',a')\right)= F_f(h^{-1}h',-h^{-1}a+h^{-1}a')\\
(\text{by }\eqref{indrephatris}\text{ and }\eqref{defpsiflat} )\quad& = \;^{h'}\overline{\psi}(a') \;^{h'}\psi(a)\left[hf\right](h')=
F_{\psi^\flat(a)\left[hf\right]}(h',a'),
\end{split}
\]
for all $f \in \Ind_{H_\psi}^HU$, $h \in H$, $h'\in H_\psi$, and $a,a'\in A$. 
Finally, if $f_1,f_2\in\Ind_{H_\psi}^HU$ then (cf.\ \cite[(11.3)]{book4} and recalling that $\lvert\psi(a'')\rvert = 1$ for $a'' \in A$)
\[
\left\langle F_{f_1},F_{f_2}\right\rangle_W=
\frac{1}{\lvert A\rvert \cdot \lvert H_\psi \rvert}\sum_{a\in A,h\in H}\lvert\psi(h^{-1}a)\rvert^2\langle f_1(h),f_2(h)\rangle_U 
=\langle f_1,f_2\rangle_{\Ind_{H_\psi}^HU}.
\]
\end{proof}

\begin{remark}\label{remconjindrep2}{\rm
In the above notation, let $h\in H$ and suppose that $\,^h\psi\neq \psi$. Then $H_{\,^h\psi}=hH_\psi h^{-1}$ and for $h_0\in H, a\in A$, we have 
$(h,0_A)^{-1}(h_0,a)(h,0_A)=(h^{-1}h_0h,h^{-1}a)$
so that, by \eqref{chiotsig}, 
\[
\!^{(h,0_A)}\left(\widetilde{\psi}\otimes\overset{\triangle}{\sigma}\right)=\widetilde{\!^h\psi}\otimes\;\overset{\triangle}{\!^h\sigma},
\]
where $\,^h\sigma$ is as in Remark \ref{remconjindrep}. Since $(h,0_A)\left(H_\psi\ltimes A\right)(h,0_A)^{-1}=H_{\,^h\psi}\ltimes A$, from Remark \ref{remconjindrep} it follows that $\Ind_{H_\psi\ltimes A}^{H\ltimes A}\left(\widetilde{\psi}\otimes\overset{\triangle}{\sigma}\right)$ and $\Ind_{H_{\,^h\psi}\ltimes A}^{H\ltimes A}\left(\widetilde{\!^h\psi}\otimes\;\overset{\triangle}{\!^h\sigma}\right)$ are equivalent representations of $H\ltimes A$.
}
\end{remark}

We now use all the notation in Section \ref{Secgenconstr} but we set $B\coloneqq B_{y_0}\equiv A/C$ and $\pi\coloneqq\pi_{y_0}$. Consider the dual group $\widehat{B}$ and if $\theta\in\widehat{B}$ define its inflation to $A$ by setting:
\begin{equation}\label{definfl}
\overset{\triangle}{\theta}(a)\coloneqq \theta(\pi(a)),\qquad\text{ for all }a\in A.
\end{equation}
Then we set 
\begin{equation}\label{HKchi}
H_\theta\coloneqq\left\{h\in H\colon\;^h\overset{\triangle}{\theta}=\overset{\triangle}{\theta}\right\}\qquad\text{and}\qquad K_\theta\coloneqq H_\theta\cap K,
\end{equation} 
and we denote by $\widetilde{\theta}$ the extension of $\overset{\triangle}{\theta}$ to $H_\theta\ltimes A$, that is,
\begin{equation}\label{defchitilde}
\widetilde{\theta}(h,a)\coloneqq \overset{\triangle}{\theta}(a)\equiv \theta\left(\pi(a)\right),
\end{equation} 
for all $(h,a)\in H_\theta\ltimes A$. Finally, we denote by $\widetilde{\widetilde{\theta}} \coloneqq \Res_{K_\theta\ltimes A}^{H_\theta\ltimes A}\widetilde{\theta}$ the extension of $\overset{\triangle}{\theta}$ to $K_\theta\ltimes A$.

Next, consider the conjugacy action of $K$ on $\widehat{B}$ given by setting $\!^k\theta(b)\coloneqq\theta(k^{-1}b)$, 
for all $\theta\in\widehat{B}$, $k\in K$, and $b\in B$. Denote by $\Theta$ a complete set of representatives for the orbits of this action and note that, by \eqref{basequatK}, the stabilizer of $\theta\in\Theta$ coincides with $K_\theta$ in \eqref{HKchi}. We also consider the orbits of $H$ on $\widehat{A}$ that contain characters inflated from $B$. First of all, from \eqref{basequatK}
it follows that, for $k\in K$, one has $^k\overset{\triangle}{\theta}=\overset{\triangle}{\left({\!^k\theta}\right)}$, and in the sequel we shall simply write $\overset{\triangle}{\!^k\theta}$, so that
 the $K$-orbit $\{ ^k\theta\colon k\in K\}$ in $\widehat{B}$ corresponds to the $K$-orbit $\left\{ ^k\overset{\triangle}{\theta}\colon k\in K\right\}$ in $\widehat{A}$ which, in turn, is contained in the $H$-orbit $\left\{ ^h\overset{\triangle}{\theta}\colon h\in H\right\}$. In general, an $H$-orbit contains more than one of these $K$-orbits: consequently, we select a subset 
\begin{equation}\label{XiTheta}
\Xi\subseteq \Theta
\end{equation}
such that $\left\{\overset{\triangle}{\xi}\colon\xi\in\Xi\right\}$ constitutes a complete set of representatives for the $H$-orbits 
in $\widehat{A}$ that contain characters inflated from $B$. Finally, for $\xi\in \Xi$ we set 
\begin{equation}\label{starp7}
[\xi]\coloneqq\Biggl\{\theta\in\Theta\colon\left\{ ^k\overset{\triangle}{\theta}\colon k\in K\right\}\subseteq\Bigl\{ ^h\overset{\triangle}{\xi}\colon h\in H\Bigr\}\Biggr\}. 
\end{equation}

We now examine a particular case of Example \ref{Esempiowreath0} (resp.\ of Example \ref{Esempiowreath01}) which easily provides an
equality (resp.\ a strict inclusion) in \eqref{XiTheta}. In some sense, these Examples are dual one to each other.

\begin{example}\label{Esempiowreath1}{\rm
Assume all the notation in Example \ref{Esempiowreath}.\ref{Esempiowreath0} with $H=\Sym(Z)$, $\ell=2$, and $K=\Sym(Z_1)\times\Sym(Z_2)$. Then the map $\pi\colon A = D^Z\rightarrow A/C = D\times D$ is given by $\pi(f)= \left(\sum_{z\in Z_1}f(z),\sum_{z\in Z_2}f(z)\right)$, for all $f\in D^Z$ (cf. \cite[Corollary 4.12]{KNAPP}). In particular, if $k\in K$ then $k\pi(f)=\pi(f)$, that is, the action of $K$ on 
$B= A/C = D^Z/\Ker(\pi) = D\times D$ (and therefore also on $\widehat{B}$) is trivial. If $\theta_1,\theta_2\in\widehat{D}$ we set $\left[\theta_1\otimes\theta_2\right](d_1,d_2)\coloneqq\theta_1(d_1)\theta_2(d_2)$, for all $d_1,d_2\in D$. Therefore, for all $f \in A = D^Z$,
\[
\overset{\triangle}{\left[\theta_1\otimes\theta_2\right]}(f)= \theta_1\!\left(\sum_{z\in Z_1}f(z)\right)\cdot \theta_2\!\left(\sum_{z\in Z_2}f(z)\right)
\]
and moreover, if $h\in H$,
\begin{equation}\label{BigEq}
\begin{split}
\,^h\overset{\triangle}{\left[\theta_1\otimes\theta_2\right]}(f)& = \overset{\triangle}{\left[\theta_1\otimes\theta_2\right]}(h^{-1}f)= \theta_1\!\left(\sum_{z\in Z_1}f(hz)\right)\cdot \theta_2\!\left(\sum_{z\in Z_2}f(hz)\right)\\
& = \theta_1\!\left(\sum_{z\in Z_1\cap hZ_1}f(z)\right)\cdot\theta_1\!\left(\sum_{z\in Z_2\cap hZ_1}f(z)\right)\cdot\\
&\qquad\qquad\qquad\cdot\theta_2\!\left(\sum_{z\in Z_1\cap hZ_2}f(z)\right)\cdot\theta_2\!\left(\sum_{z\in Z_2\cap hZ_2}f(z)\right)\!.\\
\end{split}
\end{equation}

From \eqref{BigEq} it follows that if $\lvert Z_1\rvert=\lvert Z_2\rvert$, $h Z_1=Z_2$ (so that, necessarily $h Z_2=Z_1$), 
and $\theta_1\neq \theta_2$ then 
\[
^h\overset{\triangle}{\left[\theta_1\otimes\theta_2\right]}=\overset{\triangle}{\left[\theta_2\otimes\theta_1\right]}\neq\overset{\triangle}{\left[\theta_1\otimes\theta_2\right]}.
\]
This yields an example of an inflated character whose $H$-orbit contains at least two distinct $K$-orbits of inflated characters so that, in this case, the inclusion in \eqref{XiTheta} is strict.

On the other hand, if $\lvert Z_1\rvert<\lvert Z_2\rvert$, $\theta_1\neq\theta_2$, and $h\in H\setminus K$ then $\,^h\overset{\triangle}{\left[\theta_1\otimes\theta_2\right]}$ is not an inflated character. Indeed, on the one hand, there exists $d\in D$ such that $\theta_1(d)\neq\theta_2(d)$ and, on the other hand, we necessarily have $Z_2\cap hZ_1\neq\emptyset \neq Z_2\cap hZ_2$, so that there exists 
$f\in D^Z$ such that 
\[
\sum_{z\in Z_1\cap hZ_1}f(z)=0_D,\quad \sum_{z\in Z_2\cap hZ_1}f(z)=d,\quad
\sum_{z\in Z_1\cap hZ_2}f(z)=0_D,\quad \sum_{z\in Z_2\cap hZ_2}f(z)=-d.
\]
Then $f\in C$ but from \eqref{BigEq} and our assumptions it follows that $\,^h\overset{\triangle}{\left[\theta_1\otimes\theta_2\right]}(f)=\theta_1(d)\theta_2(-d)\neq 1$.

Finally, from \eqref{BigEq} it immediately follows that $^h\overset{\triangle}{\left[\theta\otimes\theta\right]}=\overset{\triangle}{\left[\theta\otimes\theta\right]}$, for all $\theta \in \widehat{D}$ and $h\in H$.

In conclusion, if $\lvert Z_1\rvert<\lvert Z_2\rvert$ then the $H$-orbit of an inflated character contains exactly one $K$-orbit of inflated characters (more precisely, exactly one inflated character) and in this case \eqref{XiTheta} is an equality. We have also proved that, in this case, $H_{\theta_1\otimes\theta_2}=K$ if $\theta_1\neq\theta_2$ while $H_{\theta\otimes\theta}=H$.
}
\end{example}

\begin{example}\label{Esempiowreath02}{\rm
Assume the notation in Example \ref{Esempiowreath}.\ref{Esempiowreath01} with $H,K$, and $Z$ as in Example \ref{Esempiowreath1}. Then the map $\pi$ may be defined in the following way: fix $z_j\in Z_j$, set $Z_j'\coloneqq Z_j\setminus\{z_j\}$, $j=1,2$, and then, for 
$f\in D^Z$, set $[\pi(f)](z)\coloneqq f(z)-f(z_j)$ for all $z\in Z'_j, j=1,2$. Then we get a surjective homomorphism $\pi\colon D^Z\equiv A\rightarrow D^{Z'_1\sqcup Z'_2}\cong B$. A character of $D^Z$ may be represented as a tensor product  $\bigotimes_{z\in Z}\theta_z$, 
where $\theta_z\in \widehat{D}$ for all $z\in Z$, with $\left[\bigotimes_{z\in Z}\theta_z\right](f)\coloneqq \prod_{z\in Z}\theta_z(f(z))$, for all $f\in D^Z$. The action of $H\equiv\Sym(Z)$ on $\widehat{D^Z}$ is given by permuting the factors:
\[
^h\!\!\left[\bigotimes_{z\in Z}\theta_z\right](f)=\left[\bigotimes_{z\in Z}\theta_z\right](h^{-1}f)=\prod_{z\in Z}\theta_z(f(hz))=\prod_{z\in Z}\theta_{h^{-1}z}(f(z)),
\]
for all $f \in D^Z$, that is, 
\begin{equation}\label{htens}
^h\!\!\left[\bigotimes_{z\in Z}\theta_z\right]=\bigotimes_{z\in Z}\theta_{h^{-1}z}.
\end{equation}
Analogously, a character of $B$ is of the form $\left(\bigotimes_{z\in Z'_1}\theta_z\right)\bigotimes\left(\bigotimes_{z\in Z'_2}\theta_z\right)$ and its inflation to $A$ is given by:
\[
\begin{split}
\left[\overset{\displaystyle\triangle}{\left(\bigotimes_{z\in Z'_1}\theta_z\right)\bigotimes\left(\bigotimes_{z\in Z'_2}\theta_z\right)}\right]&(f)=\left[\left(\bigotimes_{z\in Z'_1}\theta_z\right)\bigotimes\left(\bigotimes_{z\in Z'_2}\theta_z\right)\right](\pi(f))\\
& = \prod_{z\in Z_1'}\theta_z(f(z))\cdot\prod_{z\in Z_1'}\overline{\theta}_z(f(z_1))\cdot\prod_{z\in Z_2'}\theta_z(f(z))\cdot\prod_{z\in Z_2'}\overline{\theta}_z(f(z_2)),
\end{split}
\]
for all $f \in D^Z$, that is,
\[
\overset{\displaystyle\triangle}{\left(\bigotimes_{z\in Z'_1}\theta_z\right)\bigotimes\left(\bigotimes_{z\in Z'_2}\theta_z\right)}=\left(\bigotimes_{z\in Z_1}\xi_z\right)\bigotimes\left(\bigotimes_{z\in Z_2}\xi_z\right)
\]
where
\[
\xi_z=\begin{cases}
\theta_z&\text{if }z\in Z'_1\sqcup Z'_2\\
\prod_{z'\in Z_1'}\overline{\theta}_{z'}&\text{if }z=z_1\\
\prod_{z'\in Z_2'}\overline{\theta}_{z'}&\text{if }z=z_2.\\
\end{cases}
\]
It follows that a character of $D^Z$ is inflated from $B$ if and only if it satisfies the conditions:
\[
\prod_{z\in Z_1}\theta_z=\mathbf{1}\quad\text{ and }\quad \prod_{z\in Z_2}\theta_z=\mathbf{1},
\]
where $\mathbf{1}$ is the trivial character of $D$. Similarly, the $h$-translate in \eqref{htens} is inflated from $B$ if and only if
\[
\prod_{z\in Z_1}\theta_{h^{-1}z}=\mathbf{1}\quad\text{ and }\quad \prod_{z\in Z_2}\theta_{h^{-1}z}=\mathbf{1}.
\]
It is then easy to construct examples of inflated characters whose $H$-orbits contain inflated characters from other $K$-orbits, so that the inclusion \eqref{XiTheta} is strict. Indeed, if the four sets $Z_i\cap hZ_j$, $i,j=1,2$, are nonempty, we can choose the $\theta_z$'s in such a way that
\begin{equation}\label{HKorbitex}
\prod_{z\in Z_i\cap hZ_j}\theta_{h^{-1}z}=\mathbf{1}\quad i,j=1,2,\quad\text{ and }\quad\,^k\!\!\left[\bigotimes_{z\in Z_1}\theta_z\right]\neq\bigotimes_{z\in Z_1}\theta_{h^{-1}z},\quad\forall k\in K.
\end{equation}
Consequently, $\bigotimes_{z\in Z}\theta_z$ and $^h\!\left[\bigotimes_{z\in Z}\theta_z\right]$ are both inflated from $B$ but belong
to distinct $K$-orbits. 
Note that in order to satisfy \eqref{HKorbitex}, it suffices to suppose that (i) $D=D_1\times D_2$, with both $D_1$ and $D_2$ nontrivial 
and each $\theta_z \in \widehat{D}$ with $z\in Z_j$ trivial on $D_j$, for $j=1,2$; (ii) $\lvert Z_i\cap hZ_j\rvert\geq 2$, for $i,j=1,2$.}
\end{example}

Consider now the following representation of $K\ltimes A$ on $L(B)$: for $(k,a)\in K\ltimes A$, $f\in L(B)$, and $b\in B$ we set
\begin{equation}\label{repBy0}
[(k,a)f](b)\coloneqq f\left(-k^{-1}\pi(a)+k^{-1}b\right).
\end{equation}
This is just \eqref{genpermrep} restricted to $K\ltimes A$ and to the invariant subspace $L(B)\equiv L\left(\{y_0\}\times B\right)$ (cf. \eqref{basequatK}). In other words, the restriction of \eqref{squarp3} defines an action of $K\ltimes A$ on $B$ and \eqref{repBy0} is just the associated permutation representation. Since the stabilizer of $(y_0,0_B)$ in $H\ltimes A$ is $K\ltimes C$, it follows that \eqref{repBy0} is equivalent to $\Ind_{K\ltimes C}^{K\ltimes A}\iota_{K\ltimes C}$. 
The following is a refinement of the results in \cite[Section 5.2]{CST2}. 

\begin{theorem}\label{theoremdec} 
With the above notation, the following holds.
\begin{enumerate}
\item\label{theoremdec2} For every $\theta\in\Theta$ let $V_\theta$ denote the subspace of $L(B)$ spanned by the $K$-translates of $\theta$. Then the representation of $K\ltimes A$ on $V_\theta$ is irreducible and, moreover, it is equivalent to 
$\Ind_{K_\theta\ltimes A}^{K\ltimes A}\widetilde{\widetilde{\overline{\theta}}}$.

\item\label{theoremdec3-bis}
The irreducible $(K\ltimes A)$-representations $V_\theta$, $\theta\in\Theta$, are pairwise inequivalent.

\item\label{theoremdec3} The decomposition of $L(B)$ into irreducible $(K\ltimes A)$-representations is given by $L(B)=\bigoplus_{\theta\in\Theta}V_\theta$.

\item\label{theoremdec3-tris}
$((K\ltimes A), (K\ltimes C))$ is a Gelfand pair.
\end{enumerate}
\end{theorem}
\begin{proof}
(1) Let $\theta\in\Theta$. First of all, note that
\begin{equation}\label{SpaceVchi}
V_\theta\equiv\left\{\sum_{k\in K}f(k)\;^k\theta\colon f\in L(K/K_\theta)\right\}.
\end{equation}
On the other hand, by Proposition \ref{Propindrep} and \eqref{defchitilde}, the space of 
\[
\Ind_{K_\theta\ltimes A}^{K\ltimes A}\widetilde{\widetilde{\overline{\theta}}}\sim\Ind_{K_\theta\ltimes A}^{K\ltimes A}\left(\widetilde{\widetilde{\overline{\theta}}}\otimes\overset{\triangle}{\iota_{K_\theta}}\right)
\]
is made up of the restriction to $K\ltimes A$ of the function $F_f$, with $f\in L(K/K_\theta)\equiv\Ind_{K_\theta}^K\iota_{K_\theta}$, that is (cf.\ \eqref{indrephatris}), 
$F_f(k,a)=\theta(\pi(k^{-1}a))f(k)$, for all $(k,a) \in  K\ltimes A$. Define a linear map $T\colon \Ind_{K_\theta\ltimes A}^{K\ltimes A}\widetilde{\widetilde{\overline{\theta}}}\rightarrow V_\theta$ 
by setting 
\begin{equation}\label{defTF}
TF_f\coloneqq \sum_{k\in K}f(k)\;^k\theta,
\end{equation}
for all $f \in L(K/K_\theta)$. Since the $^k\theta$, $k \in K/K_\theta$, are pairwise inequivalent, the map $T$ is injective and therefore
surjective (since $\dim \Ind_{K_\theta\ltimes A}^{K\ltimes A}\widetilde{\widetilde{\overline{\theta}}} = [K\ltimes A:K_\theta\ltimes A]
=[K:K_\theta] = \dim V_\theta$). Moreover, for $b\in B$, $k\in K$, and $a\in A$ we have
\begin{align*}
\left[(k,a)TF_f\right](b)& = \left[TF_f\right]\left(-k^{-1}\pi(a)+k^{-1}b\right)&(\text{by }\eqref{repBy0})\\
& = \sum_{k_1\in K}f(k_1)\;^{k_1}\theta\left(-k^{-1}\pi(a)+k^{-1}b\right)&(\text{by }\eqref{defTF})\\
& = \sum_{k_0\in K}f(k^{-1}k_0)\,\theta\!\left(-\pi(k_0^{-1}a)\right)\,^{k_0}\theta\left(b\right)& (k_0=kk_1)\\
& = \sum_{k_0\in K}\left[(k,a)F_f\right](k_0,0_A)\;^{k_0}\theta\left(b\right)&(\text{by }\eqref{indrephaaction})\\
& = \left[T(k,a)F_f\right](b),
\end{align*}
that is, $(k,a)TF_f=T(k,a)F_f$. This proves that $T$ is an intertwining operator yielding the desired equivalence between $V_\theta$ 
and the irreducible representation $\Ind_{K_\theta\ltimes A}^{K\ltimes A}\widetilde{\widetilde{\overline{\theta}}}$. 
For another proof of the irreducibility of $V_\theta$ see Remark \ref{Remeasy} below.

(2) From the little group method  (cf.\ \eqref{litgroupmeth}) it follows that different choices of $\theta\in\Theta$ give rise to inequivalent representations. It follows from (1) that the $V_\theta$, $\theta \in \Theta$, are pairwise inequivalent.

(3) The elements of $\Theta$ form a complete system of representatives for the orbits of 
$K$ on $\widehat{B}$ (see also Remark \ref{Remeasy}) and 
the characters of $B$ form a basis of $L(B)$. Since the $V_\theta$, $\theta\in\Theta$, 
are the subspaces spanned by the characters in a single $K$-orbit, we deduce that the direct sum of these subspaces is the whole of $L(B)$.

(4) The permutation representation of $K \ltimes A$ on $L(B)$ (with $B \equiv A/C = K\ltimes A/K\ltimes C)$ is multiplicity-free by
(1)-(3) and this is equivalent to $(K \ltimes A,K \ltimes C)$ being a Gelfand pair.
\end{proof}

\begin{remark}\label{Remeasy}{\rm
We briefly indicate a different (and folklore) proof of the fact that the $V_\theta$'s are irreducible. 
If $W\leq V_\theta$ is a nontrivial $(K\ltimes A)$-invariant subspace, then $\Res^{K\ltimes A}_AW$ contains a character $\psi$ of $A$, that is, there exists $f\in W, f\neq0$, such that $(1_K,a)f=\psi(a)f$, for all $a\in A$. From \eqref{repBy0} it follows that 
\[
f(\pi(a))=\bigl[(1_H,-a)f\bigr](0_{B})=\overline{\psi}(a)f(0_{B}),
\]
for all $a \in A$, so that $f(0_{B})\neq 0$ and we may assume that $f(0_{B})=1$. This implies that $f(\pi(a))=\overline{\psi}(a)$, for
all $a \in A$. Therefore, $f$ is a character of $B$ and it thus coincides with one of the $^k\theta$, $\theta\in\Theta$. 
On the other hand, $W$ must contain the whole $K$-orbit of $\theta$ so that it coincides with $V_\theta$. 
A similar argument may be used to prove that the $V_\theta$'s are pairwise inequivalent: 
indeed, different $K_\theta$-orbits yield different $\psi$'s.
}
\end{remark}

For every $\theta\in\Theta$ let 
\begin{equation}\label{Hcoverkc}
L(H_\theta/K_\theta)=\bigoplus_{j=0}^{n_\theta}m_{\theta,j}U_{\theta,j}\qquad\left(\text{equivalently, }\Ind_{K_\theta}^{H_\theta}\iota_{K_\theta}\sim\bigoplus_{j=0}^{n_\theta}m_{\theta,j}\sigma_{\theta,j}\right)
\end{equation}
denote the decomposition of the permutation representation of $H_\theta$ on $L(H_\theta/K_\theta)$ into irreducible representations; thus the positive integer $m_{\theta,j}$ is the multiplicity of the irreducible representation $\left(\sigma_{\theta,j},U_{\theta,j}\right)$. We assume that $\sigma_{\theta,0}$ is the trivial representation of $H_\theta$, so that $m_{\theta,0}=1$. 
We use the notation in \eqref{XiTheta}, \eqref{starp7}, and in Remark \ref{remconjindrep2}.

\begin{theorem}\label{Theodec}
With the above notation,
\[
\Ind_{K\ltimes C}^{H\ltimes A}\iota_{K\ltimes C}\sim \bigoplus_{\theta\in\Theta}\bigoplus_{j=0}^{n_\theta}m_{\theta,j}\Ind_{H_\theta\ltimes A}^{H\ltimes A}\left(\widetilde{\overline{\theta}}\otimes\overset{\triangle}{\sigma_{\theta,j}}\right)
\]
is a decomposition of the permutation representation of $H\ltimes A$ on $L(X)$ into irreducible representations. 
Moreover, if $\xi\in\Xi$ and $\sigma\in\widehat{H_\xi}$, then multiplicity of the irreducible representation 
$\Ind_{H_\xi\ltimes A}^{H\ltimes A}\left(\widetilde{\overline{\xi}}\otimes\overset{\triangle}{\sigma}\right)$ is equal to
\begin{equation}\label{summult}
\sum m_{\theta,j}
\end{equation}
where the sum runs over all pairs $(\theta,j)$, $\theta \in \Theta$ and $0 \leq j \leq n_\theta$, such that 
\[
\;^h\overset{\triangle}{\xi}=\overset{\triangle}{\theta}\quad(\text{so that }\theta\in[\xi])\quad\text{and}\quad\,^h\sigma\sim\sigma_{\theta,j},\quad\text{ for some }h\in H.
\]
In particular, the multiplicity of $\Ind_{H_\xi\ltimes A}^{H\ltimes A}\left(\widetilde{\overline{\xi}}\otimes\overset{\triangle}{\iota_{H_\xi}}\right)$ is equal to the cardinality of $[\xi]$.
\end{theorem}
\begin{proof}
We have the following chain of equivalences, where we repeatedly use transitivity of induction (cf.\ \cite[Proposition 1.1.10]{book3} and
\cite[11.1.5]{book4}):
\[
\begin{split}
\Ind_{K\ltimes C}^{H\ltimes A}\iota_{K\ltimes C}&\sim\Ind_{K\ltimes A}^{H\ltimes A}\Ind_{K\ltimes C}^{K\ltimes A}\iota_{K\ltimes C}\\
(\text{by Theorem } \ref{theoremdec})\quad&\sim\bigoplus_{\theta\in\Theta}\Ind_{K\ltimes A}^{H\ltimes A}\Ind_{K_\theta\ltimes A}^{K\ltimes A}\widetilde{\widetilde{\overline{\theta}}}\\
\left(\widetilde{\widetilde{\overline{\theta}}}=\Res_{K_\theta\ltimes A}^{H_\theta\ltimes A}\widetilde{\overline{\theta}}\right)\quad&\sim\bigoplus_{\theta\in\Theta}\Ind_{K_\theta\ltimes A}^{H\ltimes A}\Res_{K_\theta\ltimes A}^{H_\theta\ltimes A}\widetilde{\overline{\theta}}\\
&\sim\bigoplus_{\theta\in\Theta}\Ind_{H_\theta\ltimes A}^{H\ltimes A}\Ind_{K_\theta\ltimes A}^{H_\theta\ltimes A}\Res_{K_\theta\ltimes A}^{H_\theta\ltimes A}\widetilde{\overline{\theta}}\\
(\text{by \cite[Corollary 11.1.17]{book4}})\quad&\sim\bigoplus_{\theta\in\Theta}\Ind_{H_\theta\ltimes A}^{H\ltimes A}\left[\widetilde{\overline{\theta}}\otimes \Ind\,_{K_\theta\ltimes A}^{H_\theta\ltimes A}\left(\widetilde{\iota_A}\otimes\overset{\triangle}{\iota_{K_\theta}}\right)\right]\\
(\text{by Proposition \ref{Propindrep} with $\psi = \iota_A$})\quad&\sim\bigoplus_{\theta\in\Theta}\Ind_{H_\theta\ltimes A}^{H\ltimes A}\left(\widetilde{\overline{\theta}}\otimes \overset{\triangle}{\Ind}\,_{K_\theta}^{H_\theta}\iota_{K_\theta}\right)\\
(\text{by }\eqref{Hcoverkc})\quad&\sim\bigoplus_{\theta\in\Theta}\bigoplus_{j=0}^{n_\theta}m_{\theta,j}\Ind_{H_\theta\ltimes A}^{H\ltimes A}\left(\widetilde{\overline{\theta}}\otimes\overset{\triangle}{\sigma_{\theta,j}}\right)\!.
\end{split}
\]
Note that every summand $\Ind_{H_\theta\ltimes A}^{H\ltimes A}\left(\widetilde{\overline{\theta}}\otimes\overset{\triangle}{\sigma_{\theta,j}}\right)$ is irreducible by the little group method (cf.\ \eqref{litgroupmeth}). 

In order to compute the multiplicities \eqref{summult}, first note that if $\theta\in[\xi]$ and $\theta\neq \xi$, then there exists 
$h\in H\setminus K$ such that $^h(\overset{\triangle}{\xi})=\overset{\triangle}{\theta}$. It follows that, when using the little group method, $\overline{\theta}$ and $\overline{\xi}$ give rise to equivalent representations (see also Remark \ref{remconjindrep2}) and the sum in \eqref{summult} follows immediately. In particular, since the $h$-conjugate of the trivial representation is the trivial representation, the corresponding sum is just $\sum_{\theta\in[\xi]}1 = \vert [\xi] \vert$.
\end{proof}

\begin{corollary}\label{Cormultfree}
The permutation representation of $H\ltimes A$ on $L(X)$ is multiplicity-free if and only if both the following conditions are satisfied:
\begin{enumerate}[{\rm (i)}]
\item\label{Cormultfree1}
$\Xi=\Theta$ in \eqref{XiTheta};
\item\label{Cormultfree2}
for every $\theta\in\Theta$ the permutation representation of $H_\theta$ on $L(H_\theta/K_\theta)$ is multiplicity-free.
\end{enumerate}
Moreover, if this is the case, the decomposition of $L(X)$ into irreducible representations is:
\[
L(X)\sim \bigoplus_{\theta\in\Theta}\bigoplus_{j=0}^{n_\theta}\Ind_{H_\theta\ltimes A}^{H\ltimes A}\left(\widetilde{\overline{\theta}}\otimes\overset{\triangle}{\sigma_{\theta,j}}\right)\!,
\]
where, for different pairs $(\theta,j)$, the representations are pairwise inequivalent. Finally,
\begin{equation}\label{dimind}
\dim \Ind_{H_\theta\ltimes A}^{H\ltimes A}\left(\widetilde{\overline{\theta}}\otimes\overset{\triangle}{\sigma_{\theta,j}}\right)=\frac{\lvert H\rvert\dim\sigma_{\theta,j}}{\lvert H_\theta\rvert}.
\end{equation}
\end{corollary}
\begin{proof}
All statements immediately follow from the above theorem. Just note that \eqref{dimind} is deduced from the formula for
the dimension of an induced representation (cf.\ \cite[(1.5)]{book3} and \cite[(11.10)]{book4}).
\end{proof}

\begin{example}\label{Esempiowreath2}{\rm
In Example \ref{Esempiowreath1}, we proved that when $\lvert Z_1\rvert<\lvert Z_2\rvert$, then condition (i) in Corollary \ref{Cormultfree} is satisfied. Note that, in fact, also the second condition is satisfied: indeed $H_\theta$ is equal to either $\Sym(Z)$ or $\Sym(Z_1)\times\Sym(Z_2)$ (so that $K_\theta = H_\theta \cap H = \Sym(Z_1)\times\Sym(Z_2)$) and it is well known that 
$(\Sym(Z),\Sym(Z_1)\times\Sym(Z_2))$ is a Gelfand pair (cf.\ \cite[Chapter 6]{book}).
}
\end{example}

\section{An intermediate basis}
\label{SecIntBasis}

In order to give a concise expression for the spherical functions and for the spherical representations for the Gelfand pair in
Corollary \ref{Cormultfree}, we shall introduce a suitable basis in $L(X)$, made up of characters, and then a suitable basis for the 
vector space of all $(K\ltimes C)$-invariant functions, made up of sum of characters.
For $y\in Y$ and $\chi\in\widehat{B_y}$, define $\chi_y^\sharp \in L(X)$ by setting
\begin{equation}\label{chisharp}
\chi^\sharp(y',b)\coloneqq\begin{cases}
\chi(b)&\text{ if }y'=y\\
0&\text{ otherwise,}
\end{cases}
\end{equation}
for all $(y',b)\in X$. Clearly, the set $\left\{\chi^\sharp\colon\chi\in \widehat{B_y}, y\in Y\right\}$ is an orthogonal basis for $L(X)$.
The norm of $\chi^\sharp$ is given by

\begin{equation}
\label{e:norma-chi-sharp}
\|\chi^\sharp\|^2 = \sum_{(y',b) \in X} |\chi^\sharp(y',b)|^2 = \sum_{b \in B_y} |\chi(b)|^2 = \|\chi\|^2 = |B|.
\end{equation}

From \eqref{defftilde} and \eqref{squarp3} it follows that, for $(h,a)\in H\ltimes A$,
\begin{equation}\label{chitilde}
\overset{\triangle}{\chi^\sharp}(h,a)=\chi^\sharp((h,a)(y_0,0_A))=\chi^\sharp(hy_0,\pi_{hy_0}(a))=\begin{cases}
\chi(\pi_{hy_0}(a))&\text{if }hy_0=y\\
0&\text{otherwise.}
\end{cases}
\end{equation}
For $h\in H$, $y\in Y$, and $\chi\in\widehat{B_y}$, define $^h\chi\in\widehat{B_{hy}}$ by setting 
\begin{equation}\label{defhtransl2}
\,^h\chi(b)\coloneqq \chi(h^{-1}b),\qquad\text{ for all }b\in B_{hy}.
\end{equation}
This defines an action of $H$ on $\bigsqcup\limits_{y\in Y}\widehat{B_y}$.
We then denote by $K_y$ the stabilizer in $K$ of $y\in Y$ and by $K_\chi$ the stabilizer in $K_y$ of $\chi\in\widehat{B_y}$. 
Note that $K=K_{y_0}$ and the stabilizer $K_\theta$ in $K$ of $\theta \in \widehat{B} = \widehat{B_{y_0}}$ is nothing but $H_\theta
\cap K$ (cf.\ \eqref{HKchi}). For, on the one hand, if $k \in H_\theta \cap K$ then 
$^k\overset{\triangle}{\theta} = \overset{\triangle}{\theta}$, and, on the other hand, for $b \in B$ there exists $a \in A$ such that 
$\pi(a) = b$; it follows that 
$^k\theta(b) = ^k\theta(\pi(a)) = ^k\overset{\triangle}{\theta}(a) = \overset{\triangle}{\theta}(a) = \theta(\pi(a)) = \theta(b)$, 
showing that $H_\theta \cap K \subset K_\theta$.
The reverse inclusion follows the same lines.

\begin{remark}\label{Remchitheta0}{\rm
Given $\chi\in\widehat{B_y}$, $y\in Y$, there exists $h=h(\chi)\in H$ such that $hy_0=y$ and $\,^{h^{-1}}\chi\in\Theta$. Indeed, since the
action of $H$ on $Y$ is transitive, we can find $h_0 \in H$ such that $h_0y_0=y$. Then, $^{h_0^{-1}}\chi\in\widehat{B}$ so that there exist 
$k\in K$ and $\theta\in\Theta$ such that $^{h_0^{-1}}\chi=\,^k\theta$. We then set $h = h(\chi) \coloneqq h_0k$.
Note that $\theta = \theta(\chi)$ (the representative in $\Theta$ of the $K$-orbit in $\widehat{B}$ containing $^{h_0^{-1}}\chi$) is unique, 
while the element $k \in K$ (and, consequently, $h(\chi)$) is defined up to right multiplication by an element $k'\in K_\theta$. 
Finally, for $\theta\in \Theta$ and $y\in Y$, we denote by $\widehat{B_{y}}_{,\theta}$ the set of all characters $\chi\in\widehat{B_y}$ 
such that $^{h(\chi)}\theta=\chi$, equivalently, $\theta(\chi) = \theta$.
}
\end{remark}

We now collect some useful results.

\begin{proposition}
Let $\chi\in\widehat{B_y}$, $h_0$ and $h$ in $H$ such that $hy_0=y$, and $\theta\in \Theta$. Then
\begin{equation}\label{hh1chitheta}
\sum_{a\in A} {^{h_0}}\overset{\triangle}{\theta}(a)\,\overset{\triangle}{\overline{\chi}^{\sharp}}(h,a)=
\left\langle ^{h^{-1}h_0}\overset{\triangle}{\theta},\overset{\triangle}{\,^{h^{-1}}\chi}\right\rangle_{L(A)}\equiv\begin{cases}
\lvert A\rvert&\text{if \ } ^{h^{-1}h_0}\overset{\triangle}{\theta}=\overset{\triangle}{\,^{h^{-1}}\chi}\\
0&\text{otherwise},
\end{cases}
\end{equation}
where $\overset{\triangle}{\,^{h^{-1}}\chi}$ denotes the inflation of $\,^{h^{-1}}\chi\in\widehat{B}$ to $A$.
\end{proposition}
\begin{proof}
We have, 
\begin{align*}
\sum_{a\in A}\;^{h_0}\overset{\triangle}{\theta}(a)\,\overset{\triangle}{\overline{\chi}^{\sharp}}(h,a)& = \sum_{a\in A}\theta(\pi(h_0^{-1}a))\,\overline{\chi}\left(\pi_{hy_0}(a)\right)&(\text{by }\eqref{defhtransl}, \eqref{definfl}\text{ and }\eqref{chitilde})\\
& = \sum_{a'\in A}\theta(\pi(h_0^{-1}ha'))\,^{h^{-1}}\overline{\chi}(\pi(a'))&(\text{by }\eqref{basequat}, \eqref{defhtransl2}\text{ and }a=ha')\\
& = \left\langle^{h^{-1}h_0}\overset{\triangle}{\theta},\overset{\triangle}{\,^{h^{-1}}\chi}\right\rangle_{L(A)} &\left(\text{by }\eqref{defhtransl}, \eqref{definfl}\text{ and }\,^{h^{-1}}\chi\in\widehat{B}\right).
\end{align*}
This shows the first equality in \eqref{hh1chitheta}. The other equality follows from the orthogonality relations for characters.
\end{proof}

\begin{proposition}\label{Propex4.3}
Let $h\in H$, $a\in A$, and $\chi\in \widehat{B_y}$. Then
\begin{equation}\label{actchisharp}
(h,a)\chi^\sharp=\overline{\chi}(\pi_y(h^{-1}a))\left(\,^h\chi\right)^\sharp.
\end{equation}
In particular, $\chi^\sharp$ is $C$-invariant if and only if
\begin{equation}\label{chiCinv} 
\chi(\pi_y(c))=1 \text{\ \  for all } c\in C.
\end{equation}
Moreover, if $\chi^\sharp$ is $C$-invariant then also $\left(\,^k\chi\right)^\sharp$ is $C$-invariant, for all $k\in K$.
\end{proposition}
\begin{proof}
Let $y'\in Y$ and $b\in B_{y'}$. Then
\begin{align*}
\left[(h,a)\chi^\sharp\right](y',b)  & = \chi^\sharp(h^{-1}y',-h^{-1}\pi_{y'}(a)+h^{-1}b) && (\text{by }\eqref{genpermrep})\\
& = \begin{cases}
^h\chi(b)\overline{\chi}(\pi_y(h^{-1}a))&\text{ if }h^{-1}y'= y\\
0&\text{ otherwise } \end{cases} && (\text{by }\eqref{basequat},\eqref{chisharp} \text{ and } \eqref{defhtransl2})\\
& = \overline{\chi}(\pi_y(h^{-1}a))\left(\,^h\chi\right)^\sharp(y',b),
\end{align*}
and this proves \eqref{actchisharp}. In particular, if $c\in C$ then $(1_H,c)\chi^\sharp=\overline{\chi}(\pi_y(c))\chi^\sharp$; it follows that $\chi^\sharp$ is $C$-invariant if and only if $\chi$ satisfies \eqref{chiCinv}. 
Finally, if $\chi$ satisfies \eqref{chiCinv} and $k\in K$ then, for all $c\in C$ we have
\begin{align*}
\,^k\chi\left(\pi_{ky}(c)\right)& = \chi\left(k^{-1}\pi_{ky}(c)\right)&(\text{by }\eqref{defhtransl2})\\
& = \chi\left(\pi_{y}(k^{-1}c)\right)&(\text{by }\eqref{basequat})\\
& = 1 &(\text{as } C\text{ is }K\text{-invariant}).
\end{align*}
\end{proof}

Clearly, condition \eqref{chiCinv} is equivalent to $\Ker(\chi\circ \pi_y) \supseteq C+h_0C$, where $h_0y_0=y$ (cf.\ \eqref{starrrp3}).

We denote by $\widehat{B_y}^0$ the set of all $\chi\in \widehat{B_y}$ that satisfy \eqref{chiCinv}. Note that $\widehat{B_y}^0$ is nonempty as it contains the trivial character. By Proposition \ref{Propex4.3}, the set ${\mathcal B} \coloneqq \bigsqcup\limits_{y\in Y}\widehat{B_y}^0$ is $K$-invariant. Setting ${\mathcal B}' \coloneqq \{(y,\chi)\colon y\in Y, \chi\in \widehat{B_y}^0\}$, the map $\chi \mapsto (y,\chi)$ yields a $K$-equivariant bijection of ${\mathcal B}$ onto ${\mathcal B}'$, so that we identify them. 
Consequently, we choose a complete set $\Gamma$ of representatives for the orbits of $K$ on ${\mathcal B}$ of the form (cf.\ \eqref{defGamma0}):
\begin{equation}
\label{e:Gamma-rep}
\Gamma \coloneqq \{(z,\chi): z \in Z, \chi \in {\mathcal B}_z^0\},
\end{equation}
where ${\mathcal B}_z^0 \subset \widehat{B_z}^0$ is a complete set of representatives of the $K$-orbits in ${\mathcal B}$. 
This way, we have
\begin{equation}\label{defGamma}
\bigsqcup_{y\in Y}\widehat{B_y}^0=\bigsqcup_{(z,\chi)\in\Gamma}\,^K\chi,
\end{equation}
where $^K\chi\coloneqq\{\,^k\chi\colon k\in K\}$ is the $K$-orbit of $\chi$. 

\begin{remark}\label{Remchitheta}{\rm
By Remark \ref{Remchitheta0}, given $(z,\chi)\in\Gamma$ there exists exactly one $\theta\in\Theta$ such that $hy_0=z$ and $\,^h\theta=\chi$ for some $h\in H$: in the notation therein, we write $h=h(\chi)$. In this case, $\theta$ satisfies the condition $\Ker\left(\overset{\triangle}{\theta}\right)\supseteq h(\chi)^{-1}C+C$. Note that this condition is trivially satisfied if $h(\chi)\in H_\theta$.
 Finally, for $\theta\in \Theta, y\in Y$, we denote by $\widehat{B_{y}}_{,\theta}^0$ the set of all $\chi\in\widehat{B_y}^0$ such that $\,^{h(\chi)}\theta=\chi$. Note that $\widehat{B_{y}}_{,\theta}^0$ consists of all $\chi \in \widehat{B_{y}}_{,\theta}$ satisfying condition
\eqref{chiCinv}.
}
\end{remark}

Recall that $H_\theta$ stabilizes $\overset{\triangle}{\theta}$ but not $\theta$ (cf.\ \eqref{HKchi} and \eqref{defhtransl2}). 
In fact, the stabilizer of $\theta$, is $K_\theta = H_\theta \cap K$ (cf.\ \eqref{basequatK} and \eqref{definfl}).

\begin{lemma}\label{LemmaGammatheta}
Let $\theta\in \Theta$. Then, for $(z,\chi)\in \Gamma$ one has 
\[
Kz\cap H_\theta y_0\neq \emptyset \iff K(z,\chi)\cap H_\theta(y_0,\theta)\neq\emptyset.
\]
Moreover, if $z\in Z$ and $Kz\cap H_\theta y_0\neq \emptyset$ then there exists exactly one $\chi\in \widehat{B_z}^0$ such that $(z,\chi)\in\Gamma$ and $K(z,\chi)\cap H_\theta(y_0,\theta)\neq\emptyset$.
\end{lemma}
\begin{proof}
If $z\in Z$ and $Kz$ intersect $H_\theta y_0$, that is,  
$khy_0=z$ for some $h\in H_\theta$ and $k\in K$, then $^h\theta\in \widehat{B_{hy_0}}_{,\theta}^0$ (by definition: cf.\ Remark \ref{Remchitheta}) so that, by Proposition \ref{Propex4.3}, $^{kh}\theta\in \widehat{B_z}_{,\theta}^0$. Then by \eqref{defGamma} the 
pair $(z,\!^{kh}\theta)$ belongs to the $K$-orbit of some (unique) $(z,\chi)\in \Gamma$ (with the same $z$). 
In particular one obtains the implication $\implies$. Note that the reverse implication is obvious.
Let us now prove uniqueness of $\chi$: if 
\[
k,k'\in K,\quad h_0,h_0'\in H_\theta,\quad kh_0y_0=z=k'h_0'y_0,\quad \chi=\!^{kh_0}\theta\quad\text{ and }\quad\chi'=\!^{k'h_0'}\theta
\]
then $k_1\coloneqq k'h_0'h_0^{-1}k^{-1}\in K_z$ and $\;^{k_1}\chi=\chi'$, so that $\chi$ and $\chi'$ belong to the same $K$-orbit in \eqref{defGamma}.
\end{proof}

\begin{remark}\label{defGammatheta}{\rm
For $\theta\in\Theta$ let $\Gamma_\theta$ denote the set of all $(z,\chi)\in\Gamma$ such that the equivalent conditions in 
Lemma \ref{LemmaGammatheta} are satisfied, that is, there exist $k(\chi)\in K$ and $h_0(\chi)\in H_\theta$ such that 
\begin{equation}\label{kchih0chi}
z=k(\chi)h_0(\chi)y_0\quad\text{ and }\quad \chi=\,^{k(\chi)h_0(\chi)}\theta.
\end{equation}

Note that for $(z,\chi)\in\Gamma_\theta$ the element $h(\chi)$ in Remark \ref{Remchitheta} can be chosen of the form
$h(\chi)=k(\chi)h_0(\chi)$ with $k(\chi)\in K$ and $h_0(\chi) \in H_\theta$. Also, observe that $h_0(\chi)$ is defined 
up to left/right multiplication by elements of $K_\theta$: indeed, if 
\[
k,k'\in K,\quad h_0,h_0'\in H_\theta,\quad kh_0y_0=z=k'h_0'y_0\quad\text{ and }\quad \!^{kh_0}\theta=\chi=\!^{k'h_0'}\theta,
\]
then $k_2\coloneqq h_0^{-1}k^{-1}k'h_0'\in K_\theta$ and  $k^{-1}k'=h_0k_2h_0'^{-1}\in K\cap H_\theta=K_\theta$, so that 
$h_0'=(k'^{-1}k)h_0k_2$.  
}
\end{remark}

For $(z,\chi)\in \Gamma$, we define $\Lambda_{z,\chi} \in L(X)$ by setting
\begin{equation}\label{defPsizchi}
\Lambda_{z,\chi}\coloneqq\frac{1}{\lvert K_z\rvert \cdot \lvert C\rvert}\sum_{(k,c)\in K\ltimes C}(k,c)\chi^\sharp=\frac{1}{\lvert K_z\rvert}\sum_{k\in K}\left(\,^k\chi\right)^\sharp,
\end{equation}
where the last equality follows from \eqref{actchisharp} combined with $\chi\in\widehat{B_z}^0$; the normalization factor is chosen in such a way that $\Lambda_{z,\chi}(z,0_{B_z})=1$. 

\begin{theorem}\label{TheoremPsi}
The set $\left\{\Lambda_{z,\chi}\colon(z,\chi)\in \Gamma\right\}$ is an orthogonal basis for the vector space of all
$(K\ltimes C)$-invariant functions in $L(X)$. Moreover, $\Lambda_{z,\chi}$ vanishes outside $\bigsqcup_{y\in Kz}(y,B_y)$ and its norm is
given by
\[
\left\lVert\Lambda_{z,\chi}\right\rVert_{L(X)}=\frac{\sqrt{\lvert K_\chi\rvert\cdot\lvert B\rvert\cdot\lvert K\rvert}}{\lvert K_z\rvert}.
\]
\end{theorem}
\begin{proof}
The fact that each $\Lambda_{z,\chi}$ is $(K\ltimes C)$-invariant follows immediately from its averaging expression.
From \eqref{actchisharp} it follows that $(1_K,a)\chi^\sharp=\overline{\chi}(\pi_y(a))\chi^\sharp$, for all $a\in A$. 
As a consequence, the $\chi^\sharp$'s yield a decomposition of $L(X)$ into one-dimensional
(and therefore irreducible) representations of $A$ (and therefore of $C$). 
Let $f\in L(X)$. Suppose that $f$ is $(K\ltimes C)$-invariant. As $f$ is, in particular, $C$-invariant, it is necessarily expressed as 
a linear combination of $C$-invariant $\chi^\sharp$s (in other words, $f$ belongs to the isotypic component of the 
trivial representation of $C$). Thus, if $\mathcal{T}_\chi$ denotes a complete set of representatives for the left cosets of 
$K_\chi$ in $K$ we have
\[
f=\sum_{(z,\chi)\in \Gamma}\sum_{k\in \mathcal{T}_\chi}\alpha(z,\chi,k)\left(\,^k\chi\right)^\sharp
\]
for suitable coefficients $\alpha(z,\chi,k)\in \mathbb{C}$. Moreover,
\begin{align*}
\sum_{(k',c)\in K\ltimes C}(k',c)f& = \lvert C \rvert\sum_{(z,\chi)\in \Gamma}\sum_{k\in \mathcal{T}_\chi}\alpha(z,\chi,k)\sum_{k'\in K}\left(\,^{k'k}\chi\right)^\sharp&(\text{by }\eqref{actchisharp}\text{ and }\eqref{chiCinv} )\\
& = \lvert C \rvert\cdot \left\lvert K_\chi\right\rvert\sum_{(z,\chi)\in \Gamma}\left(\sum_{k\in \mathcal{T}_\chi}\alpha(z,\chi,k)\right)\sum_{k_1\in \mathcal{T}_\chi}\left(\,^{k_1}\chi\right)^\sharp.&(k_1=k'k)
\end{align*}
Since $\frac{1}{\lvert K\ltimes C\rvert}\sum_{(k',c)\in K\ltimes C}(k',c)$ represents the orthogonal projection onto the vector space
of all $(K\ltimes C)$-invariant vectors, then $f$ is $(K\ltimes C)$-invariant if and only if
\[
f=\frac{1}{\lvert K\ltimes C\rvert}\sum_{(k',c)\in K\ltimes C}(k',c)f,
\]
that is, if and only if
\[
\alpha(z,\chi,k)=\frac{1}{\left\lvert \mathcal{T}_\chi\right\rvert}\sum_{k'\in \mathcal{T}_\chi}\alpha(z,\chi,k')
\]
so that $\beta(z,\chi)\coloneqq\alpha(z,\chi,k)$ does not depend on $k\in \mathcal{T}_\chi$. Therefore if $f$ is $(K\ltimes C)$-invariant
we can write
\[
f=\sum_{(z,\chi)\in \Gamma}\beta(z,\chi)\sum_{k\in \mathcal{T}_\chi}\left(\,^k\chi\right)^\sharp=\sum_{(z,\chi)\in \Gamma}\frac{\lvert K_z\rvert\beta(z,\chi)}{\left\lvert K_\chi\right\rvert}\Lambda_{z,\chi} \qquad(\text{by }\eqref{defPsizchi})
\]
and this proves that $\Lambda_{z,\chi}$'s span the $(K\ltimes C)$-invariant functions in $L(X)$. It follows from its definition 
\eqref{defPsizchi} and the expressions in \eqref{chisharp} and \eqref{defhtransl2} that $\Lambda_{z,\chi}$ vanishes outside 
$\bigsqcup_{y\in Kz}(y,B_y)$. Therefore, if $(z,\chi), (z',\chi')\in \Gamma$ and $z\neq z'$ then $\Lambda_{z,\chi}$ and 
$\Lambda_{z',\chi'}$, having disjoint supports, are orthogonal. 
If $z=z'$ but $\chi\neq\chi'$ then $\Lambda_{z,\chi}$ and $\Lambda_{z,\chi'}$ are orthogonal by \eqref{defGamma}
and the orthogonal relations of characters.

Finally, we compute the norm:
\[
\begin{split}
\left\lVert\Lambda_{z,\chi}\right\rVert_{L(X)}^2& = \sum_{(y,b)\in X}\lvert\Lambda_{z,\chi}(y,b)\rvert^2\\
& = \sum_{y\in Kz}\;\sum_{b\in B_y}\frac{1}{\lvert K_z\rvert^2}\sum_{k',k\in K}\left( ^{k'}\chi\right)^\sharp(y,b)\cdot\left(\,^k\overline{\chi}\right)^\sharp(y,b)\\
(\text{by }\eqref{chisharp}\text{ and }\eqref{defhtransl2})\quad& = \sum_{y\in Kz}\;\sum_{b\in B_y}\frac{1}{\lvert K_z\rvert^2}\sum_{\substack{k',k\in K\colon\\ k'z=kz=y}} \ ^{k'}\chi(b)\cdot\!^k\overline{\chi}(b)\\
(b'=k^{-1}b)\quad& = \sum_{y\in Kz}\frac{1}{\lvert K_z\rvert^2}\sum_{\substack{k',k\in K\colon\\ k'z=kz=y}}\sum_{b'\in B_z}\,^{k'}\chi(kb')\cdot\!^k\overline{\chi}(kb')\\
& = \frac{1}{\lvert K_z\rvert^2}\sum_{\substack{k',k\in K\colon\\ k^{-1}k'z=z}}\left\langle\,^{k^{-1}k'}\chi,\,\chi\right\rangle_{L(B_z)}\\
& = \frac{\lvert B\rvert}{\lvert K_z\rvert^2}\left\lvert\left\{(k',k)\in K\times K\colon k^{-1}k'\in K_\chi\right\}\right\rvert\\
& = \frac{\lvert K\rvert\cdot\lvert K_\chi\rvert\cdot\lvert B\rvert}{\lvert K_z\rvert^2}.
\end{split}
\]
\end{proof}

Evaluating the elements in \eqref{defPsizchi} on the representatives \eqref{defGamma0} we obtain (cf.\ \eqref{chisharp})
\begin{equation}\label{calcPsizchi}
\Lambda_{z,\chi}(z',d)=\delta_{z,z'}\frac{1}{\lvert K_z\rvert}\sum_{k\in K}\left(\,^k\chi\right)^\sharp(z,d)=\delta_{z,z'}\frac{1}{\lvert K_z\rvert}\sum_{k\in K_z}\chi(k^{-1}d)
\end{equation}
for all $(z,\chi)\in \Gamma$ and $(z',d)\in \Gamma_0$.

\section{The spherical functions}\label{Secsphfunc}

In the present section, we assume that the conditions in Corollary \ref{Cormultfree} are satisfied: we then use the notations therein. 
Let $\theta \in \Theta$. For $j=0,1,\dotsc, n_\theta$, let $u_{\theta,j} \in U_{\theta,j}$  be a unitary $K_\theta$-invariant vector in the representation space of $\sigma_{\theta_j}$, so that the spherical function associated with this irreducible $H_\theta$-representation is given by
\begin{equation}\label{squarep14}
\varphi_{\theta,j}(h)=\langle u_{\theta,j},\sigma_{\theta,j}(h)u_{\theta,j}\rangle, \qquad h\in H_\theta
\end{equation}
see \cite[Corollary 4.6.4]{book}. Moreover, the operator $T_{\theta,j}$ that interwines $U_{\theta,j}$ with $L(H_\theta/K_\theta)$ is given by (cf. again \cite[Corollary 4.6.4]{book})
\begin{equation}\label{squarep14bis}
\left[T_{\theta,j}u\right](h)=\sqrt{\frac{\lvert K_\theta\rvert\dim\sigma_{\theta,j}}{\lvert H_\theta\lvert}}\langle u,\sigma_{\theta,j}(h)u_{\theta,j}\rangle,\qquad u\in U_{\theta,j},\,h\in H_\theta.
\end{equation}

In the following remark, we estabilish some elementary group theoretical counting arguments.

\begin{remark}\label{Remhk1k2}{\rm
\begin{enumerate}
\item\label{Remhk1k21}
Let $k_1,k_2\in K$ and $h_0\in H$. Then there exists $h\in H$ such that $k_1\in hH_\theta$ and $k_2\in h_0^{-1}hH_\theta$ if and only if $k_2^{-1}h_0^{-1}k_1\in H_\theta$, and if this is the case then the number of these $h$s is equal to $\lvert H_\theta\rvert$.
Proof of the ``only if'' part: $k_1=hh_1, k_2=h_0^{-1}hh_2$ with $h\in H,h_1,h_2\in H_\theta$ $\Longrightarrow$ $k_2^{-1}h_0^{-1}k_1=h_2^{-1}h^{-1}h_0\cdot h_0^{-1}\cdot hh_1=h_2^{-1}h_1\in H_\theta$. Proof of the ``if'' part: take $h\in k_1H_\theta$ so that $k_1\in hH_\theta$ and $k_2\in h_0^{-1}k_1H_\theta=h_0^{-1}hH_\theta$.

\item\label{Remhk1k22}
Given $h_0\in H$, the number of pairs $(k_1,k_2)\in K \times K$ such that $k_2^{-1}h_0^{-1}k_1\in H_\theta$ is equal to $\lvert(Kh_0K)\cap H_\theta\rvert\cdot \lvert K\cap h_0 Kh_0^{-1}\rvert$. Indeed, if $h'\in (Kh_0K)\cap H_\theta$ then 
there are $\lvert K\cap h_0 Kh_0^{-1}\rvert$ pairs $(k_1,k_2)\in K \times K$ such that $k_2^{-1}h_0^{-1}k_1=h'$ (cf. \cite[Remark 11.3.2]{book4}).
\end{enumerate}}
\end{remark}

Let $\mathcal{T}_\theta$ (resp.\ $\mathcal{S}_\theta$) denote a complete set of representatives for the left cosets of $H_\theta$ in $H$ 
(resp.\ of $K_\theta$ in $H_\theta$). Thus, $H=\sqcup_{t\in\mathcal{T}_\theta}tH_\theta$ (resp.\ 
$H_\theta=\sqcup_{s\in\mathcal{S}_\theta}sK_\theta$). We also have $H=\sqcup_{t\in\mathcal{T}_\theta}\sqcup_{s\in\mathcal{S}_\theta}tsK_\theta$ and therefore, for all $y\in Y$ and $\chi\in\widehat{B_y}$ there exist unique $t(\chi)\in\mathcal{T}_\theta$, $s(\chi)\in\mathcal{S}_\theta$, and $k(\chi)\in K_\theta$ such that
\begin{equation}\label{hchith1}
h(\chi)=t(\chi)s(\chi)k(\chi),
\end{equation}
where $h(\chi)$ is as in Remark \ref{Remchitheta0}. 

\begin{theorem}\label{Teospherfunct}
\begin{enumerate}[{\rm (1)}]
\item
For $\theta\in\Theta$, $0\leq j\leq n_\theta$, $h\in H$, and $a\in A$ set
\begin{equation}\label{defF1}
F_{\theta,j}(h,a)\coloneqq \frac{^h\overset{\triangle}{\theta}(a)}{\sqrt{\lvert K\rvert \cdot \lvert K_\theta\rvert}}\sum_{h_0\in (h^{-1}K)\cap H_\theta}\sigma_{\theta,j}(h_0^{-1})^*u_{\theta,j}.
\end{equation}
Then, $F_{\theta,j}$ is a unitary $(K\ltimes C)$-invariant vector in $W_{\theta,j}\coloneqq\Ind_{H_\theta\ltimes A}^{H\ltimes A}\left(\widetilde{\overline{\theta}}\otimes\overset{\triangle}{\sigma_{\theta,j}}\right)$. 

\item The spherical function associated with $F_{\theta,j}$ has the expression

\begin{equation}\label{sphfunct}
\Phi_{\theta,j}=\frac{\lvert K_\theta\rvert}{\lvert K\rvert}\sum_{(z,\chi)\in\Gamma_\theta}\frac{\lvert K_z\rvert}{\lvert K_\chi\rvert}\varphi_{\theta,j}\left(h_0(\chi)\right)\Lambda_{z,\chi},
\end{equation}
where $\Gamma_\theta$ and $h_0(\chi)$ are as in Remark \ref{defGammatheta}, $\Lambda_{z,\chi}$ is as in \eqref{defPsizchi}, 
and $\varphi_{\theta,j}$ is as in \eqref{squarep14}.

\item\label{Teospherfunct3} The intertwining operator $S_{\theta,j}\colon W_{\theta,j}\rightarrow L(X)$ associated with $F_{\theta,j}$ is given by
\begin{equation}\label{defStheta}
S_{\theta,j}F_f=\frac{1}{\sqrt{\lvert B\rvert}}\sum_{y\in Y}\sum_{\chi\in\widehat{B_y}_{,\theta}}\left[T_{\theta,j}f(t(\chi))\right](s(\chi))\,\chi^\sharp, \ \ \mbox{ for all } f\in \Ind_{H_\theta}^HU_{\theta,j},
\end{equation}
where $F_f$ is as in \eqref{indrephatris}, $\widehat{B_y}_{,\theta}$ and $h(\chi)$ are as in Remark \ref{Remchitheta0}, 
$t(\chi)$ and $s(\chi)$ are as in \eqref{hchith1}, and $\chi^\sharp$ is as in \eqref{chisharp}.

\end{enumerate}
\end{theorem}
\begin{proof}
(1) First of all, we note that $F_{\theta,j}\in W_{\theta,j}$. Indeed, by \eqref{indrephatris} it suffices to observe that the function 
$f_{\theta,j} \colon H \to U_{\theta, j}$ defined by setting
\begin{equation}\label{deff1}
\begin{split}
f_{\theta,j}(h)\coloneqq&\frac{1}{\sqrt{\lvert K\rvert \cdot \lvert K_\theta\rvert}}\sum_{h_0\in (h^{-1}K)\cap H_\theta}\sigma_{\theta,j}(h_0^{-1})^*u_{\theta,j}\\
(k=hh_0)\quad\equiv&\frac{1}{\sqrt{\lvert K\rvert \cdot \lvert K_\theta\rvert}}\sum_{k\in K\cap (hH_\theta)}\sigma_{\theta,j}(k^{-1}h)^*u_{\theta,j},
\end{split}
\end{equation}
for all $h \in H$, belongs to the space of $\Ind_{H_\theta}^H\sigma_{\theta,j}$ and that $F_{\theta,j}=F_{f_{\theta,j}}$. 
Note also that, from the condition $K\cap (hH_\theta)\neq\emptyset$, it follows that $f_{\theta,j}$ vanishes outside $KH_\theta$. 
Then, by \eqref{scalprodf1f2} and \cite[(11.3)]{book4},
\[
\begin{split}
\lVert F_{\theta,j}\rVert^2=\lVert f_{\theta,j}\rVert^2\qquad\qquad\;&\\
(\text{by }\eqref{deff1})\quad=\frac{1}{\lvert K\rvert\cdot\lvert K_\theta\rvert\cdot\lvert H_\theta\rvert}\sum_{h\in H}\;&\sum_{h_0,h_1\in (h^{-1}K)\cap H_\theta }
\left\langle \sigma_{\theta,j}(h_0^{-1})^*u_{\theta,j},\sigma_{\theta,j}(h_1^{-1})^*u_{\theta,j}\right\rangle\\
=\frac{1}{\lvert K\rvert\cdot\lvert K_\theta\rvert\cdot\lvert H_\theta\rvert}\sum_{h\in H}\;&\sum_{h_0,h_1\in (h^{-1}K)\cap H_\theta }\left\langle u_{\theta,j},\sigma_{\theta,j}(h_0^{-1}h_1)u_{\theta,j}\right\rangle\\
(h_0=h^{-1}k_0,h_1=h^{-1}k_1)\quad& = \frac{1}{\lvert K\rvert\cdot\lvert K_\theta\rvert\cdot\lvert H_\theta\rvert}\sum_{h\in H}\sum_{k_1,k_2\in K\cap (hH_\theta)}\langle u_{\theta,j},\sigma_{\theta,j}(k_0^{-1}k_1)u_{\theta,j}\rangle\\
(\text{by Remark }\ref{Remhk1k2}.\ref{Remhk1k21}, h_0=1)\quad& = \frac{1}{\lvert K\rvert\cdot\lvert K_\theta\rvert}\sum_{\substack{k_1,k_2\in K\colon\\ k_1^{-1}k_2\in K_\theta}}\langle u_{\theta,j},\sigma_{\theta,j}(k_1^{-1}k_2)u_{\theta,j}\rangle\\
(\text{by Remark }\ref{Remhk1k2}.\ref{Remhk1k22}, h_0=1)\quad& = 1.
\end{split}
\]

Finally, suppose that $(k_0,c)\in K\ltimes C$ and $(h,a)\in H\ltimes A$. From \eqref{deff1} and $\left(h^{-1}k_0K\right)\cap H_\theta=\left(h^{-1}K\right)\cap H_\theta$ it follows that $k_0f_{\theta,j}=f_{\theta,j}$ and, in the notation of Proposition \ref{Propindrep}, 
for $h=kh'$ with $k\in K$ and $h'\in H_\theta$, we have 
\begin{equation}
\label{e:pasqua}
\left[\left(\overset{\triangle}{\overline{\theta}}\right)^\flat(c)\right](h) = \overset{\triangle}{\overline{\theta}}(h'^{-1}k^{-1}c) =
\overset{\triangle}{\overline{\theta}}(k^{-1}c)=1.
\end{equation}
Therefore, from \eqref{indrephaaction} we deduce that
\[
(k_0,c)F_{\theta,j}\equiv(k_0,c)F_{f_{\theta,j}}=F_{\left(\overset{\triangle}{\overline{\theta}}\right)^\flat(c)\cdot(k_0f_{\theta,j})}=F_{\theta,j},
\]
where in the last equality we used \eqref{e:pasqua} (which only holds for $h \in KH_\theta$) combined with the fact that
$f_{\theta,j}$ vanishes outside $KH_\theta$. This ends the proof that $F_{\theta,j}$ is a unitary $(K\ltimes C)$-invariant vector.
\par
(2) The spherical function $\Phi_{\theta,j}$ associated with $F_{\theta,j}$ may be computed by means of \cite[Corollary 4.6.4]{book} and \cite[(11.3)]{book4} (and taking into account \eqref{defftilde}): for all $h_0\in H$ and $a_0\in A$ we have
\begin{align*}
&\qquad\qquad\qquad\overset{\triangle}{\Phi}_{\theta,j}(h_0,a_0)=\left\langle F_{\theta,j},(h_0,a_0)F_{\theta,j}\right\rangle_{W_{\theta,j}}\\
&(\text{by } \eqref{indrephaaction} \text{ and } \eqref{scalprodf1f2})\quad=\left\langle f_{\theta,j},\left(\overset{\triangle}{\overline{\theta}}\right)^\flat(a_0)\cdot\left(h_0f_{\theta,j}\right)\right\rangle_{\Ind_{H_\theta}^H\sigma_{\theta,j}}\\
&(\text{by } \eqref{defpsiflat})\quad=\frac{1}{\lvert H_\theta\rvert}\sum_{h\in H}\theta(\pi(h^{-1}a_0))\cdot\left\langle f_{\theta,j}(h),f_{\theta,j}(h_0^{-1}h)\right\rangle_{U_{\theta,j}}\\
&(\text{by }\eqref{deff1})\;=\frac{1}{\lvert K\rvert\cdot\lvert K_\theta\rvert\cdot\lvert H_\theta\rvert}\sum_{h\in H}\sum_{\substack{k_1\in K\cap (hH_\theta)\\k_2\in K\cap (h_0^{-1}hH_\theta)}}\theta(\pi(k_2^{-1}h_0^{-1}a_0))\cdot\\
&\qquad\qquad\quad\cdot\left\langle \sigma_{\theta,j}(k_1^{-1}h)^*u_{\theta,j} ,\sigma_{\theta,j}(k_2^{-1}h_0^{-1}h)^*u_{\theta,j} \right\rangle_{U_{\theta,j}}\\
&(\text{by Remark }\ref{Remhk1k2}.\ref{Remhk1k21})\;=\frac{1}{\lvert K\rvert\cdot\lvert K_\theta\rvert}\sum_{\substack{k_1,k_2\in K\colon \\k_1^{-1}h_0k_2\in H_\theta}}\;^{h_0k_2}\overset{\triangle}{\theta}(a_0)\cdot\\
&\qquad\qquad\qquad\qquad\qquad\cdot\langle u_{\theta,j} ,\sigma_{\theta,j}(k_1^{-1}h_0k_2)u_{\theta,j} \rangle_{U_{\theta,j}}\\
&(\text{by }\eqref{squarep14})\quad=\frac{1}{\lvert K\rvert\cdot\lvert K_\theta\rvert}\sum_{\substack{k_1,k_2\in K\colon\\k_1^{-1}h_0k_2\in H_\theta}}
\;^{h_0k_2}\overset{\triangle}{\theta}(a_0)\varphi_{\theta,j}(k_1^{-1}h_0k_2).
\end{align*}
We have obtained an intermediate expression of the spherical functions as bi-$(K\ltimes C)$-invariant functions defined on $(H\ltimes A)$. We now use the basis in Theorem \ref{TheoremPsi} to get an expression of these spherical functions as elements in $L(X)$. 
We have (cf.\ \eqref{defftilde} and \eqref{defPsizchi}):
\[
\begin{split}
\left\langle\overset{\triangle}{\Phi}_{\theta,j},\overset{\triangle}{\Lambda}_{z,\chi}\right\rangle_{L(H\ltimes A)}& = \frac{1}{\lvert K\rvert\cdot\lvert K_\theta\rvert\cdot\lvert K_z\rvert}\sum_{h\in H, a\in A}\sum_{\substack{k_1,k_2\in K\colon\\k_1^{-1}hk_2\in H_\theta}}\sum_{k\in K}
\;^{hk_2}\overset{\triangle}{\theta}(a)\cdot\\
&\qquad\qquad\qquad\qquad\cdot\varphi_{\theta,j}(k_1^{-1}hk_2)\overset{\displaystyle\triangle}{\left(\,^k\overline{\chi}\right)^\sharp}(h,a)\\
(\text{by }\eqref{chitilde})\quad& = \frac{1}{\lvert K\rvert\cdot\lvert K_\theta\rvert\cdot\lvert K_z\rvert}\sum_{k\in K}\sum_{\substack{h\in H\colon\\ hy_0=kz}}\sum_{\substack{k_1,k_2\in K\colon\\k_1^{-1}hk_2\in H_\theta}}\varphi_{\theta,j}(k_1^{-1}hk_2)\cdot\\
&\qquad\qquad\qquad\qquad\cdot\sum_{a\in A}\;^{hk_2}\overset{\triangle}{\theta}(a)\overset{\displaystyle\triangle}{\left(\,^k\overline{\chi}\right)^\sharp}(h,a) \\
(\text{by }\eqref{hh1chitheta}, h=h_0)\quad& = \frac{\lvert A\rvert}{\lvert K\rvert\cdot\lvert K_\theta\rvert\cdot\lvert K_z\rvert}\sum_{k\in K}\sum_{\substack{h\in H\colon\\ hy_0=kz}}\sum_{\substack{k_2\in K\colon\\ \;^{k_2}\theta=\,^{h^{-1}k}\chi}}\;\sum_{\substack{k_1\in K\colon\\k_1^{-1}hk_2\in H_\theta}}\varphi_{\theta,j}(k_1^{-1}hk_2) \\
\end{split}
\]
and this is zero unless $\chi\in\widehat{B_z}^0_{,\theta}$ (cf. Remark \ref{Remchitheta}). If this is the case, then
\begin{itemize}
\item all the $h$'s in the second sum are of the form $h=kh(\chi)k_3$, with $k_3\in K$ and $h(\chi)$ as in Remark \ref{Remchitheta0} and Remark \ref{Remchitheta};
\item in the third sum $k_2=k_3^{-1}k'$, with $k'\in K_\theta$;
\item in the fourth sum we have $k_1^{-1}hk_2=k_1^{-1}kh(\chi)k'\in H_\theta$ if and only if $k_1\in kh(\chi)H_\theta$. 
\end{itemize}
\noindent
Therefore,
\[
\begin{split}
\left\langle\overset{\triangle}{\Phi}_{\theta,j},\overset{\triangle}{\Lambda}_{z,\chi}\right\rangle_{L(H\ltimes A)}& = \frac{\lvert A\rvert}{\lvert K_z\rvert}\sum_{k\in K}\;\sum_{\substack{k_1\in \left(kh(\chi)H_\theta\right)\cap K}}\varphi_{\theta,j}(k_1^{-1}kh(\chi))\\
(k_0\coloneqq k^{-1}k_1)\quad& = \frac{\lvert A\rvert\cdot\lvert K\rvert}{\lvert K_z\rvert}\sum_{\substack{k_0\in \left(h(\chi)H_\theta\right)\cap K}}\varphi_{\theta,j}(k_0^{-1}h(\chi)). \\
\end{split}
\]
But $\left(h(\chi)H_\theta\right)\cap K\neq\emptyset$ if and only if $h(\chi)=kh_1$ with $k\in K$ and $h_1\in H_\theta$ so that, by Remark \ref{defGammatheta}, necessarily $(z,\chi)\in \Gamma_\theta$ and $h(\chi)=k(\chi)h_0(\chi)$, as in \eqref{kchih0chi}. It follows that 
\[
\left(h(\chi)H_\theta\right)\cap K=\left(k(\chi)H_\theta\right)\cap K=k(\chi)K_\theta
\]
and setting $k' \coloneqq k(\chi)^{-1}k_0$, we have $k'\in K_\theta$ and $k_0^{-1}h(\chi)=k'^{-1}h_0(\chi)$, so that
\[
\begin{split}
\left\langle\overset{\triangle}{\Phi}_{\theta,j},\overset{\triangle}{\Lambda}_{z,\chi}\right\rangle_{L(H\ltimes A)}&=\frac{\lvert A\rvert\cdot\lvert K\rvert}{\lvert K_z\rvert}\sum_{\substack{k'\in K_\theta}}\varphi_{\theta,j}(k'^{-1}h_0(\chi)). \\
&=\frac{\lvert A\rvert\cdot\lvert K\rvert\cdot\lvert K_\theta\rvert}{\lvert K_z\rvert}\varphi_{\theta,j}\left(h_0(\chi)\right).
\end{split}
\]
This gives us
\[
\left\langle\Phi_{\theta,j},\Lambda_{z,\chi}\right\rangle_{L(X)}=\frac{\lvert B\rvert\cdot\lvert K_\theta\rvert}{\lvert K_z\rvert}\varphi_{\theta,j}\left(h_0(\chi)\right).
\]

Finally, from the orthogonality relations in Theorem \ref{TheoremPsi} we deduce that
\[
\begin{split}
\Phi_{\theta,j}& = \sum_{(z,\chi)\in\Gamma_\theta}\frac{\lvert K_z\rvert^2}{\lvert K\rvert\cdot\lvert K_\chi\rvert\cdot\lvert B\rvert}\left\langle\Phi_{\theta,j},\Lambda_{z,\chi}\right\rangle_{L(X)}\Lambda_{z,\chi}\\
& = \sum_{(z,\chi)\in\Gamma_\theta}\frac{\lvert K_z\rvert\cdot\lvert K_\theta\rvert}{\lvert K\rvert\cdot\lvert K_\chi\rvert}\varphi_{\theta,j}\left(h_0(\chi)\right)\Lambda_{z,\chi}.
\end{split}
\]

(3) From \cite[Corollary 4.6.4]{book}, \cite[(11.3)]{book4}, \eqref{defftilde}, and \eqref{dimind} we get, for all $h_0\in H, a_0\in A$, 
and $F_f$ as in \eqref{indrephatris} (with $f\in \Ind_{H_\theta}^HU_{\theta,j}$),
\[
\begin{split}
\left[\overset{\displaystyle\triangle}{S_{\theta,j}F_f}\right](h_0,a_0)
& = \sqrt{\frac{\dim\sigma_{\theta,j}\lvert K\rvert\cdot\lvert C\rvert}{\lvert A\rvert\cdot\lvert H_\theta\rvert}}\langle (h_0,a_0)^{-1}F_f,F_{\theta,j}\rangle_{W_{\theta,j}}\\
(\text{by }\eqref{indrephaaction} \text{ and }\eqref{scalprodf1f2})\quad& = \sqrt{\frac{\dim\sigma_{\theta,j}\lvert K\rvert\cdot\lvert C\rvert}{\lvert A\rvert\cdot\lvert H_\theta\rvert}}
\left\langle \left(\overset{\triangle}{\theta}\right)^\flat(h_0^{-1}a_0)\cdot(h_0^{-1}f),f_{\theta,j}\right\rangle_{{\rm Ind}_{H_\theta}^H\sigma_{\theta,j}}\\
(\text{by }\eqref{defpsiflat}\text{ and }\eqref{deff1})\quad& = \frac{\sqrt{\dim\sigma_{\theta,j}\lvert C\rvert}}{\lvert H_\theta\rvert^{3/2}\sqrt{\lvert A\rvert\cdot\lvert K_\theta\rvert}}\sum_{h\in H}\sum_{k\in K\cap\left(hH_\theta\right)}\;^{h_0h}\overset{\triangle}{\theta}(a_0)\cdot\\
&\qquad\cdot\left\langle f(h_0h),\sigma_{\theta,j}(k^{-1}h)^*u_{\theta,j}\right\rangle_{U_{\theta,j}}\\
(\text{by }\eqref{timessquarep2}\text{ and }h_1\coloneqq h^{-1}k)\quad& = \frac{\sqrt{\dim\sigma_{\theta,j}\lvert C\rvert}}{\lvert H_\theta\rvert^{3/2}\sqrt{\lvert A\rvert\cdot\lvert K_\theta\rvert}}\sum_{h_1\in H_\theta,k\in K}
\;^{h_0kh_1^{-1}}\overset{\triangle}{\theta}(a_0)\langle f(h_0k),u_{\theta,j}\rangle_{U_{\theta,j}}\\
& = \frac{\sqrt{\dim\sigma_{\theta,j}\lvert C\rvert}}{\sqrt{\lvert H_\theta\rvert\cdot\lvert A\rvert\cdot\lvert K_\theta\rvert}}\sum_{k\in K}
\;^{h_0k}\overset{\triangle}{\theta}(a_0)\langle f(h_0k),u_{\theta,j}\rangle_{U_{\theta,j}}.\\
\end{split}
\]
Then for all $\chi\in\widehat{B_y}, y\in Y$, we have
\[
\begin{split}
\left\langle \overset{\displaystyle \triangle}{S_{\theta,j}F_f},\overset{\triangle}{\chi^\sharp}\right\rangle_{L(H\ltimes A)}& = \frac{\sqrt{\dim\sigma_{\theta,j}\lvert C\rvert}}{\sqrt{\lvert H_\theta\rvert\cdot\lvert A\rvert\cdot\lvert K_\theta\rvert}}\sum_{a\in A, h\in H}\sum_{k\in K}
\;^{hk}\overset{\triangle}{\theta}(a)\langle f(hk),u_{\theta,j}\rangle_{U_{\theta,j}}\overset{\triangle}{\overline{\chi}^\sharp}(a,h)\\
(\text{by }\eqref{chitilde})\quad& = \frac{\sqrt{\dim\sigma_{\theta,j}\lvert C\rvert}}{\sqrt{\lvert H_\theta\rvert\cdot\lvert A\rvert\cdot\lvert K_\theta\rvert}}\sum_{\substack{ h\in H\colon\\ hy_0=y}}\;\sum_{k\in K}
\langle f(hk),u_{\theta,j}\rangle_{U_{\theta,j}}\sum_{a\in A}\;^{hk}\overset{\triangle}{\theta}(a)\overset{\triangle}{\overline{\chi}^\sharp}(a,h)\\
(\text{by }\eqref{hh1chitheta})\quad& = \frac{\sqrt{\dim\sigma_{\theta,j}\lvert C\rvert\cdot\lvert A\rvert}}{\sqrt{\lvert H_\theta\rvert\cdot\lvert K_\theta\rvert}}\sum_{\substack{ h\in H\colon\\ hy_0=y}}\sum_{\substack{k\in K\colon \\ \;^k\theta=\,^{h^{-1}}\chi \\ }}\langle f(hk),u_{\theta,j}\rangle_{U_{\theta,j}}\\
\end{split}
\]
and, as in the the proof of (2), this is zero unless $\chi\in\widehat{B_{y}}_{,\theta}$ (now, cf.\ Remark \ref{Remchitheta0}).
If this is the case, then all the $h$'s in the first sum are of the form $h=h(\chi)k_0$, with $k_0\in K$ and $h(\chi)$ as in Remark \ref{Remchitheta0}, while in the second sum $k=k_0^{-1}k'$, with $k'\in K_\theta$, so that
\[
\begin{split}
\left\langle \overset{\displaystyle\triangle}{S_{\theta,j}F_f},\overset{\triangle}{\chi^\sharp}\right\rangle_{L(H\ltimes A)}& = \frac{\sqrt{\dim\sigma_{\theta,j}\lvert C\rvert\cdot\lvert A\rvert}}{\sqrt{\lvert H_\theta\rvert\cdot\lvert K_\theta\rvert}}\sum_{k_0\in K}\sum_{k'\in K_\theta}\langle f(h(\chi)k'),u_{\theta,j}\rangle_{U_{\theta,j}}\\
& = \lvert K\rvert\frac{\sqrt{\dim\sigma_{\theta,j}\lvert C\rvert\cdot\lvert A\rvert\cdot\lvert K_\theta\rvert}}{\sqrt{\lvert H_\theta\rvert}}\langle f(h(\chi)),u_{\theta,j}\rangle_{U_{\theta,j}}\\
\end{split}
\]
where the last equality follows from \eqref{timessquarep2} and the $K_\theta$-invariance of $u_{\theta,j}$. 
Finally,
\[
\left\langle S_{\theta,j}F_f,\chi^\sharp\right\rangle_{L(X)}=\frac{\sqrt{\dim\sigma_{\theta,j}\lvert A\rvert\cdot\lvert K_\theta\rvert}}{\sqrt{\lvert H_\theta\rvert\cdot\lvert C\rvert}}\langle f(h(\chi)),u_{\theta,j}\rangle_{U_{\theta,j}}
\]
and therefore, recalling that $\|\chi^\sharp\|^2 = |B|$ (cf.\ \ref{e:norma-chi-sharp}), 
\[
\begin{split}
S_{\theta,j}F_f& = \sum_{y\in Y}\sum_{\chi\in\widehat{B_y}}\frac{1}{\lvert B\rvert}\left\langle S_{\theta,j}F_f,\chi^\sharp\right\rangle_{L(X)}\chi^\sharp\\
& = \frac{\sqrt{\dim\sigma_{\theta,j}\lvert K_\theta\rvert}}{\sqrt{\lvert H_\theta\rvert\cdot\lvert B\rvert}}\sum_{y\in Y}\sum_{\chi\in\widehat{B_y}_{,\theta}}\langle f(h(\chi)),u_{\theta,j}\rangle_{U_{\theta,j}}\chi^\sharp\\
(\text{by }\eqref{hchith1})\quad& = \frac{\sqrt{\dim\sigma_{\theta,j}\lvert K_\theta\rvert}}{\sqrt{\lvert H_\theta\rvert\cdot\lvert B\rvert}}\sum_{y\in Y}\sum_{\chi\in\widehat{B_y}_{,\theta}}
\langle f(t(\chi)s(\chi)k(\chi)),u_{\theta,j}\rangle_{U_{\theta,j}}\chi^\sharp\\
(\text{by }\eqref{timessquarep2})\quad& = \frac{\sqrt{\dim\sigma_{\theta,j}\lvert K_\theta\rvert}}{\sqrt{\lvert H_\theta\rvert\cdot\lvert B\rvert}}\sum_{y\in Y}\sum_{\chi\in\widehat{B_y}_{,\theta}}\langle f(t(\chi)),\sigma_{\theta,j}(s(\chi))u_{\theta,j}\rangle_{U_{\theta,j}}\chi^\sharp\\
(\text{by }\eqref{squarep14bis})\quad& = \frac{1}{\sqrt{\lvert B\rvert}}\sum_{y\in Y}\sum_{\chi\in\widehat{B_y}_{,\theta}}\left[T_{\theta,j}f(t(\chi))\right](s(\chi))\,\chi^\sharp.
\end{split}
\]
\end{proof}

We now evaluate a spherical harmonic in Theorem \ref{Teospherfunct}.\ref{Teospherfunct3} on an element of $X$. 
We use \eqref{chisharp} and the decomposition of $f$ in \cite[(11.8)]{book4}, namely, $f=\sum_{t\in\mathcal{T}_\theta}tf_{u_t}$, 
where $u_t\in U_{\theta,j}$ and $f_{u_t}(h)=\sigma_{\theta,j}(h^{-1})u_t$, for $h\in H_\theta$, and 
$f_{u_t}$ vanishes in $H \setminus H_\theta$. 
The value of \eqref{defStheta} on $(y,b)\in X$ is equal to
\begin{equation}\label{descsphharm}
\begin{split}
\frac{1}{\sqrt{\lvert B\rvert}}\sum_{\chi\in\widehat{B_y}_{,\theta}}\chi(b)\left[T_{\theta,j}f(t(\chi))\right](s(\chi)) & =
\frac{1}{\sqrt{\lvert B\rvert}}\sum_{t\in\mathcal{T}_\theta}\sum_{\chi\in\widehat{B_y}_{,\theta}}\chi(b)\left[T_{\theta,j}tf_{u_t}(t(\chi))\right](s(\chi))\\
& = \frac{1}{\sqrt{\lvert B\rvert}}\sum_{t\in\mathcal{T}_\theta}\sum_{\chi\in\widehat{B_y}_{,\theta}}\chi(b)\left[T_{\theta,j}f_{u_t}(t^{-1}t(\chi))\right](s(\chi))\\
& = \frac{1}{\sqrt{\lvert B\rvert}}\sum_{t\in\mathcal{T}_\theta}\sum_{\substack{\chi\in\widehat{B_y}_{,\theta}\colon\\ t(\chi)=t}}\chi(b)\left[T_{\theta,j}f_{u_t}(1_H)\right](s(\chi))\\
& = \frac{1}{\sqrt{\lvert B\rvert}}\sum_{t\in\mathcal{T}_\theta}\sum_{\substack{\chi\in\widehat{B_y}_{,\theta}\colon\\ t(\chi)=t}}\chi(b)\left[T_{\theta,j}u_t\right](s(\chi)).
\end{split}
\end{equation}

We are now in a position to compute the spherical functions on the representatives \eqref{defGamma0}. 

\begin{corollary}
Let $\theta \in \Theta$ and $j=0,1, \ldots, n_\theta$. For $(z,d)\in \Gamma_0$, we have
\[
\Phi_{\theta,j}(z,d)=0\quad\text{ if }\quad Kz\cap H_\theta y_0=\emptyset;
\]
otherwise, if $(z,\chi)$ is the unique element in $\Gamma_\theta$ with $z$ in the first position, then
\[
\Phi_{\theta,j}(z,d)=\frac{\lvert K_\theta\rvert}{\lvert K\rvert\cdot\lvert K_\chi\rvert}\varphi_{\theta,j}\left(h_0(\chi)\right)\sum_{k\in K_z}\chi(k^{-1}d),
\]
where $h_0(\chi)\in H_\theta$ (here $\Gamma_\theta$ and $h_0(\chi)$ are as in Remark \ref{defGammatheta}).
\end{corollary}
\begin{proof}
This follows immediately from Remark \ref{defGammatheta}, and formul\ae \ \eqref{calcPsizchi}, and \eqref{sphfunct}.
\end{proof}

The orthogonality relations for the spherical functions may be easily deduced from \cite[Corollary 4.6.4]{book} and \eqref{dimind}: we have
\begin{equation}\label{normPhi}
\left\lVert\Phi_{\theta,j}\right\rVert^2_{L(X)}=\frac{\lvert X\rvert\cdot \lvert H_\theta\rvert}{\lvert H\rvert\dim\sigma_{\theta,j}}=\frac{\lvert A\rvert\cdot \lvert H_\theta\rvert}{\lvert K\rvert\cdot\lvert C\rvert\dim\sigma_{\theta,j}}
\end{equation}
so that, taking into account that the spherical functions are constant on the orbits of $(K\ltimes C)$ on $X$ (whose cardinality are
expressed in \eqref{defgammazd}), we get 
\begin{equation}\label{orthrel}
\sum_{(z,d)\in\Gamma_0}\Phi_{\theta,j}(z,d)\overline{\Phi_{\theta',j'}(z,d)}\frac{\lvert K\rvert\cdot\lvert C\rvert}{\left\lvert K_{z,d}\right\rvert\cdot \lvert hC\cap C\rvert }=\delta_{(\theta,j),(\theta',j')}\frac{\lvert A\rvert\cdot \lvert H_\theta\rvert}{\lvert K\rvert\cdot\lvert C\rvert\dim\sigma_{\theta,j}}
\end{equation}
for all $\theta, \theta'\in\Theta, 0\leq j\leq n_\theta$, and $0\leq j'\leq n_{\theta'}$.

\section{The group theoretical structure of the $q$-analog of the nonbinary Johnson scheme}
\label{qJohstructure}

In this section, we treat the linear algebraic aspects of the $q$-analog of the nonbinary Johnson scheme in \cite[pp.\ 17-32]{DUNKL}. We closely follow Dunkl's paper and extensively use his notation, but we emphasize the characteristic-free features of the approach: we work on finite dimensional vector spaces over an arbitrary field $\mathbb{F}$ (finite fields appear only in Corollary \ref{lastCor}). Moreover, we use a coordinate-free approach (which was extensively used by Dunkl): in place of matrices, we work with linear operators in order to reveal the underlying group theoretic aspects (semidirect products, groups actions, etc.); in place of bases, we use suitable direct sum decompositions. We also point out when a construction (a map, etc.) is canonical, that is, it does not depend on the choice of a basis. Finally, most of the results are described by using the group theoretical setting that we developed in Section \ref{Secgenconstr}.

Let $V$ be a finite dimensional vector space over a field $\mathbb{F}$.
We fix, once and for all, a decomposition $V=V_0\oplus V_1$ of $V$ as a direct sum of two nontrivial subspaces. 
Let  then $P\colon V\rightarrow V_1$ denote the projection onto $V_1$ along $V_0$, and set $n\coloneqq\dim V$ and $m\coloneqq \dim V_0$. 
Finally, let $\mathcal{G}_n$ denote the lattice of all subspaces of $V$ and set 
\begin{equation}
\label{e:H-n}
\mathcal{H}_n\coloneqq\{(U_0,U_1,\Psi)\colon U_0\leq V_0, U_1\leq V_1,\Psi\in\Hom(U_1,V_0/U_0)\}.
\end{equation}
With each $U\in\mathcal{G}_n$ we associate the triple $(U_0,U_1,\Psi)\in\mathcal{H}_n$ defined as follows:
\begin{itemize} 
\item $U_0\coloneqq U\cap V_0$ and $U_1\coloneqq P(U)$; 
\item for $u_1\in U_1$ let $v_0 \in V_0$ such that $v_0+u_1\in U$ and set $\Psi(u_1)\coloneqq v_0+U_0\in V_0/U_0$.
\end{itemize}

\begin{proposition}
\label{PropGH}
$\Psi$ is a well defined linear map in $\Hom(U_1,V_0/U_0)$.
Moreover, the map 
\begin{equation}\label{UU0U1}
U\longmapsto (U_0,U_1,\Psi)
\end{equation} 
yields a bijection between $\mathcal{G}_n$ and $\mathcal{H}_n$. 
Finally, $U/U_0\simeq U_1$, so that $\dim U=\dim U_0+\dim U_1$.
\end{proposition}
\begin{proof}
Let $u_1\in U_1$. Then, there exists $u \in U$ such that $P(u) = u_1$. On the other hand, there exist unique $v_0 \in V_0$ and
$v_1 \in V_1$ such that $u = v_0 + v_1$. By applying $P$, we have $u_1 = P(u) = P(v_0) + P(v_1) = P(v_1)$ so that $u_1 = v_1$
and therefore $v_0+u_1 = u \in U$.
Moreover, $u_0'\in V_0$ satisfies $u_0'+u_1 \in U$ if and only if $u_0'-u_0\in U \cap V_0 = U_0$. 
As a consequence, $\{u_0\in V_0\colon u_0+u_1\in U\}= u_0+U_0$ is an element of $V_0/U_0$, and $\Psi$ is well defined.
It follows that 
\begin{equation}\label{Uu0u1}
U=\{u_0+u_1\colon u_1\in U_1\text{ and }u_0\in\Psi(u_1)\}.
\end{equation}
Conversely, any triple $(U_0,U_1,\Psi)\in\mathcal{H}_n$ is the image of the unique $U\in\mathcal{G}_n$ defined by means of \eqref{Uu0u1}.
We end the proof by noting that the kernel of the projection $P\vert_U\colon U\rightarrow U_1$ is just $U_0$ so that
$U/\Ker P\vert_U = U/U_0 \cong {\rm Im}(P\vert_U) = P(U) = U_1$.
\end{proof}

Set
\begin{equation}
\label{H et A}
H\equiv H_{m,n-m}\coloneqq \GL(V_0)\times \GL(V_1) \mbox{ \ and \ } A\equiv A_{n-m,m}\coloneqq \Hom(V_1,V_0)
\end{equation}
and define an action of $H$ on $A$ by setting
\begin{equation}
\label{(star)p28}
(g_0,g_1)T\coloneqq g_0Tg_1^{-1},\quad \text{ for all }g_0\in \GL(V_0),\, g_1\in \GL(V_1),\text{ and } T\in \Hom(V_1,V_0).
\end{equation}

Consider the semidirect product $H\ltimes A$ as in Section \ref{Secgenconstr}. Fix $0\leq r\leq m$, $0\leq s\leq n-m$, and set
\begin{equation}
\label{e:Y-Yrs}
Y\equiv Y_{r,s}\coloneqq \{(U_0,U_1)\colon U_0\leq V_0, U_1\leq V_1, \dim U_0=r,\dim U_1=s\}
\end{equation}
which is naturally an $H$-homogeneous space: $(g_0,g_1)(U_0,U_1)\coloneqq (g_0U_0,g_1U_1)$. 
For each $y\equiv (U_0,U_1)\in Y$ define the vector space 
\begin{equation}
\label{e:B-y}
B_y\coloneqq \Hom(U_1,V_0/U_0)
\end{equation}
 and a family of isomorphisms as in \eqref{ByBhy} by setting
\begin{equation}\label{2starp28}
\begin{array}{ccc}
(g_0,g_1)\colon \Hom(U_1,V_0/U_0)&\longrightarrow&\Hom(g_1U_1,V_0/g_0U_0)\\
\Psi&\longmapsto&\left(g_0|_{V_0/U_0}\right)\Psi \left(g_1^{-1}|_{g_1U_1}\right)
\end{array}
\end{equation}
where $g_1^{-1}|_{g_1U_1}$ is the restriction of $g_1^{-1}$ to $g_1U_1$ while $g_0|_{V_0/U_0}$ is the linear map $V_0/U_0\rightarrow V_0/g_0U_0$ defined by setting $g_0|_{V_0/U_0}(v_0+U_0)\coloneqq g_0v_0+g_0U_0$. For the sake of simplicity, we will often write $g_0\Psi g_1^{-1}$ to denote $\left(g_0|_{V_0/U_0}\right)\Psi \left(g_1^{-1}|_{g_1U_1}\right)$.
For $y\equiv (U_0,U_1)\in Y$ the projection $\pi_y\colon A\rightarrow B_y$ is given by
\begin{equation}\label{defpiynbqJ}
\begin{array}{ccc}
\pi_y\colon\Hom(V_1,V_0)&\longrightarrow&\Hom(U_1,V_0/U_0)\\
T&\longmapsto&\pi_{U_0}T\iota_{U_1}
\end{array}
\end{equation}
where $\pi_{U_0}\colon V_0\rightarrow V_0/U_0$ is the canonical projection and $\iota_{U_1}\colon U_1\rightarrow V_1$ is the inclusion map.
We now check the fundamental relation \eqref{basequat} in the present setting:
\[
h\pi_y T=(g_0,g_1)\bigl(\pi_{U_0}T\iota_{U_1}\bigr)=g_0|_{V_0/U_0}\bigl(\pi_{U_0}T\iota_{U_1}\bigr)g_1^{-1}|_{g_1U_1}
=\pi_{g_0U_0}\bigl(g_0Tg_1^{-1}\bigr)\iota_{g_1U_1}
=\pi_{hy} hT,
\]
for all $h=(g_0,g_1)\in H$, $y=(U_0,U_1)\in Y$, and $T\in A$. Then the homogeneous space $X$ as in Proposition \ref{PropX} is made up of all triples $(U_0,U_1,\Psi)$ with $y=(U_0,U_1)\in Y$ and $\Psi\in B_y$. The action in \eqref{squarp3} becomes, for all
$(g_0,g_1,T) \in H \ltimes A$ and $(U_0,U_1,\Psi) \in X$,
\begin{equation}\label{squarp3bis}
(g_0,g_1,T)(U_0,U_1,\Psi)=\left(g_0U_0,g_1U_1,\pi_{g_0U_0}T\iota_{g_1U_1}+g_0\Psi g_1^{-1}
\right).
\end{equation}

The group $H\ltimes A$ also acts on the vector space $V=V_0\oplus V_1$ in a natural way: for $u=u_0+u_1\in V$,  $u_0\in V_0, u_1\in V_1$, and $(g_0,g_1,T)\in H\ltimes A$, then 
\begin{equation}\label{squarp3tris}
(g_0,g_1,T)(u_0+u_1)=\overbrace{(g_0u_0+Tg_1u_1)}^{\in V_0}\;+\;\overbrace{g_1u_1}^{\in V_1}.
\end{equation}
This way, $H\ltimes A$ corresponds to a subgroup of $\GL(V)$ and it is precisely the stabilizer of the subspace $V_0$ (see \cite[Sections 8.6 and 8.7]{book} and \cite[p.15]{DUNKL}). The homogeneous space $\GL(V)/(H\ltimes A)$ is then $\mathcal{G}_{n,m}$, the Grassmannian of all 
$m$-dimensional vector subspaces of $V$. The orbits of $H\ltimes A$ on $\mathcal{G}_{n,m}$ are 
\begin{equation}\label{defGmk}
\mathcal{G}_{n,m,k}\coloneqq\{U\in\mathcal{G}_{n,m}\colon \dim(U\cap V_0)=k\},\qquad \max\{0,2m-n\}\leq k\leq m.
\end{equation}

\begin{proposition}\label{PropHn}
The correspondence \eqref{UU0U1} yields a conjugacy between the actions of 
$H\ltimes A$ on $\mathcal{G}_n$ by means of \eqref{squarp3tris} and on $\mathcal{H}_n$ by means of \eqref{squarp3bis}.
\end{proposition}
\begin{proof}
Let $U\leq V$. From \eqref{Uu0u1} and \eqref{squarp3tris} it follows that a vector $v=v_0+v_1\in V$, with $v_0\in V_0$ and $v_1\in V_1$, belongs to $(g_0,g_1,T)U$ if and only if there exist $u_1\in U_1$ and $u_0\in\Psi(u_1)$ such that $v_0=g_0u_0+Tg_1u_1$ and $v_1=g_1u_1$. But
\[
\begin{split}
u_0\in\Psi(u_1) & \iff g_0u_0\in g_0\Psi(u_1)=\left(g_0\Psi g_1^{-1}\right)(g_1u_1)\\
& \iff g_0u_0+ Tg_1u_1\in \left\{\left(g_0\Psi g_1^{-1}\right)(g_1u_1)+\left[\pi_{g_0U_0}T\iota_{g_1U_1}\right](g_1u_1)\right\} \\
& \iff g_0u_0+ Tg_1u_1\in\left(\pi_{g_0U_0}T\iota_{g_1U_1}+g_0\Psi g_1^{-1}\right)(g_1u_1).
\end{split}
\]
In other words, with respect to \eqref{UU0U1}, $U$ corresponds to $(U_0,U_1,\Psi)$ if and only if $(g_0,g_1,T)U$ corresponds to 
$(g_0,g_1,T)(U_0,U_1,\Psi)$.
\end{proof}

If we fix $y_0=(W_0,W_1)\in Y_{r,s}$, then its stabilizer $K_{(W_0,W_1)}$ in $H$ is isomorphic to the direct product of two groups of the same structure of $H\ltimes A$. Namely, if $V_0=W_0\oplus Z_0$ and $V_1=W_1\oplus Z_1$, then
\begin{equation}\label{defW1W0}
\begin{split}
&K_{(W_0,W_1)}\cong \left(H_{r,m-r}\ltimes A_{m-r,r}\right)\times\left( H_{s,n-m-s}\ltimes A_{n-m-s,s}\right)\\
&\quad=\bigl\{\bigl((g_{00},g_{01},T_0),(g_{10},g_{11},T_1)\bigr)\colon g_{00}\in \GL(W_0),g_{01}\in \GL(Z_0), \\
&\qquad\quad T_0\in \Hom(Z_0,W_0), g_{10}\in \GL(W_1), g_{11}\in \GL(Z_1), T_1\in \Hom(Z_1,W_1)\bigr\},
\end{split}
\end{equation}
where the action of $k_0\coloneqq(g_{00},g_{01},T_0)$ (resp. $k_1\coloneqq(g_{10},g_{11},T_1)$) on $V_0=W_0\oplus Z_0$ (resp. $V_1=W_1\oplus Z_1$) is defined as in \eqref{squarp3tris}.
Moreover,
\begin{equation}\label{Cinthiscase}
C=\Ker \pi_{y_0}=\{T\in\Hom(V_1,V_0)\colon T(W_1)\leq W_0\}\cong \Hom(W_1,W_0)\oplus \Hom(Z_1,V_0).
\end{equation}

Finally, for $\max\{2r-m,0\}\leq h\leq r$ and $\max\{2s-n+m,0\}\leq i\leq s$, select
\begin{equation}\label{hilvar2}
z_{h,i}\coloneqq (U_{0,h},U_{1,i})\in Y_{r,s}\text{ such that } \dim\left(W_0\cap U_{0,h}\right)=h\;\text{ and }\; \dim\left( W_1\cap U_{1,i}\right)=i.
\end{equation}  

Then the family of all pairs $(U_{0h},U_{1i})$ constitutes a complete set of representatives for the orbits of 
$K_{(W_0,W_1)}$ on $Y_{r,s}$. Indeed, $Y_{r,s}$ clearly coincides with $\mathcal{G}_{m,r}\times\mathcal{G}_{n-m,s}$, 
and, taking \eqref{defW1W0} into account, we apply \eqref{defGmk} to $K_{(W_0,W_1)}$.

Denote by $V'$ the dual vector space of $V$ and, for $U\leq V$, denote by $U^\circ$ the annihilator of $U$, that is, the vector subspace consisting of all $\phi\in V'$ such that $\phi(u)=0$ for all $u\in U$. For $g \in \GL(V)$ define $g^t\in\GL(V')$ by setting
\begin{equation}
\label{e:centralo}
[g^t\phi](v)\coloneqq \phi(g^{-1}v)
\end{equation}
for all $\phi\in V'$ and $v\in V$. 
Then, the map $\GL(V)\ni g\mapsto g^t\in\GL(V')$ is an isomorphism. Moreover, if denoting by $\mathcal{G}'_n$ the set of all vector subspaces of $V'$, then the map $\mathcal{G}_n\ni U\mapsto U^\circ\in\mathcal{G}_n'$ is a bijection and we have
\begin{equation}\label{ggtVprime}
\left(g U\right)^\circ=g^t \left(U^\circ\right)
\end{equation}
for all $g \in \GL(V)$ and $U \in \mathcal{G}_n$.
Indeed, for $\phi\in V'$, $U \in \mathcal{G}_n$, and $g \in \GL(V)$, one has $\phi\in\left(g U\right)^\circ$ if and only if $0=\phi(gu)=[(g^{-1})^t\phi](u)$ for all $u\in U$, that is, if and only if $\phi\in g^t \left(U^\circ\right)$. In the following, for the sake of simplicity,
we identify $\GL(V)$ and $\GL(V')$ via \eqref{e:centralo}.  It follows that (in $\GL(V)$) the stabilizer of $U^\circ$ equals the stabilizer of $U$. Since $(U\cap V_0)^\circ=U^\circ+V_0^\circ$ (cf.\ \cite[Theorem 3.15]{ROMAN}), in our setting, the orbits of $H\ltimes A$ (the stabilizer of $V_0$) in $\mathcal{G}_{n,n-m}'$ (the Grassmannian formed by all $(n-m)$-dimensional vector subspaces of $V'$) are 
exactly the sets $\mathcal{G}_{n,n-m,k}'\coloneqq\{U'\in\mathcal{G}_{n,n-m}'\colon \dim(U'+ V_0^\circ)=n-k\}$, for $\max\{0,2m-n\}\leq k\leq m$. Note that $\mathcal{G}_{n,n-m,k}'$ is the image of $\mathcal{G}_{n,m,k}$ in \eqref{defGmk} under the correspondence $U\mapsto U^\circ$. Finally, recall that $U\leq W$ in $V$ implies that $W^\circ\leq U^\circ$ in $V'$. 

Consider now the natural action of $H$ on $\Hom(V_0,V_1)$ given by setting $(g_0,g_1)S\coloneqq g_1Sg_0^{-1}$, for all $(g_0,g_1)\in H$ and
$S\in \Hom(V_0,V_1)$.

Let $S\in \Hom(V_0,V_1)$ and fix decompositions $V_0=\Ker S\oplus Z_0$ and $V_1=\Im S\oplus Z_1$. 

\begin{lemma}\label{LemmaHS}
The stabilizer $H_S$ of $S$ in $H$ is made up of all elements in $K_{(\Ker S,\Im S)}$ as in \eqref{defW1W0} such that 
$S|_{Z_0}=g_{10}Sg_{01}^{-1}$.
\end{lemma}
\begin{proof}
Let $(g_0,g_1)\in H$ and suppose that it stabilizes $S$. Then $g_0$ stabilizes $\Ker S$ and $g_1$ stabilizes $\Im S$, so that $(g_0,g_1)\in K_{(\Ker S,\Im S)}$. Moreover, from \eqref{2tensp2}, \eqref{(star)p28}, \eqref{squarp3tris}, and \eqref{defW1W0} (with $(W_0,W_1)$ replaced by $(\Ker S,\Im S)$). it follows that, for all $u_0\in \Ker S$ and $z_0\in Z_0$,
\begin{equation}\label{starp32}
\begin{split}
(g_{10},g_{11},T_1)S(g_{00},g_{01},T_0)^{-1}(u_0+z_0)& = (g_{10},g_{11},T_1)S(g_{00}^{-1},g_{01}^{-1},-g_{00}^{-1}T_0g_{01})(u_0+z_0)\\
& = (g_{10},g_{11},T_1)S\left[\left(g_{00}^{-1}u_0-g_{00}^{-1}T_0z_0\right)+g_{01}^{-1}z_0\right]\\
\left(g_{00}^{-1}u_0-g_{00}^{-1}T_0z_0\in\Ker S\right)\quad& = (g_{10},g_{11},T_1)Sg_{01}^{-1}z_0\\
\left(Sg_{01}^{-1}z_0\in\Im S\right)\quad& = g_{10}Sg_{01}^{-1}z_0.
\end{split}
\end{equation}
Since $S(u_0+z_0)=Sz_0$, we deduce that $(g_{10},g_{11},T_1)S(g_{00},g_{01},T_0)^{-1}=S$ if and only if $S|_{Z_0}=g_{10}Sg_{01}^{-1}$.
\end{proof}

The orbits of $H$ on $\Hom(V_0,V_1)$ are parameterized by the rank, so that we define $\Hom_\ell(V_0,V_1)$ as the set of all linear maps from $V_0$ to $V_1$ of rank $\ell = 0,1, \ldots \min\{n-m,m\}$, so that
\[
\Hom(V_0,V_1)=\bigsqcup_{\ell=0}^{\min\{n-m,m\}}\Hom_\ell(V_0,V_1)
\] 
is the decomposition into $H$-orbits. For $\ell = 0,1, \ldots \min\{n-m,m\}$, pick $S_\ell\in \Hom_\ell(V_0,V_1)$. Then 
\begin{equation}\label{XiHom}
\{S_\ell\colon \ell=0,1,\dotsc,\min\{m,n-m\}\}
\end{equation}
is a set of representatives for these orbits.

If $u_0\in \Ker S$, $z_0\in Z_0$, $u_1\in \Im S$, $z_1\in Z_1$, and $\bigl((g_{00},g_{01},T_0),(g_{10},g_{11},T_1)\bigr)\in K_{(\Ker S,\Im S)}$, then, by \eqref{squarp3tris},
\[
\bigl((g_{00},g_{01},T_0),(g_{10},g_{11},T_1)\bigr)(u_0+z_0+u_1+z_1)= (g_{00} u_0+T_0g_{01} z_0)+g_{01}z_0+(g_{10}u_1+T_1g_{11}z_1)+g_{11}z_1
\]
so that, for $z_0=0$,
\begin{equation}\label{g11g10}
\bigl((g_{00},g_{01},T_0),(g_{10},g_{11},T_1)\bigr)\left[u_0+\left(z_1+\Im S\right)\right]
=g_{00} u_0+\left(g_{11}z_1+\Im S\right).
\end{equation}

Then with each $h=\bigl((g_{00},g_{01},T_0),(g_{10},g_{11},T_1)\bigr)\in H_S\leq K_{(\Ker S,\Im S)}$ 
(cf.\ Lemma \eqref{LemmaHS}) we associate the linear map $\mathcal{L}_h\in\GL(\Ker S)\times \GL\left(V_1/\Im S\right)$ defined by means of \eqref{g11g10}: 
\begin{equation}
\label{e:espressione-l-h}
\mathcal{L}_h\left[u_0+(z_1+\Im S)\right]\coloneqq g_{00} u_0+\left(g_{11}z_1+\Im S\right).
\end{equation}

\begin{proposition}\label{Prop2qJohn}
Let $\mathcal{L} \colon H_S \to \GL(\Ker S)\times \GL\left(V_1/\Im S\right)$ denote the map defined by \eqref{e:espressione-l-h}.
Then the following holds.
\begin{enumerate}[{\rm (1)}]
\item  $\mathcal{L}$ is a surjective group homomorphism;
\item $\Ker(\mathcal{L})$ consists of all elements $h=\bigl((g_{00},g_{01},T_0),(g_{10},g_{11},T_1)\bigr)\in H_S$ such that
$g_{00} = \Id_{\Ker S}$ and $g_{11} = \Id_{Z_1}$;
\item Suppose that $W_0\leq\Ker S$ and $\Im S\leq W_1$. Then the $\mathcal{L}$-image of 
$K_{(W_0,W_1)}\cap H_S$ is the stabilizer of $(W_0,W_1/\Im S)$ in $\GL(\Ker S)\times \GL\left(V_1/\Im S\right)$.
\end{enumerate}
\end{proposition}
\begin{proof}
(1) Surjectivity is obvious: by Lemma \ref{LemmaHS}, the elements $g_{00} \in \GL(\Ker S)$ and $g_{11} \in \GL(Z_1)$ are arbitrary.
\par
Also, given $h=\bigl((g_{00},g_{01},T_0),(g_{01},g_{11},T_1)\bigr)$ and 
$h'=\bigl((g_{00}',g_{01}',T_0'),(g_{01}',g_{11}',T_1')\bigr)$ in $H_S$,
\[
\mathcal{L}_{h}\mathcal{L}_{h'}\left[u_0+(z_1+\Im S)\right]=\mathcal{L}_h\left[g_{00}' u_0+\left(g_{11}'z_1+\Im S\right)\right]=g_{00}g_{00}' u_0+\left(g_{11}g_{11}'z_1+\Im S\right),
\]
and from \eqref{tensp2} we immediately conclude that $\mathcal{L}_{h}\mathcal{L}_{h'}=\mathcal{L}_{hh'}$. 
This shows that $\mathcal{L}$ is a surjective group homomorphism.
\par
(2) The expression for the kernel of $\mathcal{L}$ is immediately deduced from \eqref{e:espressione-l-h}.
\par
(3) Observe that $\mathcal{L}_h$ stabilizes $(W_0,W_1/\Im S)$ if and only if $g_{00}z_0 \in W_0$ 
(resp.\ $g_{11}z_1+\Im S\in W_1/\Im S$, equivalently, $g_{11}z_1\in W_1$) for all $z_0\in W_0$ 
(resp.\ $z_1\in W_1$), that is, $h$ stabilizes $(W_0,W_1)$.
\end{proof}

Suppose that $W_0\leq\Ker S$ and $\Im S\leq W_1$. 
It follows from Proposition \ref{Prop2qJohn}.(2) that the kernel of $\mathcal{L}$ is contained in 
$K_{(W_0,W_1)}\cap H_S$. From the Third Isomorphism Theorem for groups and Proposition \ref{Prop2qJohn}.(1) and (3), 
we then deduce that $\mathcal{L} \colon H_S \to \GL(\Ker S)\times \GL\left(V_1/\Im S\right)$ induces a group isomorphism
\[
H_S/\!\left(K_{(W_0,W_1)}\cap H_S\right) \cong \left(\GL(\Ker S)\times \GL\left(V_1/\Im S\right)\right)\!/\Stab(W_0,W_1/\Im S).
\]
Since $\Stab(W_0,W_1/\Im S) = \Stab(W_0) \times \Stab(W_1/\Im S)$, we deduce the following:

\begin{corollary}\label{Cor2qJohn}
Let $(W_0,W_1)\in Y_{r,s}$ such that $W_0\leq\Ker S$ and $\Im S\leq W_1$, and set $\ell \coloneqq \rk(S)$. 
Then, as homogeneous spaces,
\begin{equation}\label{mathcalFelim}
\begin{split}
H_S/\!\left(K_{(W_0,W_1)}\cap H_S\right) & \cong \GL(\Ker S)/\Stab(W_0) \times \GL\left(V_1/\Im S\right)/\Stab(W_1/\Im S)\\ 
& \equiv \mathcal{G}_{m-\ell,r}\times\mathcal{G}_{n-m-\ell,s-\ell}. 
\end{split}
\end{equation}  
\hfill\qed
\end{corollary}

For vector subspaces $U_0\leq V_0$ and $U_1\leq V_1$ set 
\begin{equation}
\label{e:D-y}
D_{(U_0,U_1)}\coloneqq\{\Omega\in\Hom(V_0,V_1)\colon\Ker\Omega\geq U_0, \Im\Omega\leq U_1\}
\end{equation}
and, for every $\Omega\in D_{(U_0,U_1)}$, denote by $\Omega_{U_0,U_1}\in\Hom(V_0/U_0,U_1)$ the quotient operator defined by setting 
\begin{equation}\label{defOmegaV0U0}
\Omega_{U_0,U_1}(v_0+U_0)=\Omega(v_0),
\end{equation} 
for all $v_0+U_0\in V_0/U_0$. Clearly, 
\begin{equation}\label{OmegaU0U1}
\Omega=\iota_{U_1}\Omega_{U_0,U_1}\pi_{U_0},
\end{equation}
and the map $D_{(U_0,U_1)}\ni\Omega\mapsto\Omega_{U_0,U_1}\in \Hom(V_0/U_0,U_1)$ is a vector space isomorphism. The group $H$ 
(cf.\ \eqref{H et A}) acts on $\sqcup_{y\in Y}D_y$ in the obvious way: for $(g_0,g_1)\in H$, $(U_0,U_1)\in Y$, and 
$\Omega\in D_{(U_0,U_1)}$, then $(g_0,g_1)\Omega\equiv g_1\Omega g_0^{-1}\in D_{(g_0 U_0,g_1U_1)}$. 

Fix $y_0=(W_0,W_1)\in Y_{r,s}$. 
Then, for $\Omega\in D_{(W_0,W_1)}$, $k\in K_{(W_0,W_1)}$ as in \eqref{defW1W0}, $w_0\in W_0$, and $z_0\in Z_0$, 
arguing as in \eqref{starp32} we get:
\begin{equation}\label{kOmega}
\begin{split}
[k\Omega](w_0+z_0) & = (g_{10},g_{11},T_1)\Omega(g_{00}^{-1},g_{01}^{-1},-g_{00}^{-1}T_0g_{01})(w_0+z_0)\\
& = (g_{10},g_{11},T_1)\Omega\left[\left(g_{00}^{-1}w_0-g_{00}^{-1}T_0z_0\right)+g_{01}^{-1}z_0\right]\\
\left(g_{00}^{-1}w_0-g_{00}^{-1}T_0z_0\in W_0\leq\Ker \Omega\right)\quad & = (g_{10},g_{11},T_1)\Omega g_{01}^{-1}z_0\\
\left(\Omega g_{01}^{-1}z_0\in\Im\Omega\leq W_1\right)\quad & = g_{10}\Omega g_{01}^{-1}z_0.
\end{split}
\end{equation}
If also $\Omega'\in D_{(W_0,W_1)}$ and $\rk(\Omega')=\rk(\Omega)$, then $\rk\left(\Omega'|_{Z_0}\right)=\rk\left(\Omega|_{Z_0}\right)$ so that there exist $g_{01}\in\GL(Z_0)$ and $g_{10}\in\GL(W_1)$ such that $g_{10}\bigl(\Omega|_{Z_0}\bigr)g_{01}^{-1}=\Omega'|_{Z_0}$. From \eqref{kOmega} it follows that there exists $k\in K_{(W_0,W_1)}$ such that $k\Omega=\Omega'$. Then, for $\ell=0,1,\dotsc,\min\{m-r,s\}$ we choose $\Omega_\ell\in D_{(W_0,W_1)}$ such that $\rk(\Omega_\ell)=\ell$ so that
\begin{equation}\label{ThetaHom}
\{\Omega_\ell\colon \ell=0,1,\dotsc,\min\{m-r,s\}\}
\end{equation}
constitutes a complete set of representatives for the orbits of $K_{(W_0,W_1)}$ on $D_{(W_0,W_1)}$.

\begin{lemma}\label{LemmaOmegaT}
Let $(W_0,W_1)\in Y_{r,s}$. Then, for $U_0\leq V_0$, $U_1\leq V_1$, and $\Omega\in D_{(U_0,U_1)}$, one has
$\Tr[\Omega T]=0$ for all $T\in C$ (cf.\ \eqref{Cinthiscase}) if and only if $\Ker\Omega\geq (W_0+U_0)$ and $\Im\Omega\leq W_1\cap U_1$, that is, if and only if $\Omega\in D_{(W_0+U_0,W_1\cap U_1)}$.
\end{lemma}
\begin{proof}
Consider the decomposition
\begin{equation}\label{decV0bis}
V_0=X_0\oplus X_1\oplus X_2\oplus X_3,
\end{equation}
where $X_0\coloneqq W_0\cap U_0$, $X_1$ (resp.\ $X_2$) is a complementary subspace of $X_0$ in $W_0$ (resp.\ $U_0$) so that $W_0=X_0\oplus X_1$ (resp.\ $U_0=X_0\oplus X_2$), and $X_3$ is a complementary subspace of $U_0 + W_0 = X_0\oplus X_1 \oplus X_2$ in $V_0$. 
We may also suppose that $Z_0=X_2\oplus X_3$ (the choice of the complementary subspace $Z_0$ in \eqref{defW1W0} is arbitrary). 
Analogously, consider the decomposition
\begin{equation}\label{decV1bis}
V_1=Y_0\oplus Y_1\oplus Y_2\oplus Y_3,
\end{equation}
where $Y_0\coloneqq W_1\cap U_1$, $Y_1$ (resp.\ $Y_2$) is a complementary subspace of $Y_0$ in $W_1$ (resp.\ $U_1$) so that 
$W_1=Y_0\oplus Y_1$ (resp.\ $U_1=Y_0\oplus Y_2$), and $Y_3$ is a complementary subspace of $U_1 + W_1 = Y_0\oplus Y_1 \oplus Y_2$ in $V_1$. 
Once again, we may also suppose that $Z_1=Y_2\oplus Y_3$. 
Then every $T\in C$ (resp.\ $\Omega\in D_{U_0,U_1}$) admits a decomposition 
\[
T=\left(\bigoplus_{i=0}^1\bigoplus_{j=0}^1T_{ij}\right)\bigoplus\left(\bigoplus_{i=2}^3\bigoplus_{j=0}^3T_{ij}\right),
\]
with $T_{ij}\in\Hom(Y_i,X_j)$ (resp.\ 
\[
\Omega=\Omega_{10}\oplus\Omega_{12}\oplus\Omega_{30}\oplus\Omega_{32},
\] 
with $\Omega_{ji}\in\Hom(X_j,Y_i)$) and therefore
\[
\Tr[\Omega T]=\Tr[\Omega_{10} T_{01}]+\Tr[\Omega_{12} T_{21}]+\Tr[\Omega_{32} T_{23}].
\]
It follows that $\Tr[\Omega T]=0$ for all $T\in C$ if and only if $\Omega_{10}=\Omega_{12}=\Omega_{32}=0$, equivalently, 
$\Omega\equiv \Omega_{30}\in \Hom(X_3,Y_0)$, and this holds if and only if 
$\Ker\Omega\geq X_0\oplus X_1\oplus X_2\equiv U_0+W_0$ and $\Im \Omega\leq Y_0=U_1\cap W_1$.
\end{proof}

For $U_0\leq V_0$ and $U_1\leq V_1$ as above, and $\ell \geq 0$, denote by $D_{(U_0,U_1)}^\ell$ the set of all 
$\Omega\in D_{(U_0,U_1)}$ such that $\rk(\Omega) = \ell$. By means of \eqref{kOmega}, we can compute
the action of an element $k=(k_0,k_1)\in K_{(W_0,W_1)}$ as in \eqref{defW1W0} on an operator $\Omega\in D_{(W_0+U_0,W_1\cap U_1)}$. Clearly, $k\Omega\equiv k_1\Omega k_0^{-1}\in D_{(W_0+k_0U_0,W_1\cap k_1U_1)}$ so that, using the decompositions in the proof of Lemma \ref{LemmaOmegaT}, we deduce that if $x_i\in k_0X_i$, $i=0,1,2,3$, then 
$\left[k\Omega\right](x_0+x_1+x_2+x_3)=\left(g_{10}|_{W_1\cap U_1}\right)\Omega_{30} g_{01}^{-1}x_3$, that is,
\begin{equation}\label{kOmegax}
[k\Omega]_{W_0+k_0U_0,W_1\cap k_1U_1}=\left(g_{10}|_{W_1\cap U_1}\right)\Omega_{W_0+U_0,W_1\cap U_1} \left(g_{01}^{-1}|_{V_0/(W_0+k_0U_0)}\right),
\end{equation}
where $g_{01}^{-1}|_{V_0/(W_0+k_0U_0)}$ is defined as in \eqref{2starp28}.
\par
With $i$ and $h$ as in \eqref{hilvar2} and for
\begin{equation}\label{hilvar}
0\leq \ell\leq \min\{i,m-2r+h\},
\end{equation}
we choose $\Omega_{h,i}^\ell\in D_{(W_0+U_{0,h},W_1\cap U_{1,i})}^\ell$. It follows from \eqref{kOmegax} that
the set $\Gamma'$ of all triples
\begin{equation}\label{defGammaprime}
\left(U_{0,h},U_{1,i},\Omega_{h,i}^\ell\right)
\end{equation}
is a complete set of representatives for the orbits of $K$ on $\{(U_0,U_1,\Omega)\colon (U_0,U_1)\in Y_{r,s},\Omega\in D_{(W_0+U_0,W_1\cap U_1)}\}$.

We now examine the condition in Lemma \ref{LemmaGammatheta} in the present setting.

\begin{lemma}\label{lemmaHKorbits}
Let $\Omega_\ell \in D_{(W_0,W_1)}$ as in \eqref{ThetaHom}. Then the $H_{\Omega_\ell}$-orbit of $(W_0,W_1)$ intersects the 
$K_{(W_0,W_1)}$-orbit of $(U_{0h},U_{1i})$ (cf.\ \eqref{hilvar2}) if and only if $h\geq 2r-m+ \ell$ and $i\geq \ell$.
\end{lemma}
\begin{proof}
Recall that $\Omega_\ell \in \Hom(V_0,V_1)$ with $\Ker \Omega_\ell \geq W_0$ and ${\rm Im} \Omega_\ell \leq W_1$,
and $(g_0,g_1) \in H$ fixes $\Omega_\ell$ if $g_1 \Omega_\ell g_0^{-1} = \Omega_\ell$.
It follows that if $(g_0,g_1)\in H_{\Omega_\ell}$ and $(U_0,U_1)=(g_0,g_1)(W_0,W_1)$, 
then, on the one hand, $\Omega_\ell U_0 = g_1 \Omega_\ell g_0^{-1} (g_0 W_0) = g_1 \Omega_\ell W_0 = 0$ (since $W_0 \leq \Ker \Omega_\ell$)
so that $U_0\leq\Ker\Omega_\ell$, and, on the other hand, 
$\Omega_\ell V_0 = g_1 \Omega_\ell g_0^{-1} V_0 \leq g_1 \Omega_\ell V_0 \leq g_1 W_1 = U_1$, so that
$\Im\Omega_\ell\leq U_1$. (Alternatively, recall that $H_{\Omega_\ell}\leq K_{(\Ker\Omega_\ell,\Im\Omega_\ell)}$, by Lemma \ref{LemmaHS}.)
Consequently, if the orbits in the statement intersect, since $W_0 + U_0 \leq \Ker \Omega_\ell$, $h = \dim(W_0\cap U_{0,h})$,  
$r = \dim W_0 = \dim U_0$, and $m = \dim V_0$, from the Grassmann identity we deduce that
\[
h = \dim(W_0\cap U_{0,h}) = \dim(W_0\cap U_0)=2r-\dim(W_0+U_0)\geq 2r-\dim\Ker\Omega_\ell=2r-(m-\ell)
\]
and, since ${\rm Im}\Omega_\ell \leq W_1\cap U_1$, trivially
\[
i = \dim(W_1\cap U_{1,i}) = \dim(W_1\cap U_1) \geq \dim({\rm Im}\Omega_\ell) = \rk(\Omega_\ell) = \ell.
\]
Conversely, suppose that $h\geq 2r-m+\ell$ (so that, once again by the Grassmann identity, 
$\dim(W_0+ U_{0h})=2r-h\leq m- \ell=\dim(\Ker\Omega_\ell)$) and 
$i\geq \ell$ (so that $\dim(W_1 \cap U_{1,i}) = i \geq \ell = \rk(\Omega_\ell) = \dim({\rm Im}\Omega_\ell)$).
Then we can find $k=(k_0,k_1)\in K_{(W_0,W_1)}$ such that 
\[
k_0(W_0+ U_{0h})\equiv W_0+(k_0U_{0h})\leq \Ker\Omega_\ell\quad\text{ and }\quad k_1(W_1\cap U_{1i})\equiv W_1\cap (k_1U_{1i})\geq \Im\Omega_\ell.
\] 
From Proposition \ref{Prop2qJohn} it follows that $H_{\Omega_\ell}$ ``induces'' all the operators in $\GL(\Ker \Omega_\ell)\times \GL(V_1/\Im\Omega_\ell)$, so that there exists $(g_0,g_1)\in H_{\Omega_\ell}$ such that $(g_0,g_1)(W_0,W_1)=(k_0U_{0h},k_1U_{1i})$. Now, the latter
clearly belongs to both the $H_{\Omega_\ell}$-orbit of $(W_0,W_1)$ and the $K_{(W_0,W_1)}$-orbit of $(U_{0h},U_{1i})$.
\end{proof}

Now suppose that $(U_0,U_1,\Psi), (W_0,W_1,\Phi)\in X$. Since $U_0\leq U_0+W_0\leq V_0$ we have a canonical projection $\pi_{U_0,W_0}\colon V_0/U_0\rightarrow V_0/(U_0+W_0)$, whose kernel is $(U_0+W_0)/U_0\simeq W_0/(U_0\cap W_0)$; similarly, we define $\pi_{W_0,U_0}\colon V_0/W_0\rightarrow V_0/(U_0+W_0)$.  

\begin{equation}
\label{bigdiagram}
\qquad\xymatrix@R=10pt{
&V_0/U_0\ar[dr]^{\pi_{U_0,W_0}}&\\
U_1\cap W_1\ar[ur]^{\Psi|_{U_1\cap W_1}}\ar[dr]_{\Phi|_{U_1\cap W_1}}& & \quad V_0/(U_0+W_0)\\
&V_0/W_0\ar[ur]_{\pi_{W_0,U_0}}
}
\end{equation}

\begin{lemma}
Let $\Omega\in D_{(W_0+U_0,W_1\cap U_1)}$ and $\Psi\in B_{(U_0,U_1)}$. Then
\begin{equation}\label{TrOmPsi}
\Tr\left[\Omega_{U_0,U_1}\Psi\right]=\Tr\left[\Omega_{U_0+W_0,U_1\cap W_1}\left(\pi_{U_0,W_0}\Psi|_{U_1\cap W_1}\right)\right].
\end{equation}
\end{lemma}
\begin{proof}
Arguing as in the proof of Lemma \ref{LemmaOmegaT} and using the same notation therein, one easily finds that 
\[
\Tr\left[\Omega_{U_0,U_1}\Psi\right]=\Tr\left[\Omega_{30}\widetilde{\Psi}_{03}\right],
\]
where $\widetilde{\Psi}_{03}\colon Y_0\rightarrow X_3$ in the decomposition 
\[
\Psi=\Psi_{01}+\Psi_{03}+\Psi_{21}+\Psi_{23}\quad\text{ with }\quad\Psi_{ij}\in\Hom\left(Y_i,(X_j\oplus U_0)/U_0\right)
\]
and $\widetilde{\Psi}_{03}\colon Y_0\rightarrow X_3$ is defined by means of $(X_3\oplus U_0)/U_0\simeq X_3$. 
\end{proof}

For $\Psi \in B_{(U_0,U_1)} = \Hom(U_1,V_0/U_0)$ and $\Phi \in B_{(W_0,W_1)} = \Hom(W_1,V_0/W_0)$ set (cf.\ \eqref{bigdiagram})
\begin{equation}\label{defPsicapPhi}
\Psi\cap \Phi\coloneqq\{v\in U_1\cap W_1\colon \pi_{U_0,W_0}[\Psi(v)]=\pi_{W_0,U_0}[\Phi(v)]\}
\end{equation}
and define $\delta\colon X\times X\rightarrow \mathbb{N}\times\mathbb{N}\times\mathbb{N}$ by 
\begin{equation}\label{defdelta}
\delta\bigl((U_0,U_1,\Psi), (W_0,W_1,\Phi)\bigr)\coloneqq \Bigl(\dim (U_0\cap W_0),\dim (U_1\cap W_1),\dim(\Psi\cap \Phi) \Bigr).
\end{equation}

\begin{theorem}
The pairs $\bigl((U_0,U_1,\Psi), (W_0,W_1,\Phi)\bigr)$ and $\bigl((U_0',U_1',\Psi'), (W_0',W_1',\Phi')\bigr)$ belong to the same 
diagonal $(H \ltimes A)$-orbit in $X\times X$ if and only if 
\begin{equation}\label{deltaeq}
\delta\bigl((U_0,U_1,\Psi), (W_0,W_1,\Phi)\bigr)=\delta\bigl((U_0',U_1',\Psi'), (W_0',W_1',\Phi')\bigr).
\end{equation}
\end{theorem}
\begin{proof}
The pairs in the statement belong to the same diagonal $(H \ltimes A)$-orbit in $X\times X$ if there exists 
$(g_0,g_1,T)\in H \ltimes A$ such that
\begin{equation}\label{deltaeq2}
(g_0,g_1,T)(U_0,U_1,\Psi)=(U_0',U_1',\Psi')\text{ and } (g_0,g_1,T)(W_0,W_1,\Phi)=(W_0',W_1',\Phi'),
\end{equation}
that is (cf.\ \eqref{squarp3bis}),
\begin{equation}\label{deltaeq3}
\begin{cases}
g_0U_0=U_0',\quad g_0W_0=W_0'\\
g_1U_1=U_1',\quad g_1W_1=W_1'\\
\pi_{U_0'}T\iota_{U_1'}+g_0\Psi g_1^{-1}=\Psi'\\
\pi_{W_0'}T\iota_{W_1'}+g_0\Phi g_1^{-1}=\Phi'.\\
\end{cases}
\end{equation}

If this is the case, we immediately have that $\dim(U_0 \cap W_0) = \dim(U_0' \cap W_0')$ and 
$\dim(U_1 \cap W_1) = \dim(U_1' \cap W_1')$. We claim that $g_1$ yields a linear isomorphism $\Psi\cap \Phi \to \Psi'\cap \Phi'$.
Indeed, for $v \in \Psi\cap \Phi$ we have
\[
\begin{split}
\pi_{U_0',W_0'}[\Psi'(g_1v)] & = \pi_{g_0U_0,g_0W_0}[\pi_{g_0U_0}T(g_1v) + g_0\Psi(v)] \\
& = \pi_{g_0U_0,g_0W_0}[g_0\Psi(v)] \\
& = g_0\pi_{U_0,W_0}[\Psi(v)]\\
& = g_0\pi_{W_0,U_0}[\Phi(v)] \ \ \mbox{ \ \ \ \ \ (as $v \in \Psi\cap \Phi$)}\\
& = \pi_{g_0W_0,g_0U_0}[g_0\Phi(v)]\\
& = \pi_{g_0W_0,g_0U_0}[\pi_{g_0W_0}T(g_1v) + g_0\Phi(v)]\\
& = \pi_{W_0',U_0'}[\Phi'(g_1v)],
\end{split}
\]
and the claim follows. We deduce that $\dim(\Psi\cap \Phi) = \dim(g_1(\Psi\cap \Phi)) = \dim(\Psi'\cap \Phi')$, so that
\eqref{deltaeq} holds, showing the only if part.
\par
Conversely, assume \eqref{deltaeq}. Decompose $V_0$ as
\begin{equation}\label{decV0}
V_0=X_0\oplus X_1\oplus X_2\oplus X_3= X_0'\oplus X_1'\oplus X_2'\oplus X_3',
\end{equation}
where 
\begin{itemize}
\item $X_0\coloneqq U_0\cap W_0$ and $X_0'\coloneqq U_0'\cap W_0'$,
\item $X_1,X_2,X'_1$ and $X_2'$ are complementary subspaces such that 
$U_0=X_0\oplus X_1$, $U_0'=X_0'\oplus X_1'$, $W_0=X_0\oplus X_2$, and $W_0'=X_0'\oplus X_2'$, 
\item $X_3$ and $X_3'$ are complementary subspaces such that \eqref{decV0} holds true. 
\end{itemize}
Observe that by our hypothesis \eqref{deltaeq} one has that $\dim X_0=\dim X_0'$ so that $\dim X_j=\dim X_j'$, for $j=1,2,3$. 
We then decompose $V_1$ as
\begin{equation}\label{decV1}
V_1=Z_0\oplus Z_1\oplus Z_2\oplus Z_3\oplus Z_4= Z_0'\oplus Z_1'\oplus Z_2'\oplus Z_3'\oplus Z_4'
\end{equation}
where 
\begin{itemize}
\item $Z_0\coloneqq \Psi\cap \Phi$ and $Z_0'\coloneqq \Psi'\cap \Phi'$, 
\item $Z_1,Z_2,Z_3,Z'_1,Z_2'$, and $Z_3'$ are complementary subspaces such that $U_1\cap W_1=Z_0\oplus Z_1$, $U_1'\cap W_1'=Z_0'\oplus Z_1'$, $U_1=Z_0\oplus Z_1\oplus Z_2$, $U_1'=Z_0'\oplus Z_1'\oplus Z_2'$, $W_1=Z_0\oplus Z_1\oplus Z_3$, and $W_1'=Z_0'\oplus Z_1'\oplus Z_3'$, 
\item $Z_4$ and $Z_4'$ are complementary subspaces such that \eqref{decV1} holds true. 
\end{itemize}
Again, by \eqref{deltaeq} one has that $\dim Z_0=\dim Z_0'$ and $\dim Z_1=\dim Z_1'$, so that $\dim Z_k=\dim Z_k'$, for $k=2,3,4$.

We denote by $\pi_j\colon V_0\rightarrow X_j$ and $\pi_j'\colon V_0\rightarrow X_j'$, $j=0,1,2,3$ (resp.\ $\iota_k\colon Z_k\rightarrow V_1$, 
$\iota_k'\colon Z_k'\rightarrow V_1$, $k=0,1,2,3,4$) the canonical projections (resp.\ inclusions) and let
\[
\sigma_j\colon V_0/U_0\rightarrow X_j, \qquad\sigma_j'\colon V_0/U_0'\rightarrow X_j',\quad j=2,3
\]
and
\[
\theta_j\colon V_0/W_0\rightarrow X_j,\qquad \theta'_j\colon V_0/W_0'\rightarrow X_j', \quad j=1,3,
\]
denote the (not canonical) surjective linear maps defined in the obvious way: for instance, $\sigma_2\left(x_2+x_3+U_0\right)\coloneqq x_2$ for $x_2\in X_2$ and $x_3\in X_3$. Clearly, $\sigma_2+\sigma_3$ and other similar sums are isomorphisms. Consider also the (not canonical) isomorphisms $\xi\colon V_0/(U_0+W_0)\rightarrow X_3$ and $\xi'\colon V_0/(U_0'+W_0')\rightarrow X_3'$ defined by setting $\xi\bigl(x_3+\left(U_0+W_0\right)\bigr)\coloneqq x_3$, for $x_3 \in X_3$, and similarly for $\xi'$. Keeping in mind \eqref{bigdiagram}, we have the commutative diagrams:
\begin{equation}\label{varident}
\begin{split}
\quad&\\
\quad&\\
\sigma_3=\xi\pi_{U_0,W_0},\quad&\theta_3=\xi\pi_{W_0,U_0}\\
\sigma_3'=\xi'\pi_{U_0',W_0'},\quad&\theta_3'=\xi'\pi_{W_0',U_0'}\\
\end{split}
\qquad
\xymatrix{
V_0/U_0\ar[rr]^{\pi_{U_0,W_0}\quad}\ar[drr]_{\sigma_3}&&V_0/(U_0+W_0)\ar[d]^\xi\\
&&X_3
}
\end{equation}

Choose $g_j\in \GL(V_j)$, $j=0,1$, in such a way that 
\begin{equation}\label{defg1g0}
g_0(X_j)=X_j',\; j=0,1,2,3,\quad\text{ and }\quad g_1(Z_k)=Z_k',\; k=0,1,2,3,4.
\end{equation}
Then the four conditions in the first two rows in \eqref{deltaeq3} are satisfied:
\[
g_0(U_0)=g_0(X_0\oplus X_1)=X_0'\oplus X_1'=U_0',
\]
and similarly for the other three conditions. Actually, the choice of $g_0|_{X_j}$, $j=0,1,2$, and of $g_1|_{Z_k}$, $k=0,1,2,3,4$, is arbitrary, while $g_0|_{X_3}$ must satisfy a condition that will be determined at the end of the proof.

We have $g_0(\sigma_2+\sigma_3)=(\sigma'_2+\sigma_3')g_0$ and $g_0(\theta_1+\theta_3)=(\theta_1'+\theta_3')g_0$. For instance, if $v_0\in V_0$ and $v_0=x_0+x_1+x_2+x_3$ with $x_j\in X_j$ we have
\[
\begin{split}
\quad&\\
\quad&\\
\sigma'_2g_0(v_0+U_0)& = \sigma'_2(g_0x_2+g_0x_3+U_0')\\
& = g_0x_2\\
& = g_0\sigma_2(v_0+U_0).
\end{split}
\qquad\qquad
\xymatrix{
V_0/U_0\ar[d]_{g_0}\ar[r]^{\sigma_2}&X_2\ar[d]_{g_0}\\
V_0/U_0'\ar[r]_{\sigma'_2}&X_2'
}
\]

By multiplying the third row of \eqref{deltaeq3} by $\sigma_2'+\sigma_3'$  (resp.\ the fourth row by $\theta_1'+\theta_3'$) on the left,
we get the equivalent conditions
\begin{equation}\label{deltaeq4}
\begin{cases}
(\sigma_2'+\sigma_3')\pi_{U_0'}T\iota_{U_1'}+g_0(\sigma_2+\sigma_3)\Psi g_1^{-1}=(\sigma_2'+\sigma_3')\Psi'\\
(\theta_1'+\theta_3')\pi_{W_0'}T\iota_{W_1'}+g_0(\theta_1+\theta_3)\Phi g_1^{-1}=(\theta_1'+\theta_3')\Phi'.\\
\end{cases}
\end{equation}
Taking into account 
\begin{itemize}
\item
$(\sigma_2+\sigma_3)\Psi =(\sigma_2+\sigma_3)\Psi \iota_{U_1}=(\sigma_2+\sigma_3)\Psi(\iota_0+\iota_1+\iota_2)$
and other similar relations;
\item for $g_1$ defined by means of \eqref{defg1g0}, we may view $g_1^{-1}$ as restricted to $g_1U_1=U_1'$ (resp.\ to $g_1W_1=W_1'$) 
and we may write  $g_1^{-1}\iota_k'=\iota_kg_1^{-1}\iota_k'$;
\item $\sigma_2\pi_{U_0}=\pi_2$ because $\sigma_2\pi_{U_0}(x_1+x_2+x_3+x_4)=\sigma_2(x_2+x_3+U_0)=x_2=\pi_2(x_1+x_2+x_3+x_4)$
and other similar relations,
\end{itemize}

\noindent
we deduce that \eqref{deltaeq4} is in turn equivalent to
\begin{equation}\label{deltaeq5}
\begin{cases}
\pi_j'T\iota_k'+g_0\sigma_j\Psi \iota_k g_1^{-1}\iota_k'=\sigma_j'\Psi'\iota_k',\quad j=2,3,\;k=0,1,2,\\
\pi_j'T\iota_k'+g_0\theta_j\Phi \iota_k g_1^{-1}\iota_k'=\theta_j'\Phi'\iota_k',\quad j=1,3,\;k=0,1,3.\\
\end{cases}
\end{equation}
Decomposing $T$ as a sum of linearly independent operators:
\[
T=\sum_{j=0}^3\sum_{k=0}^4\;\pi'_jT\iota_k'
\]
we have that the operators $\pi_0'T\iota_k'\in\Hom\left(Z_k',X_0'\right)$, $k=0,1,2,3,4$ (resp.\ $\pi_j'T\iota_4'\in\Hom\left(Z_4',X_j'\right)$, $j=0,1,2,3$, resp.\  $\pi_2'T\iota_3'\in\Hom\left(Z_3',X_2'\right)$, resp.\ $\pi_1'T\iota_2'\in\Hom\left(Z_3',X_1'\right)$) 
may be chosen arbitrarily because they do not appear in \eqref{deltaeq5}.
\par
We then set
\[
\begin{split}
&\pi'_2T\iota'_k\coloneqq\sigma_2'\Psi'\iota_k'-g_0\sigma_2\Psi\iota_kg_1^{-1}\iota_k',\quad k=0,1,2,\\
&\pi'_1T\iota'_k\coloneqq\theta_1'\Phi'\iota_k'-g_0\theta_1\Phi\iota_kg_1^{-1}\iota_k',\quad k=0,1,3,\\
&\pi'_3T\iota'_2\coloneqq\sigma_3'\Psi'\iota_2'-g_0\sigma_3\Psi\iota_2g_1^{-1}\iota_2',\\
&\pi'_3T\iota'_3\coloneqq\theta_3'\Phi'\iota_3'-g_0\theta_3\Phi\iota_3g_1^{-1}\iota_3',\\
\end{split}
\]
because these conditions on the operators in the left hand sides are contained in only one row of \eqref{deltaeq5}. On the other hands, the operators $\pi'_3T\iota_0'$ and $\pi'_3T\iota_1'$ appear in both rows of \eqref{deltaeq5} and therefore they require a more careful analysis.

First of all, from \eqref{defPsicapPhi}, \eqref{varident}, and the definition of $Z_0$ and $Z_0'$, we get 
\[
\sigma_3\Psi\iota_0=\xi\pi_{U_0,W_0}\Psi\iota_0=\xi\pi_{W_0,U_0}\Phi\iota_0=\theta_3\Phi\iota_0
\]
and similarly $\sigma_3'\Psi'\iota_0'=\theta_3'\Phi'\iota_0'$ so that we may set
\[
\pi'_3T\iota_0'\coloneqq\sigma_3'\Psi'\iota_0'-g_0\sigma_3\Psi\iota_0g_1^{-1}\iota_0'\equiv\theta_3'\Phi'\iota_0'-g_0\theta_3\Phi\iota_0g_1^{-1}\iota_0'.
\]

Finally, from \eqref{defPsicapPhi}, the definition of $Z_1,Z_1'$, and \eqref{varident} we deduce that
\[
\sigma_3\Psi\iota_1(z_1)\neq\theta_3\Phi\iota_1(z_1)\;\text{ for all }z_1\in Z_1 \ \mbox{ and } \ \sigma_3'\Psi'\iota_1'(z_1')\neq\theta_3'\Phi'\iota_1'(z_1')\;\text{ for all }z_1'\in Z_1',
\]
so that the maps
\begin{equation}
\label{e:injective}
\sigma_3\Psi\iota_1-\theta_3\Phi\iota_1\colon Z_1\longrightarrow X_3,\qquad\quad \sigma_3'\Psi'\iota_1'-\theta_3'\Phi'\iota_1'\colon Z_1'\longrightarrow X_3',
\end{equation}
are injective. Given $g_1|_{Z_1}$, we are now in a position to determine the condition for $g_0|_{X_3}$ we alluded to above.
Just define $g_0|_{X_3}$ in such a way that the following diagram is commutative:

\qquad\qquad\qquad\qquad\qquad\qquad
\xymatrix{
Z_1\ar[d]^{g_1|_{Z_1}}\ar[rrr]^{\sigma_3\Psi\iota_1-\theta_3\Phi\iota_1}&&&X_3\ar[d]^{g_0|_{X_3}}\\
Z_1'\ar[rrr]_{\sigma_3'\Psi'\iota_1'-\theta_3'\Phi'\iota_1'}&&&X_3'
}
\par
\noindent
(this is possible since \eqref{e:injective} are injective).
It follows that $\sigma_3'\Psi'\iota_1'-g_0\sigma_3\Psi \iota_1 g_1^{-1}\iota_1'=\theta_3'\Phi'\iota_1'-g_0\theta_3\Phi \iota_1 g_1^{-1}\iota_1'$
and therefore we may set
\[
\pi_3'T\iota_1'\coloneqq\sigma_3'\Psi'\iota_1'-g_0\sigma_3\Psi \iota_1 g_1^{-1}\iota_1'.
\]
This completes the construction of the element $(g_0,g_1,T) \in H \ltimes A$ mapping $(U_0,U_1,\Psi)$ (resp.\ $(W_0,W_1,\Phi)$)
to $(U_0',U_1',\Psi')$ (resp.\ $(W_0',W_1',\Phi')$), and this ends the proof.
\end{proof}

\begin{corollary}\label{Cororbits}
Fix $y_0=(W_0,W_1)\in Y_{r,s}$. Then the orbits of $K\ltimes C$ (the stabilizer of $(y_0,0_{B_{y_0}})$) on $X$ are 
\[
\left\{(U_0,U_1,\Psi)\in X\colon \delta\bigl((U_0,U_1,\Psi), (W_0,W_1,0_{B_{y_0}})\bigr)=(h,i,j)\right\},
\]
for 
\begin{equation}\label{hijvar}
2r-m\leq h\leq r,\ \ \  2s-n+m\leq i\leq s, \mbox{ \ and \ } 0\leq j\leq i.
\end{equation}
\hfill \qed
\end{corollary}

A complete set of representatives for the $(K \ltimes C$)-orbits on $X$ is then given by (cf.\ \eqref{defGamma0})
\begin{equation}\label{rephijvar}
\Gamma_0\coloneqq\left\{\left(U_{0h},U_{1i},\Psi_{hij}\right): h,i,j \text{ as in }\eqref{hijvar}\right\},
\end{equation}
where $(U_{0h},U_{1i})$ are as in \eqref{hilvar2} and $\Psi_{hij}\in\Hom(U_{1i},V_0/U_{0h})$ is chosen in such a way that $\dim\left(\Psi_{hij}\cap 0_{B_{y_0}}\right)=j$. Note that, by virtue of \eqref{defPsicapPhi}, the latter condition is equivalent to
\begin{equation}\label{starp40}
\dim\Ker\left(\pi_{U_0,W_0}\Psi_{hij}|_{W_1\cap U_{1i}}\right)=j.
\end{equation}

\begin{corollary}\label{lastCor}
If $V$ is a vector space on a finite field $\mathbb{F}_q$ then $\left(\left(H\ltimes A\right),\left(K\ltimes C\right)\right)$ is a finite, symmetric Gelfand pair. In particular, its spherical functions are real valued.
\end{corollary}
\begin{proof}
The function $\delta$ in \eqref{defdelta} is symmetric and therefore we may apply the criterion in \eqref{symmGel2} (cf.\ \cite[Lemma 4.3.4]{book}). The fact that the spherical functions of a finite, symmetric Gelfand pair are real valued is a characterization of symmetry due to Adriano Garsia 
(cf.\ \cite[Theorem 4.8.2]{book}).
\end{proof}

\section{Spherical functions for the $q$-analog of the nonbinary Johnson scheme}
\label{SectiondecpermqqJoh}

In Section \ref{qJohstructure} above we have reformulated and generalized Dunkl's analysis of the $q$-analog of the nonbinary Johnson scheme. 
In the present section we revisit this symmetric Gelfand pair in terms of the notions and results 
that we have developed in Sections \ref{Sectiondecperm}, \ref{SecIntBasis}, and \ref{Secsphfunc}. 
\par
We assume that $V=\mathbb{F}_q^n$ (the canonical $n$-dimensional vector space over the finite field $\mathbb{F}_q$, $q$ a prime power), and that we have a decomposition $V=V_0\oplus V_1$ with $V_0\simeq\mathbb{F}_q^m$ and $V_1\simeq\mathbb{F}_q^{n-m}$, $0\leq m\leq n-1$. 
Then, cf.\ \eqref{H et A}, $H=H_{m,n-m}\cong \GL(\mathbb{F}_q^m)\times\GL(\mathbb{F}_q^{n-m})$ and $A=A_{n-m,m}\simeq\Hom(\mathbb{F}_q^{n-m},\mathbb{F}_q^m)$. Fix a nontrivial additive character $\tau$ of $\mathbb{F}_q$ (cf.\ \cite[Section 7.1]{book4}); then the map 
$\Hom(\mathbb{F}_q^m,\mathbb{F}_q^{n-m})\ni S\mapsto \psi_S\in\widehat{A}$, where $\psi_S(T)\coloneqq\tau(\Tr[ST]), T\in A$, is an isomorphism of abelian groups (cf.\ \cite[Exercise 8.7.14]{book} and \cite[p.\ 18]{DUNKL}). The action \eqref{defhtransl} is then given by 
\begin{align*}
\left[^{(g_0,g_1)}\,\psi_S\right](T)& = \psi_S\left(g_0^{-1}Tg_1\right)&(\text{by }\eqref{(star)p28})\\
& = \tau\left(\Tr\left[Sg_0^{-1}Tg_1\right]\right)=\tau\left(\Tr\left[g_1Sg_0^{-1}T\right]\right)=\psi_{g_1Sg_0^{-1}}(T)&
\end{align*}
that is,
\begin{equation}\label{defhtranslpsi}
^{(g_0,g_1)}\,\psi_S=\psi_{g_1Sg_0^{-1}}
\end{equation}
for all $(g_0,g_1) \in H$. It follows that the orbits of $H$ on $\widehat{A}$ are parameterized by the orbits of $H$ on 
$\Hom(\mathbb{F}_q^m,\mathbb{F}_q^{n-m})$, and a complete system of representatives for these orbits is 
given by $\Pi\coloneqq\{\psi_{S_\ell}\colon \ell=0,1,\dotsc,\min\{m,n-m\}\}$, where $S_\ell$ is as in \eqref{XiHom}.
\par
Fix $0 \leq r \leq m$, $0 \leq s \leq n - m$. For $y=(U_0,U_1)\in Y_{r,s}$ (cf.\ \eqref{e:Y-Yrs}) 
we have $B_y\simeq \Hom(\mathbb{F}_q^s,\mathbb{F}_q^m/\mathbb{F}_q^r)$ 
(cf.\ \eqref{e:B-y}), and with each $\Omega\in D_y$ (cf.\ \eqref{e:D-y}) we associate the character 
$\chi_\Omega\in\widehat{B_y}$ defined by setting (cf.\ \eqref{defOmegaV0U0})
\begin{equation}\label{defchiOmega}
\chi_\Omega(\Psi)\coloneqq \tau\left(\Tr\left[\Omega_{U_0,U_1}\Psi\right]\right),\qquad \Psi\in B_y. 
\end{equation}
Then the map $D_y\ni\Omega\rightarrow\chi_\Omega\in\widehat{B_y}$ is an isomorphism of abelian groups. Moreover, \eqref{defhtransl2} becomes
\begin{equation}\label{actchiOmega}
\begin{split}
\left[\!^{(g_0,g_1)}\chi_\Omega\right](\Psi)& = \chi_\Omega\left(g_0^{-1}\Psi g_1\right)\qquad\qquad\qquad\qquad\qquad\qquad\qquad\quad (\text{by }\eqref{2starp28})\\
& = \tau\left(\Tr\left[\Omega_{U_0,U_1}g_0^{-1}\Psi g_1\right]\right)=\tau\Bigl(\Tr\left[\left(g_1|_{U_1}\right)\Omega_{U_0,U_1}\left(g_0^{-1}|_{V_0/g_0U_0}\right)\Psi\right]\Bigr)\\
& = \tau\left(\Tr\left[\left(g_1\Omega g_0^{-1}\right)_{g_0U_0,g_1U_1}\,\Psi\right]\right)=\chi_{g_1\Omega g_0^{-1}}(\Psi),\qquad(\text{by }\ref{defOmegaV0U0})
\end{split}
\end{equation}
where $g_1\Omega g_0^{-1}\in D_{(g_0U_0,g_1 U_1)}$, $\Psi\in B_{(g_1 U_1,g_0U_0)}$.

Fix $y_0=(W_0,W_1)\in Y$, set $B \coloneqq B_{y_0}$, and take $\Omega\in D\coloneqq D_{y_0}$, $\theta_\Omega\coloneqq\chi_\Omega$, 
$\pi \coloneqq \pi_{y_0}$. Then \eqref{definfl} becomes, for all $T\in A$, 
\[
\begin{split}
\overset{\triangle}{\theta_\Omega}(T)& = \theta_\Omega(\pi(T))\\
(\text{by }\eqref{defpiynbqJ})\quad& = \theta_\Omega\left(\pi_{W_0}T\iota_{W_1}\right)\\
(\text{by }\eqref{defchiOmega})\quad& = \tau\left(\Tr\left[\Omega_{W_0,W_1}\left(\pi_{W_0}T\iota_{W_1}\right)\right]\right)\\
& = \tau\left(\Tr\left[\left(\iota_{W_1}\Omega_{W_0,W_1}\pi_{W_0}\right)T\right]\right)=\psi_{\iota_{W_1}\Omega_{W_0,W_1}\pi_{W_0}}(T)\\
(\text{by }\eqref{OmegaU0U1})\quad& = \psi_\Omega(T),
\end{split}
\]
that is, $\overset{\triangle}{\theta_\Omega}=\psi_\Omega$. From \eqref{defhtranslpsi} it follows that $^{(g_0,g_1)}\!\left(\overset{\triangle}{\theta_\Omega}\right)=\overset{\triangle}{\theta_\Omega}$ if and only if 
$(g_0,g_1)\Omega=\Omega$, and therefore $H_{\theta_\Omega}\equiv H_\Omega$ (cf. \eqref{HKchi} and Lemma \ref{LemmaHS}). Moreover, the orbits of $K\coloneqq K_{(W_0,W_1)}$ on $\widehat{B}\equiv\{\theta_\Omega\colon\Omega\in D\}$ are parameterized by the rank of the operators in $D$, and a complete system of representatives is $\Theta=\{\theta_\ell\colon \ell=0,1,\dotsc,\min\{m-r,s\}\}$, where $\theta_\ell\coloneqq \theta_{\Omega_\ell}$; cf.\ \eqref{ThetaHom}.
As indicated in Remark \ref{defGammatheta}, $H_{\Omega_\ell}$ is {\em not} the stabilizer of $\theta_\ell\equiv \chi_{\Omega_\ell}$ (which is $K\cap H_{\Omega_\ell}$; see also \eqref{actchiOmega}).

We do not need an explicit description of $K_{\Omega_\ell}\coloneqq K\cap H_{\Omega_\ell}$ because we may use Corollary \ref{Cor2qJohn} to study the corresponding homogeneous space. We denote by $\mathcal{G}_n^q$ the lattice of all subspaces of $\mathbb{F}_q^n$ 
(identified with $\mathcal{H}_n^q$; cf.\ \eqref{e:H-n} and Proposition \ref{PropHn}) and by $\mathcal{G}_{n,m}^q$ the Grassmannian of all 
$m$-dimensional subspaces.

In the following we recover Corollary \ref{lastCor} by using the representation theory we developed in Section 4.

\begin{proposition}
\label{PropCorsat}
With the above notation, the following holds. 
\begin{enumerate}[{\rm (1)}]
\item\label{PropCorsat1} 
As homogeneous spaces,
\[
\begin{split}
H_{\Omega_\ell}/K_{\Omega_\ell} & \cong \GL(\mathbb{F}_q^{m-\ell})/(H_{r,m-\ell-r}\ltimes A_{m-\ell-r,r}) \times
\GL(\mathbb{F}_q^{n-m-\ell})/(H_{s-\ell,n-m-s}\ltimes A_{n-m-s,s-\ell})\\
& \equiv \mathcal{G}_{m-\ell,r}^q\times\mathcal{G}_{n-m-\ell,s-\ell}^q.
\end{split}
\]
\item The conditions in Corollary \ref{Cormultfree} are fulfilled. 
\end{enumerate}
\end{proposition}
\begin{proof}
(1) This is just a particular case of Corollary \ref{Cor2qJohn} (see also the comment preceding it).
\par
(2)  By the characterizations of the of $H$- and $K$-orbit of $\theta_\Omega$ by means of the rank, it follows that each $H$-orbit contains exactly one $K$-orbit, so that $\Xi=\Theta$. This shows that condition (i) in Corollary \ref{Cormultfree} is fulfilled.
From \eqref{PropCorsat1} it follows that $(H_{\Omega_\ell}, K_{\Omega_\ell})$, being $H_{\Omega_\ell}/K_{\Omega_\ell}$ equivalent 
(as a homogeneous space) to the direct product of two $q$-Johnson schemes (cf.\ \eqref{decqHahn} below), is a Gelfand pair. 
This shows that condition (ii) in Corollary \ref{Cormultfree} is fulfilled as well.
\end{proof}

\begin{corollary}
\label{c:gelfand-pair}
Let $0 \leq m \leq n$. For all $0 \leq r \leq m$ and $0 \leq s \leq n-m$ set
\[
H \coloneqq \left(\GL(\mathbb{F}_q^m) \times \GL(\mathbb{F}_q^{n-m})\right) \mbox{ \ and \ } A \coloneqq \Hom(\mathbb{F}_q^{n-m},\mathbb{F}_q^m),
\]
\begin{multline*}
K \coloneqq \left(\left(\GL(\mathbb{F}_q^r) \times \GL(\mathbb{F}_q^{m-r})\right) \ltimes \Hom(\mathbb{F}_q^{m-r}, \mathbb{F}_q^r)\right) \\ \times
\left(\left(\GL(\mathbb{F}_q^s) \times \GL(\mathbb{F}_q^{n-m-s})\right) \ltimes \Hom(\mathbb{F}_q^{n-m-s}, \mathbb{F}_q^s)\right), 
\end{multline*}
and
\[
C \coloneqq \Hom(\mathbb{F}_q^s, \mathbb{F}_q^r) \oplus \Hom(\mathbb{F}_q^{n-m-s}, \mathbb{F}_q^m).
\]
Then $(H \ltimes A,K \ltimes C)$ is a symmetric Gelfand pair.
\end{corollary}
\begin{proof}
This follows immediately from Proposition \ref{PropCorsat} and Corollary \ref{Cormultfree}.
\end{proof}

\begin{example}[The $q$-analog of the Johnson scheme]
\label{ex:q-analog-J} {\rm In the notation of Corollary \ref{c:gelfand-pair}, for $m = 0$, and therefore $r = 0$, we have 
\begin{itemize}
\item $H = \GL(\mathbb{F}_q^n)$,
\item $A = \{0\}$,
\item $K = \left(\GL(\mathbb{F}_q^s) \times \GL(\mathbb{F}_q^{n-s})\right) \ltimes \Hom(\mathbb{F}_q^{n-s}, \mathbb{F}_q^s)$,
\item $C = \{0\}$.
\end{itemize}
Thus, $H \ltimes A = H$, $K \ltimes C = K$, and therefore
\[
(H \ltimes A, K \ltimes C) = (H,K) = \left(\GL(\mathbb{F}_q^n), (\GL(\mathbb{F}_q^s) \times \GL(\mathbb{F}_q^{n-s})) \ltimes \Hom(\mathbb{F}_q^{n-s}, \mathbb{F}_q^s)\right)
\]
and we recover the $q$-analog of the Johnson scheme (cf.\ \cite[Theorem 8.6.5]{book}).
\par
Note that, in this case, $V_0 = \{0\}$ and $V_1 = V$, and therefore $U_0 = \{0\}$ and $\Hom(U_1,V_0/U_0) = \{0\}$ for all $U_1 \subset V$,
so that the corresponding homogeneous space is given by $X = {\mathcal G}^q_{n,s}$, the Grassmannian of all $s$-dimensional vector subspaces of the $n$-dimensional vector space $\mathbb{F}_q^n$.
}
\end{example}

\begin{example}[The $q$-analog of the Hamming scheme]
\label{ex:q-analog-H}{\rm 
In the notation of Corollary \ref{c:gelfand-pair}, for $m > 0$, $r=0$, and $s = n-m$ we have 
\begin{itemize}
\item $K = \GL(\mathbb{F}_q^{m}) \times \GL(\mathbb{F}_q^{n-m})$,
\item  $C = \{0\}$. 
\end{itemize}
Thus, $K \ltimes C = K$, and therefore
\[
(H \ltimes A, K \ltimes C) = \left((\GL(\mathbb{F}_q^m) \times \GL(\mathbb{F}_q^{n-m})) \ltimes \Hom(\mathbb{F}_q^{n-m}, \mathbb{F}_q^m), \GL(\mathbb{F}_q^{m}) \times \GL(\mathbb{F}_q^{n-m})\right)
\]
and we recover the $q$-analog of the Hamming scheme (cf.\ \cite[Theorem 8.7.9]{book}).
\par
Note that, in this case, $U_0 = \{0\}$, $V_0 = \mathbb{F}_q^m$, and $U_1 = V_1 = \mathbb{F}_q^{n-m}$,
therefore $\Hom(U_1,V_0/U_0) = \Hom(U_1,V_0) = \Hom(\mathbb{F}_q^{n-m}, \mathbb{F}_q^m)$
so that the corresponding homogeneous space is given by $X = \Hom(\mathbb{F}_q^{n-m},  \mathbb{F}_q^m)$.
Note that by virtue of Proposition \ref{PropGH},  $X$ can be identified with the space of all $(n-m)$-dimensional subspaces
of $\mathbb{F}_q^n$ having a trivial intersection with $\mathbb{F}_q^m$.
In the setting of Association Schemes, this is called the \emph{Bilinear Forms Scheme} (cf.\ \cite{bannai2}).
}
\end{example}

\begin{example}[The Attenuated Space]
\label{ex:attenuated}{\rm 
In the notation of Corollary \ref{c:gelfand-pair}, for $m > 0$ and $r=0$ we have $U_0 = \{0\}$, $V_0 = \mathbb{F}_q^m$, and 
$U_1$ is an $s$-dimensional vector subspace of $\mathbb{F}_q^{n-m}$,
therefore $\Hom(U_1,V_0/U_0) = \Hom(U_1,V_0) = \Hom(U_1, \mathbb{F}_q^m)$, so that 
the corresponding homogeneous space $X$ can be identified with the space of all $s$-dimensional subspaces
of $\mathbb{F}_q^n$ having a trivial intersection with $\mathbb{F}_q^m$.
In the terminology of Association Schemes, this is called the \emph{Attenuated Space} (cf.\ \cite{bannai2}. See also
\cite{WJL}). Note that if $s = n-m$, that is, $s$ is maximal, we recover the Bilinear Forms Scheme from Example \ref{ex:q-analog-H}.
The character table (i.e., the spherical functions) for the Attenuated Space was explicitly calculated in \cite{Kurihara}.
}
\end{example}

We now proceed with the harmonic analysis of the Gelfand pair in Corollary \ref{c:gelfand-pair} by decomposing the
corresponding permutation representation and giving an explicit expression for the associated spherical functions.
\par
We keep the notation from Section \ref{SecIntBasis} (cf., in particular, Remarks \ref{Remchitheta0} and \ref{Remchitheta}).
\par
Let $y = (U_0,U_1) \in Y$ and $0 \leq \ell \leq \min\{m-r,s\}$. 
Recall that $\widehat{B_{y}}_{,\theta_\ell}=\{\chi \in \widehat{B_{y}}: \, ^{h(\chi)}\chi = \theta_\ell\}$ (resp.\
$D_y^\ell = \{\Omega \in \Hom(V_0,V_1): \Ker \Omega \geq V_0, \Im \Omega \leq U_1, \rk(\Omega) = \ell\}$) and that
the map $\Omega \mapsto \chi_\Omega$ defined by \eqref{defchiOmega} yields a
bijection from $D_y = \bigsqcup_\ell D_y^\ell$ onto $\widehat{B_{y}} = \bigsqcup_\ell \widehat{B_{y}}_{,\theta_\ell}$. 
Since this bijection is $H$-equivariant (cf.\ \eqref{actchiOmega}), we have
$\widehat{B_{y}}_{,\theta_\ell}=\left\{\chi_\Omega\colon\Omega\in D_{y}^\ell\right\}$ for all $\ell$.
Moreover, condition \eqref{chiCinv} becomes
\begin{align*}
\text{for all }T\in C, \quad 1& = \chi_\Omega\left(\pi_{U_0}T\iota_{U_1}\right)&\text{by }\eqref{defpiynbqJ}\\
& = \tau\left(\Tr\left[\Omega_{U_0,U_1}\pi_{U_0}T\iota_{U_1}\right]\right)&\text{by }\eqref{defchiOmega}\\
& = \tau\left(\Tr\left[\Omega T\right]\right)&\text{by }\eqref{OmegaU0U1}
\end{align*}
and therefore, by Lemma \ref{LemmaOmegaT}, we have 
\[
\widehat{B_y}^0=\{\chi_\Omega\colon\Omega\in D_{(W_0+U_0,W_1\cap U_1)}\} \ \mbox{ and } \ 
\widehat{B_{y}}_{,\theta_\ell}^0=\{\chi_\Omega\colon\Omega\in D_{(W_0+U_0,W_1\cap U_1)}^\ell\}.
\]

From \eqref{actchiOmega} and \eqref{defGammaprime} it also follows that the complete set $\Gamma$ of representatives 
for the orbits of $K$ on $\bigsqcup_{y \in Y}\widehat{B_y}^0$ (cf.\ \eqref{e:Gamma-rep} and \eqref{defGamma}) is given by 
\[
\Gamma=\left\{(z_{h,i},\chi_{h,i}^\ell)\colon h,i \text{ as in }\eqref{hilvar2}\text{ and } \ell \text{ as in }\eqref{hilvar}\right\},
\]
where $\chi_{h,i}^\ell\coloneqq \chi_{\Omega_{h,i}^\ell}$.

Moreover, from Lemma \ref{lemmaHKorbits} it follows that, for $\ell=0,1,\dotsc,\min\{m-r,s\}$, $\Gamma_\theta$ in Remark \ref{defGammatheta} is given by:
\[
\Gamma_{\theta_\ell}=\bigl\{(z_{h,i},\chi_{h,i}^\ell)\colon \max\{2r-m+\ell,0\}\leq h\leq r,\, \max\{2s-n+m,\ell\}\leq i\leq s\bigr\}.
\]

It is well known that for $T\in \Hom(V_1,V_0)$ with $\rk(T)=x$ one has
\begin{equation}\label{qKraw}
\sum_{S\in \Hom_\ell(V_0,V_1)}\tau(\Tr[ST])=K_\ell(x;n-m,m;q),
\end{equation}
where $K_\ell(x;n-m,m;q)$ is a $q$-Krawtchouk polynomial. See \cite[Theorem 4.5]{DUNKL} and the reference therein. 
Actually, \eqref{qKraw} could also be deduced from the calculations in \cite[Section 8.7]{book}, in particular from Exercise 8.7.15 
which does not use the $q$-Krawtchouk notation.

\begin{proposition}\label{PropLambda}
In the present setting, the intermediate basis in \eqref{defPsizchi} has the following expression: the element of the basis corresponding to $\chi_{h,i}^\ell$ evaluated at the representatives in \eqref{rephijvar} is equal to
\begin{equation}\label{LambdaKraw}
\Lambda(h,i,\ell;\Psi_{h,i,j})=\frac{1}{\binom{m-2r+h}{\ell}_q\prod_{t=0}^{\ell-1}(q^i-q^t)}K_\ell(i-j;i,m-2r+h;q),
\end{equation}
$j=0,1,\dotsc,i$, and it is zero on $\Psi_{h',i',j}$ if $(h',i')\neq (h,i)$.
\end{proposition}
\begin{proof}
We have (cf. \eqref{calcPsizchi}):
\begin{multline*}
\sum_{k\in K} \,^k\chi_{h,i}^\ell(\Psi_{h,i,j})=\sum_{k=(k_0,k_1)\in K}
\tau\left(\Tr\left[\left(k_1\Omega_{h,i}^\ell k_0^{-1}\right)_{k_0U_{0h},k_1U_{1i}}\,\Psi_{h,i,j}\right]\right)\quad(\text{by }\eqref{actchiOmega})\\
=\sum_{k=(k_0,k_1)\in K_{z_{h,i}}}
\tau\left(\Tr\left[\left(k_1\Omega_{h,i}^\ell k_0^{-1}\right)_{W_0+U_{0h},W_1\cap U_{1i}}\left(\pi_{U_{0h},W_0}\Psi_{h,i,j}|_{W_1\cap U_{1i}}\right)\right]\right)\quad(\text{by }\eqref{TrOmPsi}).\\
\end{multline*}
From \eqref{kOmegax} we deduce that, in the last sum, $\left(k_1\Omega_{h,i}^\ell k_0^{-1}\right)_{W_0+U_{0h},W_1\cap U_{1i}}$ runs over all operators in $\Hom(V_0/(W_0+U_{0h}),W_1\cap U_{1i})$ of rank $\ell$, so that the number of times that every operator appears in the sum is
\begin{equation}\label{KzKchi}
\left\lvert K_{z_{h,i}}\right\rvert/\left\lvert\Hom_\ell\left(\mathbb{F}_q^{m-2r+h},\mathbb{F}_q^i\right)\right\rvert.
\end{equation}
The denominator in \eqref{LambdaKraw} coincides with the denominator in \eqref{KzKchi} (cf. \cite[Exercise 8.7.12]{book} or \cite[Proposition 4.4]{DUNKL}), while the factor $\left\lvert K_{z_{h,i}}\right\rvert$ in \eqref{calcPsizchi} simplifies. Since
\[
\begin{split}
\rk\left(\pi_{U_{0h},W_0}\Psi_{h,i,j}|_{W_1\cap U_{1i}}\right)&=\dim\left(W_1\cap U_{1i}\right)-\dim\Ker\left(\pi_{U_{0h},W_0}\Psi_{h,i,j}|_{W_1\cap U_{1i}}\right)\\
(\text{by }\eqref{starp40})\quad&=i-j,
\end{split}
\]
applying \eqref{qKraw} we end the proof of \eqref{LambdaKraw}. If $(h',i')\neq (h,i)$ then $z_{h',i'}$ is not in the $K$-orbit of $z_{h,i}$ and therefore $\Lambda(h,i,\ell;\Psi_{h',i',j})=0$.
\end{proof}

We now recall some basic facts on the $q$-Johnson scheme. In our notation, it is the Gelfand pair $(\GL(\mathbb{F}_q^n),H_{m,n-m}\ltimes A_{n-m,m})$ (cf.\ Example \ref{ex:q-analog-J} with $m = s$). Note that these results are usually stated under the assumption $m\leq n-m$, but they may easily extended to the case $m>n-m$ by means of the isomorphism $U\mapsto U^\circ$ in \eqref{ggtVprime}; see also \cite[Exercise 8.6.7]{book}. The corresponding homogeneous space is $\mathcal{G}^q_{n,m}$; cf.\ \eqref{squarp3tris} and the subsequent discussion. We denote by $V_0$ the subspace stabilized by $H_{m,n-m}\ltimes A_{n-m,m}$. The decomposition of $L(\mathcal{G}^q_{n,m})$ into spherical harmonics is given by
\begin{equation}\label{decqHahn}
L(\mathcal{G}^q_{n,m})=\bigoplus_{k=0}^{\min\{m,n-m\}}\mathcal{V}_{n,m,k}
\end{equation}
where
\begin{equation}\label{decqHahnbis}
\mathcal{V}_{n,m,k}=
\begin{cases}
(d^*)^{m-k}\left(L(\mathcal{G}^q_{n,k})\cap\Ker d\right)\quad\text{ if }m\leq n-m,\\
d^{n-m-k}\left(L(\mathcal{G}^q_{n,n-k})\cap\Ker (d^*)\right)\quad\text{ if }m> n-m,
\end{cases}
\end{equation}
and 
\[
d\colon\bigoplus_{k=1}^{n}L\left(\mathcal{G}^q_{n,k}\right)\longrightarrow \bigoplus_{k=0}^{n-1}L\left(\mathcal{G}^q_{n,k}\right),\qquad\qquad
d^*\colon\bigoplus_{k=0}^{n-1}L\left(\mathcal{G}^q_{n,k}\right)\longrightarrow \bigoplus_{k=1}^{n}L\left(\mathcal{G}^q_{n,k}\right)
\]
are the \emph{lowering} and \emph{raising operators}, respectively, defined as follows. 
For $f_1\in L\left(\mathcal{G}^q_{n,k+1}\right)$ and $f_2\in L\left(\mathcal{G}^q_{n,k-1}\right)$ one defines $df_1\in L\left(\mathcal{G}^q_{n,k}\right)$ and $d^*f_2\in L\left(\mathcal{G}^q_{n,k}\right)$ by setting, for all $U\in\mathcal{G}^q_{n,k}$,
\[
(df_1)(U)=\sum_{\substack{X\in\mathcal{G}^q_{n,k+1}\colon X\geq U\\ }}f_1(X),\qquad (d^*f_2)(U)=\sum_{\substack{Y\in\mathcal{G}^q_{n,k-1}\colon Y\leq U\\ }}f_2(Y).
\]
The spherical function $\varphi_{n,m,k}\in\mathcal{V}_{n,m,k}$ is given by:
\begin{equation}\label{qHahnsph}
\varphi_{n,m,k}(X)=
(-1)^k\frac{q^{-mk}}{(q^{n-m})_k(q^m)_k}\cdot E_k(m,n-m,m,\dim(X\cap V_0);q^{-1}),\qquad X\in\mathcal{G}_{n,m},
\end{equation}
where $E_k$ is a multiple of a $q$-Hahn polynomial; we refer to \cite[Theorem 3.7]{DUNKL} for the expression of the $E_k$'s in terms of the classical $Q_k$ notation for the $q$-Hahn polynomials. In \cite[Theorem 8.6.5]{book} we gave a direct expression for these spherical functions avoiding the terminology of $q$-orthogonal polynomials in order to keep the book at a reasonable length. Note that the expression of the spherical function in the case $m>n-m$, obtained by means of the isomorphism $U\mapsto U^\circ$ in \eqref{ggtVprime}, should be 
\[
\varphi_{n,m,k}(X)=
(-1)^k\frac{q^{-(n-m)k}}{(q^m)_k(q^{n-m})_k}\cdot E_k(n-m,m,n-m,n-\dim(X+V_0);q^{-1})
\]
but this is equal to \eqref{qHahnsph} since
\[
E_k(n-m,m,n-m,n-\dim(X+V_0);q^{-1})=q^{-k(2m-n)}E_k(m,n-m,m,2m-\dim(X+V_0);q^{-1})
\]
(\cite[Proposition 2.3(ii)]{DUNKL}) and $2m-\dim(X+V_0)=\dim(X\cap V_0)$. 

From \eqref{decqHahn} it follows that, for each $\theta_\ell$, $\ell=0,1,\dotsc,\min\{s,m-r\}$,   
\begin{equation}\label{LGGdec}
L\left(\mathcal{G}_{m-\ell,r}^q\times\mathcal{G}_{n-m-\ell,s-\ell}^q\right)\sim\bigoplus_{u=0}^{\min\{r,m-\ell-r\}}\bigoplus_{v=0}^{\min\{s-\ell,n-m-s\}}\left(\mathcal{V}_{m-\ell,r,u}\otimes\mathcal{V}_{n-m-\ell,s-\ell,v}\right)
\end{equation}
is the decomposition, into irreducible representations, of the permutation representation associated with the action described in Proposition \ref{PropCorsat}.\eqref{PropCorsat1};
we denote by $\sigma_{u,v,\ell}$ the irreducible representation of $H_{\Omega_\ell}$ over $\mathcal{V}_{m-\ell,r,u}\otimes\mathcal{V}_{n-m-\ell,s-\ell,v}$.

\begin{theorem}[Spherical functions of the $q$-analog of the nonbinary Johnson scheme]
\label{theodecirred}
With the above notation, the following holds.
\begin{enumerate}[{\rm (1)}]
\item
The decomposition of $L(X)$ into irreducible $(H\ltimes A)$-representations is:
\begin{equation}\label{DSPHA}
L(X)\sim\bigoplus_{\ell=0}^{\min \{m-r,s\}}\bigoplus_{u=0}^{\min\{r,m-\ell-r\}}\bigoplus_{v=0}^{\min\{s-\ell,n-m-s\}}\Ind_{H_{\Omega_\ell}\ltimes A}^{H\ltimes A}\left(\widetilde{\overline{\theta_\ell}}\otimes\overset{\triangle}{\sigma_{u,v,\ell}}\right).
\end{equation}
\item\label{theodecirred2}
\[
\begin{split}
&\dim\Ind_{H_{\Omega_\ell}\ltimes A}^{H\ltimes A}\left(\widetilde{\overline{\theta_\ell}}\otimes\overset{\triangle}{\sigma_{u,v,\ell}}\right)
=\binom{m}{\ell}_q\prod_{j=0}^{\ell-1}(q^{n-m}-q^j)\cdot\\
&\qquad\cdot\left(\binom{m-\ell}{u}_q -\binom{m-\ell}{u\pm 1}_q\right)\left(\binom{n-m-\ell}{v}_q -\binom{n-m-\ell}{v\pm 1}_q\right),
\end{split}
\]
where we use $u-1$ if $u\leq (m-\ell)/2$, and $u+1$ otherwise, and similarly for $v\pm 1$.
\item 
The spherical function associated with the parameters $(u,v,\ell)$ in \eqref{DSPHA} evaluated at the representative 
$\left(U_{0h},U_{1i},\Psi_{hij}\right)$ in \eqref{rephijvar} is equal to:
\begin{equation}\label{EEK}
\begin{split}
\Phi_{u,v,\ell}(h,i,j)\coloneqq&\frac{(-1)^{u+v}q^{-ru-(s-\ell)v}}{(q^{m-\ell-r})_u(q^r)_u\cdot 
(q^{n-m-s})_v(q^{s-\ell})_v\cdot\binom{m-r}{\ell}_q\prod_{t=0}^{\ell-1}(q^s-q^t)}\cdot\\
&\cdot E_u(r,m-\ell-r,r,h;q^{-1})\cdot E_v(s-\ell,n-m-s,s-\ell,i-\ell;q^{-1})\cdot\\
&\cdot K_\ell(i-j;i,m-2r+h;q),
\end{split}
\end{equation}
for $\max\{2r-m+\ell,0\}\leq h\leq r,\, \max\{2s-n+m,\ell\}\leq i\leq s$, and $0\leq j\leq i$.
\end{enumerate}
\end{theorem}
\begin{proof}
(1) This follows immediately from Corollary \ref{Cormultfree} and \eqref{LGGdec}.
\par
(2) First of all, taking into account that the orbits of $H$ on $\Hom(V_0,V_1)$ are parameterized by the rank (cf.\ \eqref{XiHom}) and then using \cite[Exercise 8.7.12]{book} or \cite[Proposition 4.4]{DUNKL}, we get:
\[
\left\lvert H/H_{\Omega_\ell}\right\rvert=\binom{m}{\ell}_{\!\!\! q} \prod_{j=0}^{\ell-1}(q^{n-m}-q^j).
\]
Then we may recall that (cf.\ \cite[Theorem 8.6.5]{book} and the isomorphism $U\mapsto U^\circ$ in \eqref{ggtVprime})
\[
\dim\sigma_{u,v,\ell}=\left(\binom{m-\ell}{u}_{\!\!\! q}-\binom{m-\ell}{u\pm 1}_{\!\!\! q}\right)\left(\binom{n-m-\ell}{v}_{\!\!\! q} -\binom{n-m-\ell}{v\pm 1}_{\!\!\! q}\right)
\]
and apply \eqref{dimind}. 
\par
(3) The bounds on $h$ and $i$ follow from \eqref{hilvar2} and Lemma \ref{lemmaHKorbits}, which are equivalent to the conditions in Lemma \ref{LemmaGammatheta}. In particular, as in Remark \ref{defGammatheta}, $h_0(\chi)$ is an element $(g_0,g_1)\in H_{\Omega_\ell}$ such that, setting $(g_0,g_1)z_{r,s}\eqqcolon (U_0,U_1)$ and $^{(g_0,g_1)}\,\theta_\ell\eqqcolon\chi_\Omega$, the pair $\bigl((U_0,U_1),\chi_\Omega\bigr)$ is in the $K$-orbit of $(z_{h,i},\chi_{h,i}^\ell)$: there exists $k=(k_0,k_1)\in K$ (the $k(\chi)$ in Remark \ref{defGammatheta}) such that $k\bigl((U_0,U_1),\chi_\Omega\bigr)=(z_{h,i},\chi_{h,i}^\ell)$. Clearly, $z_{r,s}=(W_0,W_1)$ is the point stabilized by $K$, $\dim(W_0\cap U_0)=h$, and $\dim(W_1\cap U_1)=i$. 
Moreover, in the notation of Proposition \ref{Prop2qJohn},
\begin{equation}\label{Thg0g1}
\mathcal{L}_{(g_0,g_1)}\left(W_0,W_1/\Im\Omega_\ell\right)=\left(U_0,U_1/\Im\Omega_\ell\right),
\end{equation}
as $U_1=g_1W_1$ and $(g_0,g_1)\in H_{\Omega_\ell}\leq K_{(\Ker\Omega_\ell,\Im\Omega_\ell)}$ (cf.\ Lemma \ref{LemmaHS}) imply that 
$\Im\Omega_\ell\leq U_1$. Then
\begin{equation}\label{dimWUOm}
\begin{split}
\dim\left[\left(W_1/\Im\Omega_\ell\right)\cap\left(U_1/\Im\Omega_\ell\right)\right]& = \dim\Bigl((k_0,k_1)\left[\left(W_1\cap U_1\right)/\Im\Omega_\ell\right]\Bigr)\\
& = \dim\left[\left(W_1\cap U_{1i}\right)/\Im\left(k_1\Omega_\ell k_0^{-1}\right)\right]\\
& = i-\ell.\\
\end{split}
\end{equation}
From \eqref{sphfunct}, Proposition \ref{PropLambda}, and \eqref{Thg0g1} we then get a preliminary expression for the spherical function evaluated at $\left(U_{0h},U_{1i},\Psi_{hij}\right)$:
\begin{equation}\label{SPHERFUNC}
\begin{split}
\frac{\left\lvert K_{z_{h,i}}/K_{\Omega_{h,i}^\ell}\right\rvert}{\left\lvert K/K_{\Omega_\ell}\right\rvert}\varphi_{m-\ell,r,u}(U_{0h})\varphi_{n-m-\ell,s-\ell,v}(U_{1i}/\Im\Omega_\ell)\Lambda(h,i,\ell;\Psi_{hij}).
\end{split}
\end{equation}
Then observe that $\left\lvert K_{z_{h,i}}/K_{\Omega_{h,i}^\ell}\right\rvert$ is just the denominator in \eqref{LambdaKraw} (cf.\ also \eqref{KzKchi}), so that they simplify each other, while $\left\lvert K/K_{\Omega_\ell}\right\rvert$ is just a particular case of these calculations: it is equal to the number of operators in $D_{(W_0,W_1)}$ of rank $\ell$, that is to 
$\left\lvert\Hom_\ell\left(\mathbb{F}_q^{m-r},\mathbb{F}_q^s\right)\right\rvert$, which in turn is equal to $\binom{m-r}{\ell}_q\prod_{t=0}^{\ell-1}(q^s-q^t)$. In conclusion, by means of elementary calculations, and taking \eqref{dimWUOm} into account, one easily shows that \eqref{SPHERFUNC} is equal to \eqref{EEK}.
\end{proof}

\begin{remark}
{\rm
The expression in \eqref{EEK} coincides with that in \cite[Theorem 5.6]{DUNKL}, up to the normalization constant. 
Indeed, the parameters $i,j,r,a,b,m,n,v,w,z$ in Dunkl's paper
correspond respectively to $r,s,\ell,m,n-m,u,v,h,i,i-j$ in the present paper. 
Note also that Dunkl studied $H_{ij}-H_{xu}$ invariant functions and the spherical functions correspond to the case $x=i$ and $u=j$ in his paper.
}
\end{remark}

\begin{theorem}[Orthogonality relations for the spherical functions of the $q$-analog of the nonbinary Johnson scheme]
The spherical functions satisfy the following orthogonality relations $($cf.\ \eqref{orthrel}$)$: for $0\leq \ell,\ell'\leq\min\{m-r,s\}$, 
$0\leq u,u'\leq\min\{r,m-\ell-r\}$, and $0\leq v,v,'\leq\min\{s-\ell,n-m-s\}$ we have
\[
\sum_{h=\max\{2r-m+\ell,0\}}^r\sum_{i=\max\{2s-n+m,\ell\}}^s\sum_{j=0}^i\Phi_{u,v,\ell}(h,i,j)\Phi_{u',v',\ell'}(h,i,j)\gamma(h,i,j)=\delta_{u,u'}\delta_{v,v'}\delta_{\ell,\ell'}\beta(u,v,\ell)
\]
and 
\[
\begin{split}
\gamma(h,i,j)\coloneqq
&{\binom{r}{h}_{\!\!\! q}\binom{m-r}{r-h}_{\!\!\! q} q^{(r-h)^2}\cdot\binom{s}{i}_{\!\!\! q}\binom{n-m-s}{s-i}_{\!\!\! q}}q^{(s-i)^2}\cdot q^{i(r-h)}\cdot q^{(s-i)(m-r)}\cdot\\
&\cdot \binom{i}{j}_{\!\!\! q}\;\prod_{t=0}^{j-1}(q^{m-2r+h}-q^t).
\end{split}
\]
is the cardinality of the $(K\ltimes C)$-orbit containing $(U_{0h},U_{1i},\Psi_{hij})$, and
\[
\beta(u,v,\ell)\coloneqq \frac{\binom{m}{r}_{\! q}\binom{n-m}{s}_{\! q} q^{s(m-r)}}{\binom{m}{\ell}_{\! q}\prod_{j=0}^{\ell-1}(q^{n-m}-q^j)\left(\binom{m-\ell}{u}_{\! q} -\binom{m-\ell}{u-1}_{\! q}\right)\left(\binom{n-m-\ell}{v}_{\! q} -\binom{n-m-\ell}{v-1}_{\! q}\right)}
\] 
is the square of the norm of $\Phi_{u,v,\ell}$.
\end{theorem}
\begin{proof}
Suppose that $(U_0,U_1)$ belongs to the $K$-orbit of $(U_{0h},U_{1i})$ in $Y_{r,s}$. In order to compute the number of $\Psi\in\Hom(U_1,V_0/U_0)$ such that $(U_0,U_1,\Psi)$ is in the $(K\ltimes C)$-orbit of $(U_{0h},U_{1i},\Psi_{hij})$, we use the following decompositions:
\[
\frac{V_0}{U_0}=\frac{X_0}{U_0}\oplus \frac{X_1}{U_0} \ \ \mbox{ and } \ \  U_1=Z_0\oplus Z_1,
\]
where $X_0\coloneqq W_0+U_0$, $Z_0\coloneqq W_1\cap U_1$, and ${X_1}/{U_0}$ and $Z_1$ are complementary subspaces. Then $\Psi=\Psi_{00}+\Psi_{01}+\Psi_{10}+\Psi_{11}$, where $\Psi_{ij}\in\Hom\left(Z_i,X_j/U_0\right)$ and the condition $\dim\Ker\left(\pi_{U_0,W_0}\Psi|_{W_1\cap U_1}\right)=j$ (cf. \eqref{starp40}) is satisfied if and only if $\rk(\Psi_{01})=j$, while $\Psi_{00},\Psi_{10}$ and $\Psi_{11}$ are arbitrary. The number of $\Psi_{01}$'s may be computed by means of \cite[Exercise 8.7.12]{book} or \cite[Proposition 4.4]{DUNKL} and it is equal to
\[
\binom{i}{j}_{\!\!\! q}\;\prod_{t=0}^{j-1}(q^{m-2r+h}-q^t)
\] 
because $\dim(W_0+U_0)=2r-h$ implies that $\dim(X_1/U_0)=\dim(V_0/U_0)-\dim\left[(W_0+U_0)/U_0\right]=m-2r+h$. Moreover, note that the number of $(U_0,U_1)$'s is equal to $\lvert K/K_{z_{h,i}}\rvert$, the cardinality of the $K$-orbit on $Y_{r,s}$ containing 
$(U_{0h},U_{1i})$, and it may be computed by means of \cite[Corollary 8.5.5]{book} or \cite[Corollary 3.2]{DUNKL}:
\[
\lvert K/K_{z_{h,i}}\rvert={\binom{r}{h}_{\!\!\! q}\binom{m-r}{r-h}_{\!\!\! q} q^{(r-h)^2}\cdot\binom{s}{i}_{\!\!\! q}\binom{n-m-s}{s-i}_{\!\!\! q}}q^{(s-i)^2}.
\]
In conclusion, the cardinality of the $(K\ltimes C)$-orbit of $(U_{0h},U_{1i},\Psi_{hij})$ is equal to
\[
\lvert K/K_{z_{h,i}}\rvert\cdot \left\lvert\Hom\left(Z_0,X_0/U_0\right)\right\rvert\cdot\left\lvert\Hom\left(Z_1,X_0/U_0\right)\right\rvert\cdot\left\lvert\Hom\left(Z_1,X_1/U_0\right)\right\rvert
\cdot\left\lvert\Hom_j\left(Z_0,X_1/U_0\right)\right\rvert
\]
and one easily gets the expression of $\gamma(h,i,j)$ in the statement.

The square of the norm of $\Phi_{u,v,\ell}$ is equal to (cf.\ \eqref{normPhi} and \cite[Corollary 4.6.4.(iii)]{book})
\[
\left\lVert\Phi_{u,v,\ell} \right\rVert_{L(X)}^2=
\frac{\lvert X\rvert}{\dim\Ind_{H_{\Omega_\ell}\ltimes A}^{H\ltimes A}\left(\widetilde{\overline{\theta_\ell}}\otimes\overset{\triangle}{\sigma_{u,v,\ell}}\right)}.
\]
From \cite[Corollary 8.5.6]{book} we get:
\[
\lvert X\rvert=\lvert Y_{rs}\rvert\cdot \lvert B\rvert=\left\lvert \mathcal{G}^q_{m,r}\right\rvert\cdot \left\lvert \mathcal{G}^q_{n-m,s}\right\rvert\cdot\left\lvert \Hom(\mathbb{F}_q^s,\mathbb{F}_q^m/\mathbb{F}_q^r)\right\rvert=\binom{m}{r}_{\!\!\! q}\binom{n-m}{s}_{\!\!\! q}q^{s(m-r)}
\]
and one finally derives the expression of $\beta(u,v,\ell)$ by applying Theorem \ref{theodecirred}.\eqref{theodecirred2}. 
\end{proof}

In order to get an explicit expression for the spherical harmonics, we now apply \eqref{descsphharm} in the present setting. For $0\leq \ell \leq \min\{m-r,s\}$ and $S\in \Hom_\ell(V_0,V_1)$ we choose $t_S\in H$ such that $t_S\Omega_\ell=S$. This way, we get a system of representatives for the left cosets of $H_{\Omega_\ell}$ in $H$.  We then apply Corollary \ref{Cor2qJohn} and, for every $y=(Y_0,Y_1)\in Y$ with $\Omega_\ell\in D_y$, we choose $s_y\in H_{\Omega_\ell}$ such that: 
\begin{equation}\label{Lsy}
\mathcal{L}_{s_y}\left(W_0,W_1/\Im\Omega_\ell\right)=\left(Y_0,Y_1/\Im\Omega_\ell\right).
\end{equation}
We thus obtain a complete system of representatives for the left cosets of $K_{\Omega_\ell}$ in $H_{\Omega_\ell}$. Then, for all $0\leq u\leq\min\{r,m-\ell-r\}$ and $0\leq v\leq\min\{s-\ell,n-m-s\}$, we denote by $\mathcal{V}_{m-\ell,r,u}(\Ker S)\otimes\mathcal{V}_{n-m-\ell,s-\ell,v}(V_1/\Im S)$ the subspaces in \eqref{LGGdec} constructed as in \eqref{decqHahnbis} but using $\Ker S$ and $V_1/\Im S$ in place of $\Ker \Omega_\ell$ and $V_1/\Im\Omega_\ell$.

\begin{theorem}\label{PropSphharm}
For every $S\in \Hom_\ell(V_0,V_1)$ choose 
\begin{equation}\label{choicefS}
f_{0S}\otimes f_{1S}\in\mathcal{V}_{m-\ell,r,u}(\Ker S)\otimes\mathcal{V}_{n-m-\ell,s-\ell,v}(V_1/\Im S)
\end{equation}
and define $F\in L(X)$ by setting
\[
F(U_0,U_1,\Psi)=\sum_{\Omega\in D_{(U_0,U_1),\ell}}\chi_\Omega(\Psi)f_{0\Omega}\left(U_0\right) f_{1\Omega}\left(U_1/\Im\Omega\right),
\]
for all $(U_0,U_1,\Psi)\in X$.
Then the subspace made up of all these $F$'s corresponds to the image of the space of $\Ind_{H_{\Omega_l}\ltimes A}^{H\ltimes A}\left(\widetilde{\overline{\theta_\ell}}\otimes\overset{\triangle}{\sigma_{u,v,\ell}}\right)$ $($cf.\ \eqref{DSPHA}$)$ under the intertwining operator in Theorem \ref{Teospherfunct}.\eqref{Teospherfunct3}.
\end{theorem}
\begin{proof}
First of all, note that $t_S^{-1}\left(f_{0S}\otimes f_{1S}\right)$ corresponds to $T_{\theta,j}u_t$ in \eqref{descsphharm}. 
Indeed, if $(Y_0,Y_1)\in Y$ and $D_{(Y_0,Y_1)}$ contains $\Omega_\ell$, then
\[
\left[t_S^{-1}\left(f_{0S}\otimes f_{1S}\right)\right](Y_0,Y_1/\Im\Omega_\ell)=f_{0S}(t_SY_0)f_{1S}\left((t_SY_1)/\Im S\right),
\]
that is, $t_S^{-1}\left(f_{0S}\otimes f_{1S}\right)\in\mathcal{V}_{m-\ell,r,u}(\Ker \Omega_\ell)\otimes\mathcal{V}_{n-m-\ell,s-\ell,v}(V_1/\Im\Omega_\ell)$, as these spaces are defined just by means of the inclusion relation and the dimension. Moreover, the decompositions \eqref{decqHahn} and \eqref{LGGdec} are given by describing directly the image of the intertwining operators, that is, in the present setting, 
$T_{\theta,j}u_t$ is just an arbitrary element in $\mathcal{V}_{m-\ell,r,u}(\Ker \Omega_\ell)\otimes\mathcal{V}_{n-m-\ell,s-\ell,v}(V_1/\Im\Omega_l)$.

If $\Omega\in D_{(U_0,U_1),\ell}$, then $\Omega\in\Hom_l(V_0,V_1)$ and $t(\chi)$ in \eqref{hchith1} is just $t_\Omega$; cf.\ \eqref{actchiOmega}.
Since 
\[
t_\Omega^{-1}U_0\leq t_\Omega^{-1}\Ker\Omega=\Ker\Omega_\ell \quad\text{ and }\quad t_\Omega^{-1}U_1\geq t_\Omega^{-1}\Im\Omega=\Im\Omega_\ell,
\] 
from \eqref{Lsy} it follows that, in the present setting, the element $s(\chi)$ in \eqref{hchith1} is given by $s_{(t_\Omega^{-1}U_0,t_\Omega^{-1}U_1)}$.
Then, omitting the normalization factor $\frac{1}{\sqrt{\lvert B\rvert}}$, \eqref{descsphharm} yields the value at the point $(U_0,U_1,\Psi)\in X$ of the spherical harmonic corresponding to the choice in \eqref{choicefS}: 
\[
\begin{split}
&\sum_{\Omega\in D_{(U_0,U_1)}^\ell}\chi_\Omega(\Psi)\left[t_S^{-1}\left(f_{0\Omega}\otimes f_{1\Omega}\right)\right]\left(\mathcal{L}_{s_{(t_\Omega^{-1}U_0,t_\Omega^{-1}U_1)}}(W_0,W_1/\Im\Omega_\ell)\right)\\
(\text{by }\eqref{Lsy})\quad&\qquad=\sum_{\Omega\in D_{(U_0,U_1),\ell}}\chi_\Omega(\Psi)\left[t_\Omega^{-1}\left(f_{0\Omega}\otimes f_{1\Omega}\right)\right]\left(t_\Omega^{-1}U_0,(t_\Omega^{-1}U_1)/\Im\Omega_l\right)\\
&\qquad=\sum_{\Omega\in D_{(U_0,U_1),\ell}}\chi_\Omega(\Psi)f_{0\Omega}\left(U_0\right) f_{1\Omega}\left(U_1/\Im\Omega\right).
\end{split}
\]
\end{proof}

\begin{remark}
{\rm
The space in Theorem \ref{PropSphharm} corresponds to $V_{uv\ell}$ in \cite[Definition 4.14]{DUNKL}, but Dunkl used a different description. He introduced a basis for $L(\mathcal{G}_n^q)\equiv L(\mathcal{H}_n^q)$ (cf. Proposition \ref{PropGH}) in the following way. For every triple $(Y_0,Y_1,\Omega)$ with $Y_0\leq V_0, Y_1\leq V_1$, and $\Omega\in \Hom(V_0,V_1)$ such that $Y_0\leq \Ker \Omega$ and $\Im\Omega\leq Y_1$ (in our notation, $\Omega\in D_{(Y_0,Y_1)}$), he defined a function $(Y_0,Y_1,\Omega)^\diamondsuit\colon\mathcal{H}_n^q\rightarrow\mathbb{C}$ by setting (cf.\ \eqref{defchiOmega})
\[
\left[(Y_0,Y_1,\Omega)^\diamondsuit\right](U_0,U_1,\Psi) \coloneqq \begin{cases}
\tau\left(\Tr\left[\Omega_{U_0,U_1}\Psi\right]\right)&\text{ if }Y_0\leq U_0\leq \Ker\Omega\;\text{ and }\;Y_1\leq U_1,\\
0&\text{ otherwise}.
\end{cases}
\]
Then, on this basis, he defined suitable intertwining operator $D_0$ and $D_1$, similar to $d$ in \eqref{decqHahnbis}, with $D_0$ acting on the $V_0$-subspaces, and $D_1$ on the $V_1$-subspaces. Finally, $V_{uv \ell}$ is described as the intersections of the kernels of $D_0$ and 
$D_1$. We directly refer to Dunkl's paper for more details.}
\end{remark}

\section{Open problems}
In this final section, we collect some open problems related to a possible further generalization of our main construction.

First recall that given a short exact sequence 
\begin{equation}
\label{e:short-es}
1 \rightarrow N \overset{\iota}{\rightarrow} G \overset{\pi}{\rightarrow} H \rightarrow 1
\end{equation}
of (finite) groups, one refers to the group $G = G(\iota,\pi)$ as to the {\em extension} of $H$ by $N$ (by means of the
monomorphism $\iota$ and the epimorphism $\pi$). 
The representation theory of $G$ may be constructed by means of the representation theories of $N$ and $H$, the action of $H$ on the dual group $\widehat{N}$, and a generalization of the \emph{little group method} of Mackey and Wigner (cf.\ \cite{CST3,CST4, book3, book4,book6}) 
by means of the so-called \emph{projective representations} (see \cite[Theorem 7.23]{book6}).

Recall that a short exact sequence \eqref{e:short-es} is termed \emph{split} if there exists a \emph{section} $t \colon H \to G$
(that is, $\pi \circ t = \Id_H$) which is a group homomorphism. Equivalently, \eqref{e:short-es} is split if and only if 
$G = H \ltimes N$ is the semidirect product of $N$ by $H$.

\begin{definition}
\label{def:sub-extension}
{\rm Let
\begin{equation}
\label{e:short-es-bis}
1 \rightarrow N_j \overset{\iota_j}{\rightarrow} G_j \overset{\pi_j}{\rightarrow} H_j \rightarrow 1
\end{equation}
$j=1,2$, be two short exact sequences equipped with monomorphisms $\iota_{N_2}^{N_1} \colon N_2 \to N_1$, 
$\iota_{G_2}^{G_1} \colon G_2 \to G_1$, and $\iota_{H_2}^{H_1} \colon H_2 \to H_1$ such that the diagram

\qquad\qquad\qquad\qquad\qquad\qquad\qquad\xymatrix{
N_1 \ar[r]^{\iota_1} & G_1 \ar[r]^{\pi_1} & H_1\\
N_2 \ar[u]_{\iota_{N_2}^{N_1}}\ar[r]^{\iota_2} & G_2 \ar[u]_{\iota_{G_2}^{G_1}}\ar[r]^{\pi_2} & H_2 \ar[u]_{\iota_{H_2}^{H_1}}\\
}

\noindent
is commutative. One says that $G_2 = G(\iota_2,\pi_2)$ is a \emph{sub-extension} of $G_1 = G(\iota_1,\pi_1)$.}
\end{definition}

Let $A$ be an abelian group, let $H$ be a group acting on $A$, let $K$ a subgroup of $H$,
and let $C$ be a subgroup of $A$ which is $K$-invariant. Let $G_1 \coloneqq H \ltimes A$ and
$G_2 \coloneqq K \ltimes C$ denote the corresponding \underline{split} extensions. 
Note that, via the inclusion maps, $G_2$ is a sub-extension of $G_1$. 
In Theorem \ref{Theodec} we determined the decomposition of the associated permutation representation $L(G_1/G_2)$
into irreducible representations and in Corollary \ref{Cormultfree} we gave necessary and sufficient condition for 
$(G_1, G_2)$ being a Gelfand pair.
Moreover, if the latter is the case, in Section 6 we computed the corresponding spherical functions
(cf.\ Theorem \ref{Teospherfunct}).

\begin{quote}
\begin{openproblem}
{\rm
Let \eqref{e:short-es-bis}, $j=1,2$, be two short exact sequences such that $G_2$ is a sub-extension of $G_1$. 
Determine the decomposition of the associated permutation representation $L(G_1/G_2)$ into irreducible representations.
Determine general necessary and sufficient conditions for $(G_1, G_2)$ being a Gelfand pair and,
if this is the case, compute the corresponding spherical functions. 
In \cite[Theorem 7.23]{book6} we presented a list of all irreducible representation of $G_2$ (by means of the
representation theories of $N_2$ and $H_2$, of the action of $H_2$ on the dual group $\widehat{N_2}$, and the generalization of the 
little group method in terms of projective representations we alluded to above). Let then $\sigma\in\widehat{G_2}$ (expressed via
the above list) and determine the decomposition of its induction $\Ind_{G_2}^{G_1}\sigma$. Also establish necessary and sufficient conditions
for this decomposition being multiplicity-free (in our monograph \cite{book5}, we have developed a theory of multiplicity-free induced representations, or Gelfand triples, and studied their associated spherical functions). 
This last question is also relevant in the split case (this would constitute a generalization of Theorem \ref{Theodec}).\\
}
\end{openproblem}
\end{quote}

The automorphism group of a finite abelian $p$-group is described in \cite[Section 1.7]{book4}. 

\begin{quote}
\begin{openproblem}
{\rm
Let $A$ be a finite abelian $p$-group. 
Determine whether of not it is possible to construct an ``abelian'' analog of the $q$-Johnson scheme by considering the action of $\Aut(A)$, the automorphism group of $A$, on the subgroups of $A$.
} 
\end{openproblem}
\end{quote}

\begin{quote}
\begin{openproblem}
{\rm
Let $A$ and $B$ be two finite abelian $p$-groups. 
Consider the natural action of $\Aut(A)\times\Aut(B)$ on $\Hom(A,B)$ given by setting  $(h,g)\varphi\coloneqq h\circ\varphi\circ g^{-1}$ for all $h\in\Aut(A)$, $g\in\Aut(B)$, and $\varphi\in\Hom(A,B)$. 
Determine whether of not it is possible to construct an ``abelian'' analog of the $q$-Hamming scheme described in \cite[Section 8.7]{book} by considering the natural action of 
$\left(\Aut(A)\times\Aut(B)\right)\ltimes\Hom(A,B)$ on the pairs $(\psi,C)$ such that $C\leq A$ and $\psi\in\Hom(C,B)$ given by
\[
(h,g,\varphi)(\psi,C)\coloneqq \bigl((h\circ\varphi\circ g^{-1})\rvert_{hC},hC\bigr).
\]
} 
\end{openproblem}
\end{quote}

\begin{quote}
\begin{openproblem}
{\rm
Determine whether of not the Gelfand pair in Corollary \ref{lastCor} has an ``abelian'' analog obtained by replacing vector spaces 
(resp.\ the general linear group (and its subgroups)) by abelian $p$-groups (resp.\ by the automorphism group (and its subgroups)
of these abelian $p$-groups).
} 
\end{openproblem}
\end{quote}

\end{document}